\documentclass{article}

\usepackage{amsmath,amsfonts,amssymb,amstext,amsthm,amscd,mathrsfs}
\usepackage{epsfig}
\usepackage[all]{xy}
\usepackage{pdfsync}
\usepackage{graphicx}
\usepackage{tikz-cd}

\def\Z{{ \mathbf{Z}}}
\def\C{{\mathbf{C}}}
\def\N{{\mathbf N}}
\def\D{{\mathbf D}}
\def\B{{\mathbf B}}
\def\R{{\mathbf R}}
\def\Oo{{\mathcal O}}
\def\P{{\mathbf{P}}}
\def\A{{\mathbf A}}
\def\P{{\mathbf P}}
\def\J{{\tilde{J}}}
\def\Jj{{\mathcal J}}
\def\m{{\mathfrak{m}}}

\def\X{{\mathfrak{X}}}

\def\elem(#1,#2){  \{ {{#1}\over \overline {\ #2\ }}\}  }
\def\polar{{P_k(X;L^{d-k})}}

\newtheorem{definition}{Definition}[section]

\newtheorem{remark}{Remark}[section]

\newtheorem{ex}{Example}[section]
\newtheorem{exercise}{Exercise}[section]
\newenvironment{example}{\begin{ex}\em}{\end{ex}}
\newtheorem{theorem}{Theorem}[section]
\newtheorem{corollary}{Corollary}[section]

\newtheorem{proposition}[theorem]{Proposition}
\newtheorem{lemma}{Lemma}[section]

\begin{document}
\title{Local polar varieties in the geometric study of singularities}
\author{Arturo Giles Flores and Bernard Teissier}
\maketitle
\textit{This paper is dedicated to those administrators of research who realize how much damage is done by the evaluation of mathematical research solely by the rankings of the journals in which it is published, or more generally by bibliometric indices. We hope that their lucidity will become widespread in all countries.}
\pagestyle{myheadings}
\markboth{\rm Arturo Giles Flores and Bernard Teissier}{\rm Local polar varieties}

\begin{abstract}
\noindent  This text presents several aspects of the theory of equisingularity of complex analytic spaces from the standpoint of Whitney conditions. The goal is to describe from the geometrical, topological, and algebraic viewpoints a canonical locally finite partition of a reduced complex analytic space $X$ into nonsingular strata with the property that the local geometry of $X$ is constant on each stratum. Local polar varieties appear in the title because they play a central role in the unification of viewpoints. The geometrical viewpoint leads to the study of spaces of limit directions at a given point of $X\subset \C^n$ of hyperplanes of $\C^n$ tangent to $X$ at nonsingular points, which in turn leads to the realization that the Whitney conditions, which are used to define the stratification, are in fact of a Lagrangian nature. The local polar varieties are used to analyze the structure of the set of limit directions of tangent hyperplanes. This structure helps in particular to understand how a singularity differs from its tangent cone, assumed to be reduced. The multiplicities of local polar varieties are related to local topological invariants, local vanishing Euler-Poincar\'e characteristics, by a formula which turns out to contain, in the special case where the singularity is the vertex of the cone over a reduced projective variety, a Pl\"ucker-type formula for the degree of the dual of a projective variety. \par\medskip\centerline{\textbf{R\'esum\'e}}\par\medskip\noindent Ce texte pr\'esente plusieurs aspects de la th\'eorie de l'\'equisingularit\'e des espaces analytiques complexes telle qu'elle est d\'efinie par les conditions de Whitney. Le but est de d\'ecrire des points de vue g\'eom\'etrique, topologique et alg\'ebrique une partition canonique localement finie d'un espace analytique complexe r\'eduit $X$ en strates non singuli\`eres telles que la g\'eom\'etrie locale de $X$ soit constante le long de chaque strate. Les vari\'et\'es polaires locales apparaissent dans le titre parce qu'elles jouent un r\^ole central dans l'unification des points de vue. Le point de vue g\'eom\'etrique conduit \`a l'\'etude des directions limites en un point donn\'e de $X\subset \C^n$ des hyperplans de $\C^n$ tangents \`a $X$ en des points non singuliers. Ceci am\`ene \`a r\'ealiser que les conditions de Whitney, qui servent \`a d\'efinir la stratification, sont en fait de nature lagrangienne. Les vari\'et\'es polaires locales sont utilis\'ees pour analyser la structure de l'ensemble des positions limites d'hyperplans tangents. Cette structure aide \`a comprendre comment une singularit\'e diff\`ere de son c\^one tangent, suppos\'e r\'eduit. Les multiplicit\'es des vari\'et\'es polaires locales sont reli\'ees \`a des invariants topologiques locaux, des caract\'eristiques d'Euler-Poincar\'e \'evanescentes, par une formule qui se r\'ev\`ele, dans cas particulier o\`u la singularit\'e est le sommet du c\^one sur une vari\'et\'e projective r\'eduite, donner une formule du type Pl\"ucker pour le calcul du degr\'e de la vari\'et\'e duale d'une vari\'et\'e projective.
\end{abstract}

\tableofcontents
\vskip.3truecm
\section{Introduction}\label{basic}
The origin of these notes is a course imparted by the second author in the
``$2^{{\rm ndo}}$ Congreso Latinoamericano de Matem\'aticos'' celebrated in Cancun, Mexico on
July 20 -26, 2004. The first redaction was subsequently elaborated by the authors.\par
The theme of the course was the local study of analytic subsets of $\C^n$, which is the local study of reduced complex analytic spaces. 
That is, we will consider subsets defined in a neighbourhood of a point $0\in \C^n$ 
by equations:
\[f_1(z_1,\ldots,z_n)=\cdots=f_k(z_1,\ldots,z_n)=0\]
\[f_i \in \C\{z_1,\ldots,z_n\}, \; f_i(0)=0, \, i=1,\ldots,k.\]
Meaning that the subset $X\subset \C^n$ is thus defined in a neighbourhood $U$ of $0$,
where all the series $f_i$ converge. Throughout this text, the word ``local" means that we work with ``sufficiently small" representatives of a germ $(X,x)$. \textit{For simplicity we assume throughout that the spaces we study are equidimensional; all their irreducible components have the same dimension}. The reader who needs the general case of reduced spaces should have no substantial difficulty in making the extension.\par
The purpose of the course was to show how to \textit{stratify} $X$. In other words, partition $X$ into
finitely many \textbf{nonsingular}\footnote{We would prefer to use \textit{regular} to emphasize the absence of singularities, but this term has too many meanings. } complex analytic manifolds $\{X_{\alpha}\}_{\alpha \in A}$, which will 
be called \textbf{strata}, such that:
\begin{itemize}
\item[i)] The closure $\overline{X_{\alpha}}$ is a closed complex analytic subspace of $X$, for all $\alpha \in A$.
\item[ii)] The frontier $\overline{X_{\beta}} \setminus X_{\beta}$ is a union of strata $X_{\alpha}$, for all $\beta \in A$.
\item[iii)] Given any $x \in X_{\alpha}$, the ``geometry'' of all the closures $\overline{X_{\beta}}$
            containing $X_{\alpha}$ is locally constant along $X_{\alpha}$ in a neighbourhood of x. 
\end{itemize}
\noindent The stratification process being such that at each step one can take connected components, one usually assumes that the strata are connected and then the closures $\overline{X_{\alpha}}$ are equidimensional. 
\par Recall that two closed subspaces $X\subset U\subset\C^n$ and $X'\subset U'\subset\C^n$, where $U$ and $U'$ are open, have the same \textit{embedded topological type} if there exists a homeomorphism $\phi\colon U\to U'$ such that $\phi (X)=X'$. If $X$ and $X'$ are representatives of germs $(X,x)$ and $(X',x')$ we require that $\phi(x)=x'$ and we say the two germs have the same embedded topological type.\par\noindent If by ``geometry" we mean the embedded local topological type at $x\in X_\alpha$ of $\overline{X_{\beta}}\subset \C^n$ \textit{and} of its sections by affine subspaces of $\C^n$ of general directions passing near $x$ or through $x$, which we call the \textit{total local topological type}, there is a \textbf{minimal} such partition, in the sense that any other partition with the same properties must be a sub-partition of it. Characterized by differential-geometric conditions, called Whitney conditions, bearing on limits of tangent spaces and of secants, it plays an important role in generalizing to singular spaces geometric concepts such as Chern classes and integrals of curvature. The existence of such partitions, or stratifications, without proof of the existence of a minimal one, and indeed the very concept of stratification\footnote{Which was subsequently developed, in particular by Thom in \cite{Th} and Mather in \cite{Ma1}.}, are originally due to Whitney in \cite{Whi1}, \cite{Whi2}.  In these papers Whitney also initated the study in complex analytic geometry of limits of tangent spaces at nonsingular points. In algebraic geometry the first general aproach to such limits is due to Semple in \cite{Se}.\par In addition to topological and differential-geometric characterizations, the partition can also be described algebraically by means of \textbf{Polar Varieties}, and this is one of the main points of these lectures.\par
Apart from the characterization of Whitney conditions by equimultiplicity of polar varieties, one of the main results appearing in these lectures is therefore the equivalence of Whitney conditions for a stratification $X=\bigcup_\alpha X_\alpha$ of a  complex analytic space $X\subset \C^n$ with the local topological triviality of the closures $\overline{X_\beta}$ of strata along each $X_\alpha$ which they contain, to which one must add the local topological triviality along $X_\alpha$ of the intersections of the $\overline{X_\beta}$ with (germs of) general nonsingular subspaces of $\C^n$ containing $X_\alpha$.\par
Other facts concerning Whitney conditions also appear in these notes, for example that the Whitney conditions are in fact of a Lagrangian  nature, related to a condition of relative projective duality between the irreducible components of the normal cones of the $\overline{X_\beta}$ along the $X_\alpha$ and of some of their subcones, on the one hand, and the irreducible components of the space of limits of tangents hyperplanes to $\overline{X_\beta}$ at nonsingular points approaching $X_\alpha$, on the other. This duality, applied to the case where $X_\alpha$ is a point,  gives a measure of the geometric difference between a germ of singular space at a point and its tangent cone at that point, assumed to be reduced. Among the important facts concerning polar varieties of a germ $(X,x)$ is that their multiplicity at a point depends only on the \textit{total local topological type} of the germ. This is expressed by a formula (see theorem \ref{formula}) relating the multiplicities of polar varieties to local vanishing Euler-Poincar\'e characteristics. \par Applying this formula to the cone over a projective variety $V$ gives an expression for the degree of the projective dual variety $\check V$ which depends only on the local topological characters of the incidences between the strata of the minimal Whitney stratification of $V$ and the Euler characteristics of these strata and their general linear sections.  In particular we recover with topological arguments the formula for the class (another name for the degree of the dual variety) of a projective hypersurface with isolated singularities.\par The original idea of the course was to be as geometric as possible. Since many proofs in this story are quite algebraic, using in particular the notion of integral dependence on ideals and modules (see \cite{Te3}, \cite{Ga1}), they are often replaced by references. Note also that in this text, intersections with linear subspaces of the nonsingular ambient space are taken as reduced intersections.\par
\par\medskip
 We shall begin by trying to put into historical context the appearance of polar varieties, as a means to give the reader a little insight and intuition into what we will be doing. A part of what follows is taken from \cite{Te5}; see also \cite{Pi2}.

\indent It is possible that the first example of a polar variety is in the treatise on 
conics of Apollonius. The cone drawn
from a point $0$ in affine three-space outside of a fixed sphere 
around that sphere meets the sphere along a circle $C$. If we
consider a plane not containing the point, and the projection $\pi$ 
from $0$ of the affine three-space onto that plane, the circle $C$ is
the set of critical points of the restriction of $\pi$ to the sphere. 
Fixing a plane $H$, by moving the point $0$, we can obtain any circle drawn on the
sphere, the great circles beeing obtained by sending the point to 
infinity in the direction perpendicular to the plane of the 
circle. 
\par \noindent 
Somewhat later, around 1680, John Wallis
asked how many tangents can be drawn to a nonsingular curve of degree
$d$ in the plane from a point in that plane and conjectured that this 
number should always be $\leq d^2$. In modern terms, he was proposing 
to compare the visual
complexity of a curve (or surface) as measured by the number of 
"critical"  lines of sight with its algebraic 
complexity as measured by the degree.
Given an algebraic surface
$S$ of degree $d$ in affine three-space and a point $0$ outside it, 
the lines emanating from
$0$ and tangent to $S$ touch $S$ along an algebraic curve $P$. Taking 
a general hyperplane $H$ through $0$, we see that the number of 
tangents drawn
from $0$ to the curve $S\cap H$ is the number 
of points in $P\cap H$ and therefore is bounded by the degree
of the algebraic curve $P$ drawn on $S$. This algebraic curve is an example of a 
\textit{polar curve}; it is the generalization of Apollonius' circles. Wallis' question was first answered by Goudin and Du S\'ejour who showed in 1756  that this number is $\leq d(d-1)$ (see \cite{Kl3}) and later by Poncelet, who 
saw (without writing a single
equation; see \cite[p. 68, p. 361 and ff.]{Pon}) that the natural setting for the problem which had been 
stated in the real affine plane by Wallis was the complex projective 
plane and that the number of
tangents drawn from a point with projective coordinates
$(\xi:\eta:\zeta)$ to the curve $C$ with homogeneous equation 
$F(x,y,z)=0$ is equal to the number of intersection points of $C$ with a curve of degree $d-1$. 
The equation of this curve was written explicitely later by Pl\"ucker (see \cite{W}, \cite{Pl}):
$$\xi\frac{\partial F}{\partial x}+\eta\frac{\partial F}{\partial 
y}+\zeta\frac{\partial F}{\partial z}=0.\leqno{(1)}$$
This equation is obtained by \textit{polarizing} the polynomial 
$F(x,y,z)$ with respect to $(\xi:\eta:\zeta)$, a terminology which 
comes from the study of conics; it is the
method for obtaining a bilinear form from a quadratic form (in 
characteristic $\neq 2$).\par\noindent It is the polar curve of 
$C$ with respect to the point $(\xi:\eta:\zeta)$, or rather, in the terminology of \cite{Te3}, the projective curve associated to the  \textit{relative polar surface} of the map $(\C^3,0)\to (\C^2,0)$ given by $(F,\xi x+\eta y+\zeta z)$. The term emphasizes that it is attached to a morphism, unlike the polar varieties \`a la Todd, which in this case would be the points on the curve $C$ where the tangent line contains the point $(\xi:\eta:\zeta)$.  In any case, it is of degree $d-1$, where $d$ is the degree of the polynomial $F$ and by B\'ezout's theorem, except in the case 
where the curve $C$ is not reduced,
i.e., has multiple components, the number of intersection points 
counted with multiplicities is exactly $d(d-1)$. So we
conclude with Poncelet that the number (counted with multiplicities) of points 
of the nonsingular curve $C$ where the tangent goes through the 
point with coordinates
$(\xi:\eta:\zeta)$ is equal to $d(d-1)$. The equations written by Pl\"ucker shows that, as the point varies in the 
projective plane, the equations $(1)$ describe a \textit{linear 
system} of curves of degree $d-1$ in the
plane, which cuts out a linear system of points on the curve
$C$, of degree $d(d-1)$. It is the most natural linear system after that which comes from the lines (hyperplanes)
$$\lambda x+\mu y +\nu z=0,\ \ (\lambda : \mu :\nu)\in \P^2$$
\noindent
and comes from the linear system of points in the dual space $\check \P^2$ while the linear system $(1)$ can be seen as coming from the linear system of lines in $\check\P^2$. The projective duality between $\P^2$ and the space $\check\P^2$ of lines in $\P^2$ exchanges the two linear systems. The \textit{dual curve} $\Check C\in \check \P^2$ is the closure of the set of points in $\check\P^2$ corresponding to lines in $\P^2$ which are tangent to $C$ at a nonsingular point. Its degree is called the \textit{class} of the curve $C$. It is the number of intersection points of $\check C$ with a general line of $\check\P^2$, and that, by construction, is the number of tangents to $C$ passing through a given general point of $\P^2$.    \par In the
theory of algebraic curves, there is an important formula called the Riemann-Hurwitz formula. Given an algebraic map $f\colon C\to
C'$ between compact nonsingular complex algebraic curves, which is of degree $\hbox{deg}f=d$ (meaning that for a general
point $c'\in C'$, $f^{-1}(c')$ consists of $d$ points, and is ramified at the points $x_i\in C, \ 1\leq
i\leq r$, which means that near $x_i$, in suitable local coordinates on $C$ and $C'$, the
map $f$ is of the form $t\mapsto t^{e_i+1}$ with $e_i\in \N,\ \  e_i\geq 1$ .
The integer $e_i$ is the \textit{ramification index} of $f$ at $x_i$. Then we
have the Riemann-Hurwitz formula relating the genus of $C$ and the genus of
$C'$ via $f$ and the ramification indices: $$2g(C)-2 = d(2g(C')-2)+\sum_ie_i.$$ If we apply this formula to the case $C'=\P^1$, knowing that any compact
algebraic curve is a finite ramified covering of $\P^1$, we find that we can
calculate the genus of $C$ from any linear system of points made of the fibers of a map $C\to
\P^1$ if we know its degree and its singularities: we get $$2g(C)=2-2d+\sum e_i.$$ The
ramification points $x_i$ can be computed as the so-called \it Jacobian divisor \rm of the linear
system, which consists of the singular points, properly counted, of the singular
members of the linear system. In particular if $C$ is a plane curve and the
linear system is the system of its plane sections by lines through a general point $x=(\xi
:\eta : \zeta)$ of $\P^2$, the map $f$ is the projection from $C$ to $\P^1$ from
 $x$; its degree is the degree $m$ of $C$ and its ramification points are exactly
the points where the line from $x$ is tangent to $C$. Since $x$ is general, these
are simple tangency points, so the $e_i$ are equal to $1$, and their number is
equal to the class $\check m$ of $C$; the formula gives $$2g(C)-2=-2m+\check
m\ ,$$ thus giving for the genus an expression linear in the degree and the class.\par This is the first example
of the relation between  the ``characteristic classes" (in this case only the
genus) and the polar classes; in this case the curve itself, of degree $m$ and
the degree of the polar locus, or apparent contour from $x$, in this case the class $\check m$.
After deep work by Zeuthen, C. Segre, Severi, it was Todd who in three fundamental papers (\cite{To1}, \cite{To2}, \cite{To3}) found the correct generalization of the formulas known for curves and surfaces.\par\noindent
More precisely, given a nonsingular $d$-dimensional variety $V$ in the complex projective space $\P^{n-1}$, for a linear
subspace $D\subset \P^{n-1}$ of dimension $n-d+k-3$, i.e., of codimension $d-k+2$, with $0\leq k\leq d$, let us set
$$P_k(V;D)=\{v\in V /{\rm dim}(T_{V,v}\cap D)\geq k-1\}.$$
This is the \textit{Polar variety} of $V$ associated to $D$; if $D$ is general, it is either empty
or purely of codimension $k$ in $V$. If $n=3$, $d=1$ and $k=1$, we find the points of the projective plane curve $V$ where the tangent lines go through the point $D\in\P^2$. A tangent hyperplane to $V$ at a point $v$ is a hyperplane containing the tangent space $T_{V,v}$. The polar variety $P_k(V,D)$ with respect to a general $D$ of codimension $d-k+2$ consists of the points of $V$ where a tangent hyperplane contains $D$, a condition which is equivalent to the dimension inequality. We see that this construction is a direct generalization of
the apparent contour. The eye $0$ is replaced by the linear subspace $D$!\par Todd shows
that (rational equivalence classes of) \footnote{Which he invents for the occasion.} the following formal linear combinations of varieties of codimension $k$, for $0\leq k\leq d$:
$$V_k=\sum_{j=0}^{j=k}(-1)^j{d-k+j+1\choose j}P_{k-j}(V;D_{d-k+j+2})\cap H_j,$$ where $H_j$ is a 
linear subspace of codimension $j$ and $D_{d-k+j+2}$ is of codimension $d-k+j+2$, are independent of all the choices made and of the
embedding of $V$ in a projective space, provided that the $D's$ and the $H$'s have been
chosen general enough. Our $V_k$ are in fact Todd's $V_{d-k}$. The intersection numbers arising from intersecting these classes and hyperplanes in the right way to obtain numbers contain a wealth of numerical invariants, such as Euler characteristic and genus. Even the arithmetic genus, which is the generalization of the differential forms definition of the genus of a curve, can be computed. Around 1950 it was realized that the classes of Todd, which had also been considered independently by Eger, are nothing but the Chern classes of the tangent bundle of $V$. \par On the other hand, the basic topological invariant of the variety $V$, its Euler-Poincar\'e
characteristic (also called Euler characteristic for short) satisfies the equality: $$\chi (V)={\rm deg}V_d=\sum_{j=0}^d(-1)^j(j+1)(P_{d-j}(V). H_j),\eqno{(E)}$$ where $P_{d-j}(V)$ is the polar variety of codimension $d-j$ with respect to a general $D$ of codimension $j+2$, which is omitted from the notation, and $(a. b)$
denotes the intersection number in $\P^{n-1}$. In this case, since we intersect with a linear space of
complementary dimension, $(P_{d-j}(V). H_j)$ is just the degree of the projective variety $P_{d-j}(V)$. \par So Todd's
results give a rather complete generalization of the genus formula for curves, both in its analytic and its
topological aspects. This circle of ideas was considerably extended, in a cohomological framework, to generalized notions of genus and characteristic classes for nonsingular varieties; see \cite{H} and \cite[Chapters 48,49]{Po}. Todd's construction was modernized and extended to the case of a singular projective variety by R. Piene (see \cite{Pi1}).\par
What we use here is a local form, introduced in \cite{L-T1}, of the polar varieties of Todd, adapted to the singular case and defined for any equidimensional and reduced germ of a complex analytic space. The case of a singular projective variety which we have just seen can be deemed to be the special case where our singularity is a cone.\par We do not take classes in (Borel-Moore) homology or elsewhere because the loss of geometric information is too great, but instead look at ``sufficiently general" polar varieties of a given dimension as geometric objects. The hope is that the equisingularity class (up to a Whitney equisingular deformation) of each general polar variety of a germ is an analytic invariant.\footnote{Indeed, the statement at the end of remark 3.2 in \cite{Te3} should be entitled "problem" and not "theorem".} What is known is that the multiplicity is, and this is what we use below.\par\medskip
\textit{Since the stratification we build is determined by local conditions and is canonical, the stratifications defined in the open subsets of a covering of a complex analytic space $X$ will automatically glue up. Therefore it suffices to study the stratifications locally assuming $X\subset\C^n$, as we do here. We emphasize that the result of the construction for $X$, unlike its tools, is independent of the embedding $X\subset\C^n$.}

\section{Limits of tangent spaces, the conormal space and the tangent cone.}\label{symp}

    To set the working grounds, let us fix a reduced and pure-dimensional germ
of analytic subspace $(X,0)\subset (\C ^n,0)$. That is, we are assuming that $X$ is given to us by an ideal $I$
of $\C \{z_1,\ldots,z_n\}$ generated say by $(f_1,\ldots,f_k)$, containing all analytic functions 
vanishing on $X$, and also that all the irreducible components of $X$, corresponding to the minimal prime ideals of $\C \{z_1,\ldots,z_n\}$ which contain $I$, have the same dimension $d$.\\

By definition, a singular point of a  complex analytic space 
is a point where the tangent space cannot be defined as usual. As a 
substitute, we can look at all limit positions of tangent spaces at nonsingular points tending to 
a given singular point. 

\begin{definition}
Given a closed $d$-dimensional analytic subset $X$ in an open set of $\C^n$, a d-plane $T$ of the Grassmannian $G(d,n)$ of $d$-dimensional vector subspaces of $\C^n$  is a limit at $x \in X$ of tangent spaces to the analytic space
$X$ if there exists a sequence $\{x_i\}$ of nonsingular points of $X$ and a sequence of d-planes $\{T_i\}$ of
$G(d,n)$ such that for all i, the d-plane $T_i$ is the direction of the tangent space to $X$ at $x_i$, the sequence 
$\{x_i\}$ converges to $x$ and the sequence $\{T_i\}$ converges to $T$.
\end{definition}

How can we determine these limit positions? Recall that if $X$ is an analytic space then $\mathrm{Sing} X$, the set of
singular points of X, is also an analytic space and the nonsingular part $X^0=X\setminus \mathrm{Sing}X$ is
dense in $X$ and has the structure of a complex manifold. \\
\par\noindent
 Let $X$ be a representative of $(X,0)$. Consider the application (the Gauss map)
\begin{align*} \gamma_{X^0}: X^0 & \longrightarrow G(d,n)\\
                              x  & \longmapsto T_{X^0,x},
\end{align*} 
where $T_{X^0,x}$ denotes
the direction in $G(d,n)$ of the tangent space to the manifold $X^0$ at the point $x$. Let $NX$ be the closure
of the graph of $\gamma_{X^0}$ in $X\times G(d,n)$. It can be proved that $NX$ is an analytic 
subspace of dimension d (\cite[theorem 16.4]{Whi1}, ). 

\begin{definition}
The morphism $\nu_{X}: NX \longrightarrow X$ induced by the first projection of $X \times G(d,n)$
is called the \textit{Semple-Nash modification} of $X$. \[ \xymatrix { &\ \ \ \ \ \ \ \ \ \ \ \ \ \ \ \ \ \ \ \ \ \ NX\ \ \subset X\times G(d,n)\ar[dl]_{\nu_{X}}\ar[dr]^{\gamma_{X}} &  \\
                 X  & & G(d,n)} \]
\end{definition} 
It is an isomorphism over the nonsingular part of $X$ and is proper since the Grassmannian is compact and the projection $X\times G(d,n)\to X$ is proper. It is therefore a proper birational map. It seems to have been first introduced by Semple (see the end of \cite{Se}) who also asked whether iterating this construction would eventually resolve the singularities of $X$, and later rediscovered by Whitney (see  (\cite{Whi1}) and also by Nash, who asked the same question as Semple. It is still without answer except for curves. \par\noindent
The notation $NX$ is justified by the fact that the Semple-Nash modification is independent, up to a unique $X$-isomorphism, of the embedding of $X$ into a nonsingular space. See \cite[\S 2]{Te2}, where the abstract construction of the Semple-Nash modification is explained in terms of the Grothendieck Grassmannian associated to the module of differentials of $X$.
 The fiber 
$\nu_X^{-1}(0)$ is a closed algebraic
subvariety of $G(d,n)$; set-theoretically, it is the 
set of limit positions of tangent spaces at points of $X^0$ 
tending to $0$.\par\noindent
For an exposition of basic results on limits of tangent spaces in the case of germs of complex analytic surfaces,  good references are \cite{ACL} and \cite{Sn1}; the latter makes connections with the resolution of singularities. For a more computational approach, see \cite{O1}. \par\noindent In \cite{Hn}, H. Hennings has announced a proof of the fact that if $x$ is an isolated singular point of $X$, the dimension of $\nu_X^{-1}(x)$ is ${\rm dim}X-1$, generalizing a result of A. Simis, K. Smith and B. Ulrich in \cite{SSU}.\\\par
The Semple-Nash modification is somewhat difficult to handle from the viewpoint of intersection theory
because of the complexity due to the rich geometry of the Grassmannian. There is a less intrinsic but more amenable way of encoding the 
limits of tangent spaces. The idea is to replace a tangent space to $X^0$ by the 
collection of all the hyperplanes of $\C^n$ which contain it. Tangent hyperplanes live in a 
projective space, namely the dual projective space $\check{\P}^{n-1}$,
which is easier to deal with than the Grassmannian.

\subsection{Some symplectic Geometry}
In order to describe this set of tangent hyperplanes, we are going to use the 
language of symplectic geometry and Lagrangian submanifolds. Let us start
with a few definitions.\\

Let $M$ be any $n$-dimensional manifold, and let $\omega$ be a de Rham 2-form on M,
that is, for each $p \in M$, the map
\[\omega_p:T_{M,p} \times T_{M,p} \rightarrow \R \]
is skew-symmetric bilinear on the tangent space to $M$ at $p$, and $\omega_p$ varies 
smoothly with $p$. We say that $\omega$ is \textbf{symplectic} if it is closed and $\omega_p$ 
is non-degenerate for all $p \in M$. Non degeneracy means that the map which to $v\in T_{M,p}$ associates the homomorphism $w\mapsto \omega (v,w)\in \R$ is an isomorphism from $T_{M,p}$ to its dual. A \textbf{symplectic manifold} is a pair
$(M,\omega)$, where $M$ is a manifold and $\omega$ is a symplectic form. These definitions extend, replacing $\R$ by $\C$, to the case of a complex analytic manifold i.e., nonsingular space.\\ \par

For any manifold $M$, its cotangent bundle $T^*M$ has a canonical symplectic
structure as follows. Let
\begin{align*}
      \pi:T^*M & \longrightarrow M \\
       p=(x,\xi) & \longmapsto   x ,   
\end{align*} 
%\[\xymatrix {T^*M\ar[d]^\pi & & p=(x,\xi)\ar[d] & & \xi \in T_x^*M \\ 
%                M   & &             x } \]
where $\xi \in T_{M,p}^*$, be the natural projection. The \textbf{Liouville 1-form} $\alpha$ on $T^*M$ may be defined pointwise
by:
\[\alpha_{p}(v)=\xi\left( (d\pi_{p})v\right), \; \; {\rm for} \; v \in T_{T^*M,p}.\]
Note that $d\pi_{p}$ maps $T_{T^*M,p}$ to $T_{M,x}$, so that $\alpha$ is well defined. 
The \textbf{canonical symplectic 2-form} $\omega$ on $T^*M$ is defined as 
\[\omega = - d \alpha.\]
And it is not hard to see that if $(U,x_1,\ldots,x_n)$ is a coordinate chart for $M$
with associated cotangent coordinates $(T^*U,x_1,\ldots,x_n,\xi_1,\ldots,\xi_n)$, then
locally:
\[\omega=\sum_{i=1}^n dx_i\wedge d\xi_i. \]

\begin{definition}
Let $(M,\omega)$ be a 2n-dimensional symplectic manifold. A submanifold $Y$ of $M$ is a
\textbf{Lagrangian submanifold} if at each $p \in Y$, $T_{Y,p}$ is a Lagrangian subspace
of $T_{M,p}$, i.e., $\omega_p |_{T_{Y,p}}\equiv 0$ and ${\rm dim} T_{Y,p}= \frac{1}{2} {\rm dim} T_{M,p}$. Equivalently,
if $i:Y\hookrightarrow M$ is the inclusion map, then $Y$ is \textbf{Lagrangian} if and only 
if $i^*\omega=0$ and ${\rm dim}Y=\frac{1}{2}{\rm dim}M$.\\  Let $M$ be a nonsingular complex analytic space of even dimension equipped with a closed non degenerate $2$-form $\omega$. If $Y\subset M$ is a complex analytic subspace, which may have singularities, we say that it is a \textbf{Lagrangian subspace} of $M$ if it is purely of dimension $\frac{1}{2}{\rm dim}M$ and there is a dense nonsingular open subset of the corresponding reduced subspace which is a Lagrangian submanifold in the sense that $\omega$ vanishes on the tangent space.
\end{definition}

\begin{ex}
The \textbf{zero section} of $T^*M$
\[X:=\{ (x,\xi) \in T^*M | \xi=0 \; in \; T_{M,x}^*\}\]
is an $n$-dimensional Lagrangian submanifold of $T^*M$.
\end{ex}
\begin{exercise} Let $f(z_1,\ldots ,z_n)$ be a holomorphic function on an open set $U\subset \C^n$. Consider the differential $df$ as a section $df\colon U\to T^*U$ of the cotangent bundle. Verify that the image of this section is a Lagrangian submanifold of $T^*U$. Explain what it means. What is the image in $U$ by the natural projection $T^*U\to U$ of the intersection of this image with the zero section?
\end{exercise}

\subsection{Conormal space.}

Let now $M$ be a complex analytic manifold and $X\subset M$ be a possibly singular complex subspace of pure dimension d, and let as before
$X^0 =X \setminus \mathrm{Sing}X$ be the nonsingular part of X, which is a submanifold of $M$.

\begin{definition}
%The \textbf{conormal space} at $x \in X^0$ is
Set
\[N^*_{X^0,x} = \{ \xi \in T_{M,x}^* | \xi (v)=0, \; \forall v \in T_{X^0,x}\};\]
this means that the hyperplane $\{\xi=0\} $ contains the tangent space to $X^0$ at x.\\
The \textbf{conormal bundle} of $X^0$ is
\[T^*_{X^0}M =\{ (x,\xi) \in T^*M | x \in X^0, \; \xi \in N^*_{X^0,x}\}. \]
\end{definition}

\begin{proposition}\label{si}
Let $i:T^*_{X^0}M\hookrightarrow T^*M$ be the inclusion and let $\alpha$ be the Liouville 1-form in $T^*M$
as before. Then $i^*\alpha =0$. In particular the conormal bundle $T^*_{X^0}M$ is a conic Lagrangian submanifold of $T^*M$,
and has dimension n.
\end{proposition}
\begin{proof}$\;$\linebreak
See \cite[Proposition 3.6]{CS}.
\end{proof}
\noindent In the same context we can define the \textbf{conormal space of X in M} as 
the closure $T_X^*M$ of $T^*_{X^0}M$ in $T^*M$, with the \textbf{conormal map} $\kappa_X:T^*_XM \rightarrow X$,
induced by the natural projection $\pi:T^*M\to M$. The conormal space is of dimension $n$. It may be singular and by Proposition \ref{si}, $\alpha$ vanishes on every tangent vector at a nonsingular point,
so it is by construction a Lagrangian subspace of $T^*M$. \newline

The fiber $\kappa_X^{-1}(x)$ of the conormal map $\kappa_X\colon 
T^*_XM \to X$ above a point
$x\in X$ consists, if $x\in X^0$, of the vector space $\C^{n-d}$ of all the equations of 
hyperplanes tangent to $X$ at $x$, in the sense that they contain the 
tangent space $T_{X^0,x}$. If $x$ is a singular
point, the fiber consists of all equations of limits of hyperplanes tangent at nonsingular points of $X$ tending to 
$x$. \par\noindent
The fibers of $\kappa_X$ are invariant under multiplication by an 
element of $\C^*$, and we can divide by the equivalence relation this 
defines. The idea is to consider only up to homothety the defining forms of tangent 
hyperplanes, and not a specific linear form defining it. That is, the conormal space is
stable by vertical homotheties (a property also called \textbf{conical}), so we can ``projectivize'' it.
Moreover, we can characterize those subvarieties of the cotangent space which are
the conormal spaces of their images in $M$.

\begin{proposition}\label{conormalequiv}{\rm (see \cite[Chap.II, \S 10]{P})} Let $M$ be a nonsingular 
analytic variety of dimension $n$ and let $L$ be a closed conical 
irreducible analytic subvariety of $T^*M$. The following conditions 
are equivalent:\par\noindent
1) The variety $L$ is the conormal space 
of its image in $M$. \par\noindent
2) The Liouville 1-form $\alpha$ 
vanishes on all tangent vectors to $L$ at every nonsingular point of 
$L$.\par\noindent
3) The symplectic 2-form $\omega=-\hbox{\rm 
d}\alpha$ vanishes on every pair of tangent vectors to $L$ at every 
nonsingular point of $L$.
\end{proposition}

Since conormal varieties are conical, we may as well projectivize with respect
to vertical homotheties of $T^*M$ and work in $\P T^*M$, where it still makes sense
to be Lagrangian since $\alpha$ is homogeneous by definition\footnote{In symplectic geometry it is called \textbf{Legendrian} with respect to the natural contact structure on $\P T^*M$.}.

Going back to our original problem we have $X\subset M=\C^n$, so $T^*M=\C^n\times \check{\C}^n$
and $\P T^*M=\C^n \times \check{\P}^{n-1}$. So we have the \textbf{(projective) conormal space}
$\kappa_X:C(X) \to X$ with $C(X)\subset X \times \check{\P}^{n-1}$, where $C(X)$ denotes the projectivization of the
conormal space $T^*_XM$. Note that we have not 
changed the name of the map $\kappa_X$ after projectivizing since there is no 
ambiguity, and that the dimension of $C(X)$ is $n-1$, which shows immediately that it depends on 
the embedding of $X$ in an affine space.\par\noindent
When there is no ambiguity we shall often omit the subscript in $\kappa_X$. We have the following result:

\begin{proposition}\label{dimC}
The (projective) conormal space $C(X)$ is a closed, reduced, complex analytic subspace of
$X \times \check{\P}^{n-1}$ of dimension $n-1$. For any $x\in X$ the dimension of the fiber $\kappa_X^{-1}(x)$ is at most $n-2$.
\end{proposition}
\begin{proof}$\;$\newline
These are classical facts. See \cite[Chap. III]{CS} or \cite{Te3}, proposition 4.1, p. 379.
\end{proof}

\subsection{Conormal spaces and projective duality}\label{du}

Let us assume for a moment that $V\subset \P^{n-1}$ is a projective algebraic variety. 
In the spirit of last section, let us take $M=\P^{n-1}$ with homogeneous coordinates $(x_0:\ldots:x_{n-1})$, and 
consider the dual projective space $\check{\P}^{n-1}$ with coordinates $(\xi_0:\ldots:\xi_{n-1})$; its 
points are the hyperplanes of $\P^{n-1}$ with equations $\sum_{i=0}^{n-1}\xi_ix_i=0$. 

\begin{definition}\label{incidence}
  Define the \textbf{incidence variety} $I\subset \P^{n-1} \times \check\P^{n-1}$ as the set of 
points satisfying:
\[ \sum_{i=0}^{n-1} x_i\xi_i=0,\]
where $(x_0:\ldots:x_{n-1};\xi_0:\ldots:\xi_{n-1}) \in \P^{n-1} \times \check\P^{n-1}$
\end{definition}

\begin{lemma}{\rm (Kleiman; see \cite[\S 4]{Kl2})}\label{inclema}
    The projectivized cotangent bundle of $\P^{n-1}$ is naturally isomorphic to $I$.   
\end{lemma}\begin{proof}$\;$\newline
 Let us first take a look at the cotangent bundle of $\P^{n-1}$:
 \[\pi:T^*\P^{n-1} \longrightarrow \P^{n-1}.\]
 Remember that the fiber $\pi^{-1}(x)$ over a point $x$ in $\P^{n-1}$ is
 by definition isomorphic to $\check{\C}^{n-1}$, the vector space
 of linear forms on $\C^{n-1}$. Recall that projectivizing the
 cotangent bundle means projectivizing the fibers, and so we get a
 map:
\[\Pi:\P T^*\P^{n-1} \longrightarrow \P^{n-1}\]
 where the fiber is isomorphic to $\check{\P}^{n-2}$. So we can
 see a point of $\P T^*\P^{n-1}$ as a pair $(x,\xi)\in \P^{n-1}\times\check\P^{n-2}$. 
 On the other hand, if we fix a point $x \in \P^{n-1}$, the equation defining
 the incidence variety $I$ tells us that the set of points $(x ,\xi) \in
 I$ is the set of hyperplanes of $\P^{n-1}$ that go through the point
 $x $, which we know is isomorphic to $\check{\P}^{n-2}$.

 Now to explicitly define the map, take a chart $\C^{n-1} \times \left\{ \check{\C}^{n-1} \setminus \{0\} \right\} $ of the
 \linebreak manifold $T^*\P^{n-1} \setminus \{\mathrm{zero \; section}\}$, where the $\C^{n-1}$ corresponds to a usual
 chart of $\P^{n-1}$ and $\check{\C}^{n-1}$ to its associated cotangent chart.
 Define the map:
 \begin{align*}
    \phi_i: \C^{n-1} \times \left\{ \check{\C}^{n-1} \setminus \{0\} \right\} &\longrightarrow \P^{n-2} \times \check{\P}^{n-2} \\
            (z_1,\ldots,z_{n-1};\xi_1,\ldots,\xi_{n-1}) &\longmapsto \left( \varphi_i(z),
     (\xi_1:\cdots:\xi_{i-1}:-\sum_{j=1}^{n-1*_i} z_j\xi_j: \xi_{i+1}:\cdots: \xi_{n-1})\right)  
 \end{align*}
  where $\varphi_i(z)=(z_1:\cdots:z_{i-1}:1:z_{i+1}:\cdots:z_{n-1})$ and the star means that the index $i$ is excluded from the sum.
  
  An easy calculation shows that $\phi_i$ is injective, has its image in the incidence variety $I$ and is well defined
  on the projectivization $\C^{n-1} \times \check\P^{n-2}$. It is also clear, that varying $i$ from $1$ to $n-1$
  we can reach any point in $I$. Thus, all we need to check now is that the $\phi_j$'s paste
  together to define a map. For this, the important thing is to remember that if $\varphi_i$ and $\varphi_j$
  are charts of a manifold, and $h:= \varphi_j^{-1}\varphi_i=(h_1,\ldots,h_{n-1})$ then the change of coordinates 
  in the associated cotangent charts $\tilde{\varphi_i}$ and $\tilde{\varphi_j}$ is given by:
  \[ \xymatrix{ & \ T^*M  \ar[dr]^{\tilde{\varphi_j}^{-1}} & \\
     \C^{n-1} \times \check{\C}^{n-1} \ar[ur]^{\tilde{\varphi_i}} \ar[rr]_h & & \C^{n-1} \times \check{\C}^{n-1} \\ 
      (x_1,\ldots,x_{n-1}; \xi_1,\ldots,\xi_{n-1})&\longmapsto   &(h(x);(Dh^{-1}|_x)^T (\xi))} \]
  
  %Without loss of generality, let us take a point $([x], \chi) \in T^*\P^{n-1} \setminus \{\mathrm{zero \; section}\}$
  %which lies in the intersection of the charts $\phi_0$ and $\phi_1$.  
  
\end{proof}
\noindent
By Lemma \ref{inclema} the incidence variety $I$ inherits the Liouville 1-form $\alpha \ (:=\sum \xi_i dx_i$ locally) 
from its isomorphism with $\P T^*\P^{n-1}$. Exchanging $\P^{n-1}$ and $\check{\P}^{n-1}$, $I$ is also
isomorphic to $\P T^*\check{\P}^{n-1}$ so it also inherits the 1-form $\check{\alpha}(:=\sum x_i d\xi_i$ 
locally).  

\begin{lemma}{\rm(Kleiman; see \cite[\S 4]{Kl2})}
Let I be the incidence variety as above. Then
$\alpha + \check{\alpha}= 0$ on $I$.
\end{lemma}
\begin{proof}$\;$\newline
 Note that if the polynomial $\sum_{i=0}^{n-1} x_i\xi_i$ defined a function on $\P^{n-1} \times \check\P^{n-1}$, we 
 would obtain the result by differentiating it. The idea of the proof is basically the same, it involves 
 identifying the polynomial $\sum_{i=0}^{n-1} x_i\xi_i$ with a section of the line bundle $p^*O_{\P^{n-1}}(1) \otimes   
 \check{p}^*O_{\check\P^{n-1}}(1)$ over $I$, where $p$ and  $\check{p}$ are the natural projections of $I$ to $\P^{n-1}$ 
 and $\check\P^{n-1}$ respectively and $O_{\P^{n-1}}(1)$ denotes the canonical line bundle, introducing the appropriate flat 
 connection on this bundle, and differentiating.
\end{proof}
 
In particular, this lemma tells us that if at some point $z \in I$ we have that $\alpha=0$,
then $\check{\alpha}=0$ too. Thus, a closed conical 
irreducible analytic subvariety of $T^*\P^{n-1}$ as in Proposition \ref{conormalequiv}
 is the conormal space of its image in $\P^{n-1}$ if and only if it is the
conormal space of its image in $\check\P^{n-1}$.
So we have $\P T_V^*\P^{n-1} \subset I \subset \P^{n-1} \times \check\P^{n-1}$ and the restriction
of the two canonical projections:
\[ \xymatrix{ & \P T_V^*\P^{n-1} \subset I \ar[dl]_p \ar[dr]^{\check{p}} & \\
                V\subset \P^{n-1} & &\check\P^{n-1}\supset \check{V}}  \]

\begin{definition}
    The \textbf{dual variety $\check{V}$} of $V\subset\P^{n-1}$ is the image by the map $\check{p}$ of $\P T_V^* \P^{n-1} \subset I$ in $\check\P^{n-1}$.
So by construction $\check{V}$ is the closure in $\check\P^{n-1}$ of the set of hyperplanes 
tangent to $V^0$. 
\end{definition}

  We immediately get by symmetry that $\check{\check{V}}=V$. What is more,
we see that establishing a projective duality is equivalent to finding a Lagrangian subvariety in $I$; its images in $\P^{n-1}$ and $\check{\P}^{n-1}$
are necessarily dual. \\

\begin{lemma}\label{cone}
        Let us assume that $(X,0)\subset (\C^n,0)$ is the cone over a projective algebraic 
variety $V\subset \P^{n-1}$.
Let $x \in X^0$ be a nonsingular point of $X$. Then the tangent space $T_{X^0,x}$,
contains the line $\ell=\overline {0x}$ joining $x$ to the origin. Moreover, the tangent map at $x$ to the projection $\pi\colon X\setminus\{0\}\to V$ induces an isomorphism 
$T_{X^0,x} / \ell \simeq T_{V,\pi (x)}$.

\end{lemma}
\begin{proof}$\;$\newline
  This is due to Euler's identity for a homogeneous polynomial of degree $m$:
\[m.f=\sum_{i=1}^n x_i\frac{\partial f}{\partial x_i} \]
and the fact that if $\{f_1,\ldots, f_r\}$ is a set of homogeneous polynomials defining $X$, 
then $T_{X^0,x}$ is the kernel of the matrix:
\[\left( \begin{array}{c}d f_1 \\ \cdot \\ \cdot \\ d f_r \end{array} \right)\] representing the differentials $df_i$ in the basis $dx_1,\ldots ,dx_n$.
\end{proof}

     It is also important to note that the tangent space to $X^0$ is constant along all nonsingular points $x$ of $X$ in the same
generating line since the partial derivatives are homogeneous as well, and contains the generating line. By Lemma \ref{cone}, the quotient of this tangent space by the generating line is the tangent space to $V$ at the point corresponding to the generating line. \\

     So, $\P T_X^*\C^n$ has an image in $\check{\P}^{n-1}$ which is the projective dual
of V.\newline
\[\xymatrix {& \P T_V^*\P^{n-1}\ar[dl]\ar[dr] & &
                 \P T_X^*\C^n\subset \check{\P}^{n-1}\times \C^n \ar[dl]\ar[dr] & \\
              V\subset \P^{n-1} & & \check{\P}^{n-1}\supset \check{V}& &
              X\subset \C^n  }\]

The fiber over 0 of $\P T_X^*\C^n \to X$ is equal
to $\check{V}$  as subvariety of $\check{\P}^{n-1}$: it is the set of limit positions at 0 of hyperplanes tangent to $X^0$.\par\noindent
For more information on projective duality, in addition to Kleiman's papers one can consult \cite{Tev}.\par\medskip
A \textit{relative} version of the conormal space and of projective duality will play an important role in these notes. Useful references are \cite{HMS}, \cite{Kl2}, \cite{Te3}. The relative conormal space is used in particular to define the \textit{relative polar varieties}.\par\noindent

 Let $f:X\to S$ be a morphism of reduced analytic spaces, with purely $d$-dimensional fibers and such that there exists a closed nowhere dense
      analytic space such that the restriction to its complement $X^0$ in $X$ :
      \[f|_{X^0} : X^0 \longrightarrow S\]
      has all its fibers smooth. They are manifolds of dimension $d={\rm dim }X-{\rm dim}S$. Let us assume furthermore that the map $f$ is induced, via a closed embedding $X\subset Z$ by a smooth map $F\colon Z\to S$. This means that locally on $Z$ the map $F$ is analytically isomorphic to the first projection $S\times\C^N\to S$. Locally on $X$, this is always the case because we can embed the graph of $f$, which lies in $X\times S$, into $\C^N\times S$.\par\noindent Let us denote by $\pi_F\colon T^*(Z/S)\to Z$ the relative cotangent bundle of $Z/S$, which is a fiber bundle whose fiber over a point $z\in Z$ is the dual $T^*_{Z/S,x}$ of the tangent vector space at $z$ to the fiber $F^{-1}(F(z))$. 
      For $x \in X^0$, denote by $M(x)$ the submanifold $f^{-1}(f(x))\cap X^0$ of $X^0$. Using this submanifold we will build the \textbf{conormal space of $X$ relative to $f$}, denoted
      by $T_{X/S}^*(Z/S)$, by setting
      \[N^*_{M(x),x} = \{ \xi \in T^*{Z/S,x}\vert \xi (v)=0, \; \forall v \in T_{M(x),x}\}\] and
      \[T^*_{X^0/S}(Z/S) =\{ (x,\xi) \in T^*(Z/S) \vert\  x \in X^0, \; \xi \in N^*_{M(x),x}\}, \] 
      and finally taking the closure of $T^*_{X^0/S}(Z/S)$ in $T^*(Z/S)$, which is a complex analytic space $T^*_{X/S}(Z/S)$ by general theorems (see \cite{Re}, \cite{K-K}). Since $X^0$ is dense in $X$, this closure maps onto $X$ by the natural projection $\pi_F\colon T^*(Z/S)\to Z$.\par 
      Now we can projectivize with respect to the homotheties on $\xi$, as in the case where $S$ is a point we have seen above. We obtain the (projectivized) relative conormal space $C_f(X)\subset \P T^*(Z/S)$ (also denoted by $C(X/S)$), naturally endowed with a map
      $$\kappa_f\colon C_f(X)\longrightarrow X.$$ We can assume that locally the map $f$ is the restriction of the first projection to $X\subset S\times U$, where $U$ is open in $\C^n$. Then we have $T^*(S\times U/S)=S\times U\times \check\C^n$ and $\P T^*(S\times U/S)=S\times U\times \check\P^{n-1}$. This gives an inclusion $C_f(X)\subset X\times\check \P^{n-1}$ such that $\kappa_f$ is the restriction of the first projection,  and a point of $C_f(X)$ is a pair $(x,H)$, where $x$ is a point of $X$ and $H$ is a limit direction at $x$ of hyperplanes of $\C^n$ tangent to the fibers of the map $f$ at points of $X^0$. By taking for $S$ a point we recover the classical case studied above.\par
\begin{definition}\label{RELA} Given a smooth morphism $F\colon Z\to S$ as above, the projection to $S$ of $Z=S\times U$, with $U$ open in $\C^n$, we shall say that a reduced complex subspace $W\subset T^*(Z/S)$ is \textbf{$F$-Lagrangian} (or \textbf{$S$-Lagrangian} if there is no ambiguity on $F$) if the fibers of the composed map $q:=(\pi_F\circ F)\vert W\colon W\to S$ are purely of dimension $n={\rm dim}Z-{\rm dim}S$ and the differential $\omega_F$ of the relative \linebreak Liouville differential form $\alpha_F$ on $\C^n\times\check\C^n$ vanishes on all pairs of tangent vectors at smooth points of the fibers of the map $q$.
\end{definition}
With this definition it is not difficult to verify that $T^*_{X/S}(Z/S)$ is $F$-Lagrangian, and by abuse of language we will say the same of $C_f(X)$. But we have more:
\begin{proposition}{\rm (L\^e-Teissier, see \cite{L-T2}, proposition 1.2.6)}\label{SpecLag} Let $F\colon Z\to S$ be a smooth complex analytic map with fibers of dimension $n$. Assume that $S$ is reduced. Let $W\subset T^*(Z/S)$ be a reduced closed complex subspace and set as above $q= \pi_F\circ F\vert W\colon W\to S$. Assume that the dimension of the fibers of $q$ over points of dense open analytic subsets $U_i$ of the irreducible components $S_i$ of $S$ is $n$. 
\begin{enumerate}
\item If the Liouville form on $T^*_{F^{-1}(s)}=(\pi_F\circ F)^{-1}(s)$ vanishes on the tangent vectors at smooth points of the fibers $q^{-1}(s)$ for $s\in U_i$ and all the fibers of $q$ are of dimension $n$, then the Liouville form vanishes on tangent vectors at smooth points of all fibers of $q$.
\item The following conditions are equivalent:
\begin{itemize}
\item The subspace $W\subset T^*(Z/S)$ is $F$-Lagrangian;
\item The fibers of $q$, once reduced, are all purely of dimension $n$ and there exists a dense open subset $U$ of $S$ such that for $s\in U$ the fiber $q^{-1}(s)$ is reduced and is a Lagrangian subvariety of $(\pi_F\circ F)^{-1}(s)$;\par\noindent
If moreover $W$ is homogeneous with respect to homotheties on $T^*(Z/S)$, these conditions are equivalent to: 
\item All fibers of $q$, once reduced, are purely of dimension $n$ and each irreducible component $W_j$ of $W$ is equal to $T^*_{X_j/S}(Z/S)$, where $X_j=\pi_F(W_j)$.
\end{itemize}
\end{enumerate}
\end{proposition}
Assuming that $W$ is irreducible, the gist of this proposition is that if $W$ is, generically over $S$, the relative conormal of its image in $Z$, and if the dimension of the fibers of $q$ is constant, then $W$ is everywhere the relative conormal of its image. This is essentially due to the fact that the vanishing of a differential form is a closed condition on a cotangent space. In section \ref{reldu} we shall apply this, after projectivization with respect to homotheties on $T^*(Z/S)$,  to give the Lagrangian characterization of Whitney conditions.
\subsection{Tangent cone}\label{tgt}

At the very beginning we mentioned how the limit of tangent spaces can be thought of as 
a substitute for the tangent space at singular points. There is another common substitute
for the missing tangent space, the \textbf{tangent cone}.

   Let us start by the geometric definition. Let $X\subset \C^n$ be a representative of
$(X,0)$. The canonical projection $\C^n\setminus \{0\} \to \P^{n-1}$ 
induces the secant line map
\begin{align*}
s_X\colon X\setminus\{0\}&\to \P^{n-1},\\ 
x & \mapsto[0x],
\end{align*}  
 where $[0x]$ is the direction of the secant line $\overline{0x}\subset\C^n$.  Denote by $E_0X$ the closure in $X\times \P^{n-1}$ of the graph of $s_X$. $E_0X$ is an 
analytic subspace of dimension $d$, and the natural projection $e_0X\colon 
E_0X\to X$ induced by the first projection is called the \textbf{blowing up} of $0$ in $X$. The fiber 
$e_0^{-1}(0)$ is a projective subvariety of $\P^{n-1}$ of dimension $d-1$, not necessarily reduced  
(see \cite[\S 10]{Whi1}). 

\begin{definition}\label{tanconedef1}
The cone with vertex $0$ in $\C^n$ corresponding to 
the subset $|e_0^{-1}(0)|$ of the projective space $\P^{n-1}$
is the set-theoretic \textbf{tangent cone}.
\end{definition}

The construction shows that set-theoretically $e_0^{-1}(0)$ is the set of limit directions of secant lines $\overline{0x}$ for points $x\in 
X\setminus\{0\}$ tending to $0$. This means more precisely that for 
each sequence $(x_i )_{i\in \N}$ of points of $X\setminus\{0\}$, tending to $0$ as $i\to\infty$ we can, 
since $\P^{n-1}$ is compact, extract a subsequence such that the directions $[0x_i]$ of the secants
$\overline{0x_i}$ converge. The set of such limits is the underlying set of $e_0^{-1}(0)$ (see \cite[Theorem 5.8]{Whi2}). \\

The algebraic definition is this: $\Oo=\Oo_{X,0}= \C\{z_1,\ldots,z_n\}/\left<f_1,\ldots, f_k
\right>$ be the local algebra of $X$ at $0$ and let $\m=\m_{X,0}$ 
be its maximal ideal. There is a natural filtration of $\Oo_{X,0}$ by 
the powers of $\m$:
\[\Oo_{X,0}\supset \m \supset\cdots\supset \m^i\supset \m^{i+1}\supset\cdots,\]
which is \textit{separated} in the sense that $\bigcap_{i=o}^\infty 
\m^i=(0)$ because the ring $\Oo_{X,0}$ is local and N\oe therian. Thus, for any element $h\in \Oo_{X,0}$ there is a unique integer $\nu(h)$ such that $h\in \m^{\nu(h)}\setminus \m^{\nu(h)+1}$. It is called the $\m$-adic order of $h$ and the image of $h$ in the quotient $\m^{\nu(h)}/ \m^{\nu(h)+1}$, which is a finite dimensional $\C$-vector space, is called the \textit{initial form} of $h$ for the $\m$-adic filtration. Initial forms are the elements of a ring, which we now define in the case of immediate interest. The general definition \ref{initialform} is given below :

\begin{definition}
We define the \textbf{associated graded ring of $\Oo$ with respect to $\m$}, written ${\rm gr}_{\m}O$ to
be the graded ring
\[{\rm gr}_{\m}\Oo\colon = \bigoplus_{i\geq 0}\m^i/ \m^{i+1},\]
where $\m^0=\Oo$ and the multiplication is induced by that of $\Oo$.
\end{definition}

   Note that ${\rm gr}_{\m}\Oo$ is generated as $\C-$algebra by $\m/\m^2$, which is a finite dimensional
vector space. Thus, ${\rm gr}_{\m}\Oo$ is a finitely generated $\C-$algebra, to which we can associate 
a complex analytic space ${\rm Specan}{\rm gr}_{\m}\Oo$. It is nothing but the affine algebraic variety ${\rm Spec}{\rm gr}_{\m}\Oo$ viewed as a complex analytic space with the sheaf of holomorphic functions replacing the sheaf of algebraic functions. Moreover, since ${\rm gr}_{\m}\Oo$ is graded and 
finitely generated in degree one, the associated affine variety ${\rm Specan}{\rm gr}_{\m}\Oo$ is a cone.
(For more on Specan, see \cite[Appendix I, 3.4 and Appendix III 1.2]{He-Or} or additionally \cite[p. 172]{K-K}). 
\begin{definition}\label{tanconedef2}
   We define the \textbf{tangent cone} $C_{X,0}$ as the complex analytic space ${\rm Specan}({\rm gr}_{\m}\Oo)$.
\end{definition}
   
    We have yet to establish the relation between the geometric and algebraic definitions of
the tangent cone. In order to do that we will need to introduce the \textbf{specialization of $\mathbf{X}$ to its 
tangent cone}, which is a very interesting and important construction in its own right.\\

    Take the representative $(X,0)$ of the germ associated to the analytic algebra $\Oo$ from above. Now,
 the convergent power series $f_1,\ldots, f_k$ define analytic functions in a small enough polycylinder $P(\alpha):=\{z \in \C^n \; : \vert z_i \vert < \alpha_i \}$
 around $0$. Suppose additionally that 
 the initial forms of the $f_i$'s generate the homogeneous ideal of initial forms of elements of
 $I=\left<f_1,\ldots,f_k\right>$.\\
     Let $f_i=f_{i,m_i}(z_1,\ldots,z_n)+f_{i,m_{i+1}}(z_1,\ldots,z_n)+f_{i,m_{i+2}}(z_1,\ldots,z_n)+\ldots$, and set 
    \[F_i:= v^{-m_i}f_i(vz_1,\ldots,vz_n)=\] \[f_{i,m_i}(z_1,\ldots,z_n)+vf_{i,m_{i+1}}(z_1,\ldots,z_n)+v^2f_{i,m_{i+2}}(z_1,\ldots,z_n)+\ldots\in 
     \C [[v,z_1,\ldots,z_n]].\]
   Note that the series $F_i$, actually converge in the domain of $\C^{n}\times \C$ defined by the inequalities
 $\vert vz_i\vert< \alpha_i $ thus defining analytic functions on this open set. Take the analytic space
 $ \X \subset \C^n \times \C$ defined by the $F_i$'s and the analytic map defined by the projection to the 
 $t$-axis.
 \[\xymatrix{\ \ \ \ \ \ \ \ \ \ \ \ \ \ \X \subset \C^n \times \C \ar[d]_p\\ \C }\]
 Now we have a family of analytic spaces parametrized by an open subset of the complex line $\C$.
 Note that for $v \neq 0$, the analytic space $p^{-1}(v)$ is isomorphic to $X$ and in fact for
 $v=1$ we recover exactly the representative of the germ $(X,0)$ with which we started. But for
 $v=0$, the analytic space $p^{-1}(0)$ is the closed analytic subspace of $\C^n$ defined by
 the homogeneous ideal generated by the initial forms of elements of $I$.\par\medskip
 
    We need a short algebraic parenthesis in order to explain the relation between this
ideal of initial forms and our definition of tangent cone (definition \ref{tanconedef2}).
 \par\medskip
    Let $R$ be a Noetherian ring, and $I\subset J \subset R$ ideals such that
    \[R\supset J \supset\cdots\supset J^i \supset J^{i+1}\supset\cdots\ \ .\]	 
is a separated filtration in the sense that $\bigcap_{i=o}^\infty J^i=(0)$ (see \cite[Chap. III, \S 3, No.2]{B1}).

    Take the quotient ring $A=R/I$, define the ideal $\tilde{J_i}:= (J^i + I)/I \subset A$ and consider the 
induced filtration
    \[A\supset \tilde{J} \supset\cdots\supset \tilde{J_i} \supset \tilde{J}_{i+1}\supset\cdots.\]
Note that in fact $\tilde{J}_i$ is the image of $J^i$ in $A$ by the quotient map.

    Consider now the associated graded rings
 \[{\rm gr}_JR= \bigoplus_{i=0}^\infty J^i/J^{i+1}, \]	
 \[{\rm gr}_{\tilde{J}}A= \bigoplus_{i=0}^\infty \tilde{J}_i/\tilde{J}_{i+1}. \]  	
 
 \begin{definition}\label{initialform}
  Let $f \in I$. Since $\bigcap_{i=o}^\infty J^i=(0)$, there exists a largest natural number
  $k$ such that $f \in J^k$. Define the \textbf{initial form} of $f$ with respect to $J$ as
  \[{\rm in}_J f := f \;\; \mathrm{(mod}\; J^{k+1}) \; \in {\rm gr}_JR. \]
  
 Using this, define the \textbf{ideal of initial forms}, or \textbf{initial ideal} of $I$ as the ideal
  of ${\rm gr}_JR$ generated by the initial forms of all the element of $I$. 
  \[{\rm In}_JI:=\left<\mathrm{in}_J f\right>_{f\in I} \subset {\rm gr}_JR.\]
 \end{definition}
 
 \begin{lemma}\label{key}
  Using the notations defined above, the following sequence is exact:
  \[ \xymatrix { 0 \ar[r] & {\rm In}_JI \ar@{^{(}->}[r] & {\rm gr}_JR \ar[r]^\phi & {\rm gr}_{\tilde{J}}A 
     \ar[r] & 0 } \]
  that is, ${\rm gr}_{\tilde{J}}A\cong {\rm gr}_JR/{\rm In}_JI$.
 \end{lemma}
 \begin{proof}$\;$\\
  First of all, note that
  \[ \J_i/\J_{i+1} \cong \frac{\frac{J^i+I} {I}}{\frac{J^{i+1}+I}{I}} \cong 
      \frac{J^i+I}{J^{i+1}+I} \cong \frac{J^i}{I\cap J^i + J^{i+1}},\] 
  where the first isomorphism is just the definition, the second one is one of the classical isomorphism
  theorems and the last one comes from the surjective map $J^i \rightarrow \frac{J^i+I}{J^{i+1}+I}$
  defined by $x \mapsto x + J^{i+1}+I$. This last map tells us that there are natural 
  surjective morphisms:
  \begin{align*}
    \varphi_i: \frac{J^i}{J^{i+1}} & \longrightarrow \frac{\J_i}{\J_{i+1}}\cong \frac{J^i}{I\cap J^i + J^{i+1}},\\ 
      x + J^{i+1} & \longmapsto x + I\cap J^i + J^{i+1} ,    
  \end{align*}
  which we use to define the surjective graded morphism of graded rings \linebreak
  $\phi: {\rm gr}_JR \to {\rm gr}_\J A $. Now all that is left to prove is that the kernel of $\phi$
  is exactly ${\rm In}_JI$.
  
  Let $f \in I$ be such that $\mathrm{in}_Jf= f + J^{k+1} \in J^k/J^{k+1}$, then
  \[\phi({\rm in}_Jf)= \varphi_k (f + J^{k+1})= f + I\cap J^k + J^{k+1}= 0.\]
  because $f \in I\cap J^k$. Since by varying $f \in I$ we get a set of generators of the ideal ${\rm In}_JI$,
  we have ${\rm In}_JI\subset Ker \;\phi$.  
  
  To prove the other inclusion, let $g=\bigoplus \overline{g_k} \in Ker \; \phi$, where we
  use the notation $\overline{g_k}:= g_k + J^{k+1} \in J^k/J^{k+1}$. Then, $\phi(g)=0$ implies 
  by homogeneity \linebreak $\phi(\overline{g_k})=\varphi_k(g_k + J^{k+1})=0$ for all $k$.
 Suppose that $\overline{g_k}\neq 0$ then 
  \[\varphi(g_k + J^{k+1})= g_k + I\cap J^k + J^{k+1}= 0 \]
  implies $g_k = f + h$, where $0 \neq f \in (I\cap J^k)\setminus J^{k+1}$ and $h$ belongs to 
  $J^{k+1}$. This means that $g_k \equiv f \;\; ( \mathrm{mod} \; J^{k+1})$,
  which implies $g_k + J^{k+1}=\mathrm{in}_Jf$ and concludes the proof.
 \end{proof}

In order to relate our definition of the tangent cone with the space we obtained in our previous description of 
 the specialization, just note that in our case the roles of $R$ and $J\subset R$ are played by the ring of convergent power 
 series $\C\{z_1,\ldots,z_n\}$, and its maximal ideal $\m$ respectively, while $I$ corresponds to the ideal
 $<f_1,\ldots,f_k>$ defining the germ $(X,0)\subset (\C^n,0)$ and $A$ to its analytic algebra $O_{X,0}$.
 
 More importantly, the graded ring ${\rm gr}_{\m}R$, with this choice of $R$, is naturally isomorphic to the ring of 
 polynomials $\C[z_1,\ldots,z_n]$ in such a way that definition \ref{initialform} coincides with the usual concept of 
 initial form of a series and tells us that
 \[{\rm gr}_{\m} \Oo_{X,0} \cong \frac{\C[z_1,\ldots,z_n]}{{\rm In}_{\m}I}. \]
 
 We would like to point out that there is a canonical way of working out the specialization in the algebraic setting
 that, woefully, cannot be translated word for word into the analytic case, but which, nonetheless 
 takes us to a weaker statement. Suppose that $X$ is an algebraic variety, that is, the $f_i$'s
 are polynomials in $z_1,\ldots, z_n$, and
 consider the  \textit{extended Rees Algebra}\footnote{This algebra was introduced for an ideal $J\subset \Oo$ by D. Rees in \cite{Rs}, in the form $\Oo[v,Jv^{-1}]\subset \Oo[v,v^{-1}]$.} of $\Oo=\Oo_{X,0}$ with respect to $\m$ (see \cite[Appendix]{Za}, or  \cite[section 6.5]{Eis}):
  \[\mathcal{R} = \bigoplus_{i \in \Z} \m^i v^{-i} \subset \Oo[v,v^{-1}],\]
  where we set $\m^i=\Oo$ for $i\leq 0$. Note that $\mathcal{R}\supset \Oo[v] \supset \C[v]$, in fact it 
  is a finitely generated $\Oo-$algebra and consequently a finitely generated $\C$-algebra 
  (See \cite[pages 120-122]{Mat}). Moreover we have:\\
   
\begin{proposition}\label{Reesalgebra}
Let $\mathcal{R}$ be the extended Rees algebra defined above. Then:
\begin{itemize}
\item[i)] The $\C[v]-$algebra $\mathcal{R}$ is torsion free.
\item[ii)] $\mathcal{R}$ is faithfully flat over $\C[v]$
\item[iii)] The map $\phi:\mathcal{R}\to {\rm gr}_{\m}\Oo$ sending $xv^{-i}$ to the image of $x$ in 
           $\m^i/\m^{i+1}$ is well defined and induces an isomorphism $\mathcal{R}/(v\cdot\mathcal{R}) 
           \simeq {\rm gr}_{\m}\Oo$.
\item[iv)] For any $v_0 \in \C\setminus \{0\}$ the map of $\C$-algebras $\mathcal{R} \to  \Oo$ determined by
           $xv^{-i} \mapsto xv_0^{-i}$ induces an isomorphism $\frac{\mathcal{R}}{(v-v_0)\cdot \mathcal{R}}
            \simeq \Oo$.
\end{itemize}
\end{proposition}
\begin{proof}$\;$\newline
See the appendix written by the second author in \cite{Za}, \cite[p. 171]{Eis}, or additionally \cite{B2} in the exercises for \S 6 of of Chap VIII.
\end{proof}

The proposition may be a little technical, but what it says is that the extended Rees algebra 
is a way of producing flat degenerations of a ring to its associated graded ring, since the inclusion morphism 
$\C[v] \hookrightarrow \mathcal{R}$ is flat. Now taking the space $\X$ associated
to $\mathcal{R}$ and the map $\X\to \C$ associated to the inclusion $\C[v] \hookrightarrow \mathcal{R}$, we obtain 
a map:
\[\varphi: \X \longrightarrow \C \]
such that 
\begin{itemize}
\item  The map $\varphi$ is faithfully flat.
\item  The fiber $\varphi^{-1}(0)$ is the algebraic space associated to $\mathrm{gr}_{\m}O$, that is,
                the tangent cone $C_{X,0}$.
\item  The space $\varphi^{-1}(v_0),$ is  isomorphic to $X$, for all $v_0\neq 0$.
\end{itemize}
Thus, we have produced a 1-parameter flat family of algebraic spaces
specia-\linebreak lizing $X$ to $C_{X,0}$.\\Recall that flatness means that the fibers of the map vary continuously; in this case, it means that every point in any fiber is the limit of a sequence of points points in the nearby fibers. Faithful flatness means that in addition the map is surjective; in our case surjectivity is automatic, but for example the inclusion of an open set in $X$ is flat but not faithfully flat.\\
Note that in this construction we may replace the maximal ideal $\m$ by any ideal $J$ to obtain a faithfully flat specialization of $\Oo$ to ${\rm gr}_J\Oo$.\par\medskip\noindent 
As you can see, the problem when trying to translate this into the analytic case is first of all,
that in general the best thing we can say is that the algebra $\mathcal{R}$ is finitely generated over $\Oo$, but 
not even essentially finitely generated over $\C$, which means that it cannot be viewed as corresponding to an open set in an affine algebraic variety.  \\

Given any finitely generated algebra over an analytic algebra such as $\Oo$, there 
is a ``smallest'' analytic algebra which contains it, in the sense that any map of $\C$-algebras from our algebra
to an analytic algebra factors uniquely through this ``analytization''. The proof: our algebra is a 
quotient of a polynomial ring $\Oo[z_1,\ldots, z_s]$ by an ideal $I$; take the quotient of the corresponding convergent 
power series ring $\Oo\{z_1,\ldots, z_s\}$, which is an analytic algebra, by the ideal generated by $I$; it is again an analytic algebra. So we can use this to translate our result into 
a similar one which deals with germs of analytic spaces.\\

Taking the analytic algebra $\mathcal{R}^h$ associated to $\mathcal{R}$, and the analytic germ 
$\X$ associated to $\mathcal{R}^h$, we have a germ of map induced by the inclusion $\C\{t\} 
\hookrightarrow \mathcal{R}^h$:
\[ \varphi: (\X,0) \longrightarrow (\mathbf{D},0) \] 
which preserves all the properties established in the algebraic case, that is:
\begin{itemize}
\item  The map $\varphi$ is faithfully flat.
\item The fiber $\varphi^{-1}(0)$ is the germ of analytic space associated to $\mathrm{gr}_{\m}\Oo$, that is
                the tangent cone $C_{X,0}$
\item The fiber $\varphi^{-1}(v_0)$ is a germ of analytic space isomorphic to $(X,0)$, for all $v_0\neq 0$.
\end{itemize}
This means that we have produced a 1-parameter flat family of germs of analytic spaces
specializing $(X,0)$ to $(C_{X,0},0)$. The way this construction relates to our previous
analytic construction is explained in the next exercise.

\begin{exercise}\label{ejercontan}$\;$\newline
1) Suppose $(X,0)$ is a germ of hypersurface.\par\noindent Then $\Oo=\C\{z_1,\ldots,z_n\}/\left<f(z_1,\ldots,z_n)\right>$.
Show that\par\noindent  $\mathcal{R}^h=\C\{v,z_1,\ldots,z_n\}//\left<v^{-m}f(vz_1,\ldots,tz_n)\right>$. \par\noindent
Note that this makes sense since, as we saw above, writing \[f=f_m(z_1,\ldots,z_n)+f_{m+1}(z_1,\ldots,z_n)+ \ldots,\]
where $f_k$ is an homogeneous polynomial of degree $k$, then:
\[v^{-m}f(vz_1,\ldots ,vz_n)= f_m(z_1,\ldots,z_n)+vf_{m+1}(z_1,\ldots,z_n)+ \ldots\ \in\C\{v,z_1,\ldots,z_n\}.\]
2) More generally, take $I\subset \C\{z_1,\ldots,z_n\}$ and choose generators $f_i$ such
that their initial forms $f_{i,m_i}$ generate the ideal of all initial forms of elements
of $I$. Then:
\[\mathcal{R}^h=\C\{v,z_1,\ldots,z_n\}/\left<v^{-m_i}f_i(vz_1,\ldots,vz_n)\right>.\]
\end{exercise}
\noindent Let $\D^*$ be the punctured disk. It is important to note that this computation implies that the biholomorphism $\C^n\times \D^*\rightarrow \C^n\times \D^*$ determined by $(z,v)\mapsto (vz,v)$ induces an isomorphism $\varphi^{-1}(\D^*)\simeq X\times \D^*$.

Finally, we can use this constructions to prove that our two definitions of the tangent cone
are equivalent.

\begin{proposition}\label{limsec}
Let $|C_{X,0}|$ be the underlying set of the analytic space $C_{X,0}$. Then, generating lines
in $|C_{X,0}|$ are the limit positions of secant lines $\overline{0x_i}$ as $x_i \in X\setminus \{0\}$ tends
to 0.
\end{proposition}
\begin{proof}$\;$\newline
Since $\varphi:(\X,0)\to (\mathbf{D},0)$ is faithfully flat, the special fiber of the map $\varphi$ is contained in the closure of $\varphi^{-1}(\D^*)$ (see \cite[Proposition 3.19]{Fi}). The isomorphism $\varphi^{-1}(\D^*)\simeq X\times \D^*$ which we have just seen shows that for every point 
$\overline{x} \in \varphi^{-1}(0)=C_{X,0}$ there are sequences of points $(x_i,v_i) \in X \times \D^*$ 
tending to $\overline{x}$. Thus $\overline{x}$ is in the limit of
secants $\overline{0x_i}$.\newline
\end{proof}
So, we finally know that our two concepts of tangent cone coincide, at least set-theoretically.
In general the tangent cone contains very little information about $(X,0)$, as shown by the next example.

\begin{ex}$\;$\newline
For all curves $y^2-x^m$, $m\geq 3$, the tangent cone is $y^2=0$ and
it is  non-reduced. 
\end{ex}
We shall see below in Proposition \ref{propnc} that the \textit{constancy} of multiplicity along a nonsingular subspace does contain useful geometric information.
 \begin{remark}\label{initial}A collection of functions in $I$ whose initial forms generate the initial ideal $\mathrm{in}_{\m}I$ certainly generates $I$. The converse is not true (exercise) and if the initial forms of the $f_i$ do not generate the initial ideal, the equations $v^{-m_i}f_i(vz_1,\ldots,vz_n)=0$ seen above de not describe the specialization to the tangent cone because the special fiber is {\rm not} the tangent cone; in fact they do not describe a flat degeneration. We have seen how important this condition in the proof of Proposition \ref{limsec} and shall see it again in section \ref{sec:specialization}. A set of elements of $I$ whose initial forms generate the initial ideal is called a {\rm standard basis} of $I$.\end{remark}

\subsection{Multiplicity}\label{mul}

Nevertheless the analytic structure of $C_{X,0}={\rm Specan}{\rm gr}_{\m}O$ does
carry some significant piece of information on $(X,0)$, its  multiplicity. \\

For a hypersurface, $f=f_m(z_1,\ldots,z_n)+f_{m+1}(z_1,\ldots,z_n)+ \ldots$, the 
multiplicity at 0 is just $m$=the degree of the initial polynomial. And, from the example
above, its tangent cone is also a hypersurface with the same multiplicity at 0 in this sense.
Although the algebraic relation between a germ and its tangent cone is more complicated in general, this equality of multiplicities is preserved as we are going to see. \\ 

Let $\Oo$ be the analytic algebra of $(X,0)\subset (\C^n,0)$, with maximal ideal $\m$ as before. We have
the following consequences of the fact that $\Oo$ is a Noetherian $\C$-algebra:\\
\begin{itemize}
\item For each $i\geq 0$, the quotient $\Oo/\m^{i+1}$ is a finite dimensional vector space over $\C$,
   and the generating function:
\[\sum_{i\geq0}({\rm dim}_{\C}\; \Oo/\m^{i+1})T^i= \frac{Q(T)}{(1-T)^d} \]
is a rational function with numerator $Q(T)\in \mathbb{Z}[T]$, and $Q(1) \in \mathbb{N}$.
See \cite[p. 117-118]{A-M}, or \cite[chap VIII]{B2}.
\item For large enough $i$: 
   \[{\rm dim}_{\C} \; \Oo/m^{i+1}=e_{\m}(\Oo)\frac{i^d}{d!} + {\rm lower \; order \; terms},\]
   and $e_{\m}(\Oo)=Q(1)$ is called \textbf{the multiplicity of} $X$ \textbf{at} $0$, which we will denote
   by $m_0(X)$. See  \cite[chap VIII]{B2}, \cite[\S14]{Mat}.
\item A linear space $L_d+t$ of dimension $n-d$ at a sufficiently small distance $\vert t\vert>0$
   from the origin and of general direction has the property that for a sufficiently small ball $\B(0,\epsilon)$ centered at $0$ in $\C^n$, if $\epsilon$ is small enough and  $\vert t\vert$ small enough with respect to $\epsilon$, $L_d+t$ meets $X\cap \B(0,\epsilon)$ transversally at $m_0(X)$ nonsingular points of $X$, which tend to $0$ with $\vert t\vert$. See \cite[ p. 510-555]{He-Or}.
\item The multiplicity of $X$ at 0 coincides with the multiplicity of $C_{X,0}$ at 0. This follows from the fact that the generating function defined above is the same for the $\C$-algebra ${\mathcal O}$ and for ${\rm gr}_{\m} {\mathcal O}$. See \cite[Chap. VIII, \S 7]{B2}, and also \cite[Thm 5.2.1 \& Cor.]{He-Or}. 
\item If ${\rm gr}_{\m} {\mathcal O}=\C[T_1,\ldots ,T_n]/\left<F_1,\ldots ,F_c\right>$ where the $F_i$ are homogeneous polynomials of respective degrees $d_i$ forming a regular sequence in the polynomial ring, in the sense that for each $i$, the image of $F_i$ in the quotient $\C[T_1,\ldots ,T_n]/\left<F_1,\ldots , F_{i-1}\right>$ is not a zero divisor, then $e_{\m}(\Oo)=d_1\ldots d_c$. See \cite[chap VIII, \S 7, no. 4]{B2}.
\end{itemize}

\section{Normal Cone and Polar Varieties:\\ the normal/conormal diagram}
The normal cone is a generalization of the idea of tangent cone, where the point is replaced by a closed analytic subspace, say $Y\subset X$. If $X$ and $Y$ were nonsingular it would be the normal bundle of $Y$ in $X$. We will only consider
the case where $Y$ is a nonsingular subspace of $X$ and denote by $t$ its dimension (an integer, and not a vector of $\C^n$ as in the previous section !).\\
We will take a global approach here. Let $(X,\Oo_X)$ be a reduced complex analytic space
of dimension $d$ and $Y\subset X$ a closed complex subspace defined by a coherent sheaf of ideals 
$J \subset \Oo_X$. It consists, for every open set $U\subset X$ of all elements of $\Oo_X(U)$ vanishing on 
$Y\cap U$, and as one can expect the structure sheaf $\Oo_Y$ is isomorphic to $(\Oo_X/J)|_Y$. 
Analogously to the case of the tangent cone, let us consider $\mathrm{gr}_J\Oo_X$, but now as the 
associated sheaf of graded rings of $\Oo_X$ with respect to $J$:
\[{\rm gr}_J\Oo_X= \bigoplus_{i\geq0} J^i/J^{i+1}= \Oo_X/J \oplus J/J^2 \oplus \cdots \ .\]

\begin{definition}\label{defnc}
We define the \textbf{normal cone $C_{X,Y}$ of $X$ along $Y$}, as the complex analytic space
${\rm Specan}_Y (\mathrm{gr}_J\Oo_X)$.
\end{definition}

   Note that we have a canonical inclusion $\Oo_Y \hookrightarrow \mathrm{gr}_J\Oo_X$, which gives $\mathrm{gr}_J\Oo_X$ the structure of
a locally finitely presented graded $\Oo_Y$-algebra and consequently, by the Specan construction, a canonical analytic 
projection $C_{X,Y} \stackrel{\Pi}{\longrightarrow} Y$, in which the fibers are cones. The natural surjection $\mathrm{gr}_J\Oo_X\to \Oo_X/J=\Oo_Y$ obtained by taking classes modulo the ideal $\bigoplus_{i\geq 1} J^i/J^{i+1}$ corresponds to an analytic section $Y\hookrightarrow C_{X,Y}$ of the map $\Pi$ sending each point $y\in Y$ to the vertex of the cone $\Pi^{-1}(y)$.\par\noindent
 To be more precise, note that the sheaf of graded $\Oo_X$-algebras $\mathrm{gr}_J\Oo_X$ is a sheaf on $X$ with support $Y$. Using the fact that $J$ and all its powers are coherent $\Oo_X$-modules and that every point $x\in X$ has a basis of neighborhoods $U_\alpha$ such that $\Oo_X(U_\alpha)$ is noetherian it is not difficult to prove that  $\mathrm{gr}_J\Oo_X$ is a locally finitely presented graded $\Oo_Y$-algebra: the space $Y$ is covered by open sets $V$ where the restriction of $\mathrm{gr}_J\Oo_X$ has  a presentation of the form:
	\[ {\rm gr}_J\Oo_X(V) \cong \frac{\Oo_Y(V)\left[T_1, \ldots , T_r \right]}{\left<g_1,\ldots, g_s\right>}.\]
where the $g_i$'s are homogeneous polynomials in $\Oo_Y(V)\left[T_1, \ldots , T_r \right]$ and the images of the $T_j$ in the quotient by $\left<g_1,\ldots, g_s\right>$ are a system of generators of the $\Oo_Y(V)$-module $(J/J^2)(V)$, the image in the quotient of a system of generators of the ideal $J(U)\subset\Oo_X(U) $, where $U$ is an open set of $X$ such that $U\cap Y=V$. The ideal $\left<g_1,\ldots, g_s\right>$ then defines a closed subset $C_{X,Y}\vert V$ of $V\times \C^r$ which is invariant by homotheties on the $T_j's$ and is the restriction $\Pi^{-1}(V)$ over $V$ of the normal cone. \\

  Let us now build, in analogy to the case of the tangent cone, the \textbf{specialization of} 
$X$ \textbf{to the normal cone of} $Y$.\par\noindent Let us first take a look at it in the algebraic case,
when we suppose that $X$ is an algebraic variety, and $Y\subset X$ a closed algebraic subvariety defined
by a coherent sheaf of ideals $J \subset \Oo_X$.\\

   Keeping the analogy with the tangent cone and the Rees algebra technique, we consider the locally
 finitely presented sheaf of graded $\Oo_X$-algebras 
\[\mathcal{R} = \bigoplus_{n \in \Z} J^n v^{-n} \subset \Oo_X[v,v^{-1}],\ \ \ \mathrm{where} \;\; J^n=\Oo_{X,0}\; {\rm for}\ n\leq 0.
\]
Note that we have $\C[v]\subset \Oo_X [v]\subset \mathcal{R}$, where $\C$ denotes the constant sheaf,
thus endowing $\mathcal{R}(X)$ with the structure of a $\C[v]$-algebra that results in an algebraic
map 
\[p: {\rm Spec}\mathcal{R} \longrightarrow \C \]
Moreover, the $\C[v]$-algebra $\mathcal {R}$ has all the analogous properties of Proposition \ref{Reesalgebra}, which 
in turn gives the corresponding properties to $p$, defining a faithfully flat 1-parameter family of varieties such that:
\begin{itemize}
\item [1)] The fiber over 0 is ${\rm Spec}(\mathrm{gr}_J\Oo_X)$.
\item [2)] The general fiber is an algebraic space isomorphic to $X$.
\end{itemize} 
Thus, the map $p$ gives a \textbf{specialization} of $X$ to its normal cone $C_{X,Y}$ along $Y$.\\

     Let us now look at the corresponding construction for germs of analytic spaces. Going back to 
 the complex space $(X,\Oo_X)$, and the nonsingular subspace $Y$ of dimension $t$, take a point $0 \in Y$, and a 
 local embedding $(Y,0)\subset (X,0) \subset (\C^n,0)$. Since $Y$ is nonsingular we can assume it is linear, by 
 choosing a sufficiently small representative of the germ $(X,0)$ and adequate local coordinates on $\C^n$.
 Let $\Oo_{X,0}=\C\{z_1,\ldots,z_n\}/\left< f_1,\ldots,f_k\right>$ be the analytic algebra of the germ, where
 $J=\left<z_1,\ldots,z_{n-t}\right>$ is the ideal defining $Y$ in $X$. Consider now the finitely generated
 $\Oo_{X,0}$-algebra:
 \[\mathcal{R}= \bigoplus_{n \in \Z} J^n v^{-n}, \;\; \mathrm{where} \;\; J^n=\Oo_{X,0}\; {\rm for}\ n\leq 0.\]
 So again, taking the analytic algebra $\mathcal{R}^h$ associated to $\mathcal{R}$ and the analytic germ 
$Z$ associated to $\mathcal{R}^h$, we have a germ of map induced by the inclusion $\C\{v\} 
\hookrightarrow \mathcal{R}^h$:
\[ p: (Z,0) \longrightarrow (\mathbf{D},0) \] 
which preserves all the properties established in the algebraic case, that is:
\begin{itemize}
\item The map $p$ is faithfully flat.
\item The special fiber $p^{-1}(0)$ is the germ of analytic space associated to ${\rm gr}_J \Oo_{X,0}$, that is
                the germ of the normal cone $C_{X,Y}$
\item The fibers  $p^{-1}(v)$ are germs of analytic spaces isomorphic to $(X,0)$ for all $v\neq 0$.
\end{itemize}
Thus, we have produced a 1-parameter flat family of germs of analytic spaces
specializing $(X,0)$ to $(C_{X,Y},0)$.\\
      	
    Using this it can be shown as in the case of the tangent cone that after choosing a local retraction $\rho:(X,0)\to (Y,0)$,
the underlying set of $(C_{X,Y},0)$ can be identified with the set of limit positions of secant
lines $\overline{x_i\rho(x_i)}$ for $x_i \in X\setminus Y$ as $x_i$ tends to $y \in Y$ (for a proof of this, see \cite[\S 2]{Hi}).
We shall see more about this specialization in the global case a little later.\\

    Keeping this germ approach, with $(Y,0)\subset (X,0) \subset \C^n$, and $Y$ a linear subspace of dimension
$t$ we can now interpret definition \ref{initialform} and Lemma \ref{key} in the following way:\par\noindent Using the notations
of section \ref{tgt}, let $R:= \C\{z_1,\ldots,z_{n-t},y_1\ldots,y_t\}$, $J=\left<z_1,\ldots,z_{n-t}\right>\subset R$ the 
ideal defining $Y$, $I=\left< f_1,\ldots,f_k\right>\subset R$ the ideal defining $X$ and $A=R/I=\Oo_{X,0}$. Then,
the ring $R/J$ is by definition $\Oo_{Y,0}$ which is isomorphic to $\C \{y_1,\ldots,y_t\}$, and it is not hard to prove
that \[{\rm gr}_JR \cong \Oo_{Y,0}[z_1,\ldots,z_{n-t}].\]

    More to the point, take an element $f \in I \subset R$. Then we can write
    \[f=\sum_{(\alpha,\beta) \in \N^t \times \N^{n-t}} c_{\alpha \beta} y^{\alpha}z^{\beta}.\]
Now define $\nu_Yf= \; \mathrm{min} \; \{\vert \beta\vert \; | \; c_{\alpha \beta}\neq 0\}$.
One can prove that 
    \[{\rm in}_Jf = \sum_{|\beta|=\nu_Yf} c_{\alpha \beta} y^{\alpha}z^{\beta},\]
which after rearranging the terms with respect to $z$ gives us a polynomial in the variables $z_k$
with coefficients in $\Oo_{Y,0}$, that is, an element of $\mathrm{gr}_JR$. Note that these "polynomials" define
analytic functions in $Y\times \C^{n-t}=\C^t \times \C^{n-t}$, and thus realize, by the Specan construction, the 
germ of the normal cone $(C_{X,Y},0)$ as a germ of analytic subspace of $(\C^n,0)$ with a canonical analytic map
to $(Y,0)$. Let us clarify all this with an example.
  
\begin{ex}
   Take $(X,0) \subset (\C^3,0)$ defined by $x^2-y^2z=0$, otherwise known as Whitney's umbrella. Then
 from what we have discussed we obtain:
 \begin{itemize}
   \item [i)] The tangent cone at $0$, $C_{X,0} \subset\C^3 $, is the analytic subspace defined by $x^2=0$.
   \item[ii)] For $Y=z$-axis, the normal cone $C_{X,Y} \subset \C^3$ of $X$ along $Y$ is the analytic
              subspace defined by $x^2-y^2z=0$, that is the space $X$ itself, which is a cone with vertex $Y$.
   \item[iii)] For $Y=y$-axis, the normal cone $C_{X,Y} \subset \C^3$ of $X$ along $Y$ is the analytic
              subspace defined by $y^2z=0$.         
 \end{itemize}
\end{ex}

\begin{proposition}{\rm (Hironaka, Teissier)}\label{propnc} 
Given a t-dimensional closed nonsingular subspace $Y\subset X$ and a point $y \in Y$, let $T_{Y,y}$ denote the tangent space to $Y$ at $y$. For any local embedding 
$(Y,0) \subset (X,0) \subset \C^n$, the following conditions are equivalent:
\begin{itemize}
\item[i)] The multiplicity $m_y(X)$ of $X$ at the points $y \in Y$ is locally constant on $Y$
near 0.
\item[ii)] The dimension of the fibers of the maps $C_{X,Y}\to Y$ is locally constant on $Y$
near 0.
\item[iii)] For every point $y \in Y$ there exists a dense Zariski open set $D$ of the Grassmanian of $(n-d+t)$-dimensional linear subspaces of $\C^n$ containing
$T_{Y,y}$ such that if $W$ is a representative in an open $U\subset \C^n$ of a germ $(W,y)$ at $y$ of a nonsingular $(n-d+t)$-dimensional subspace of $\C^n$ containing $Y$ and whose tangent space at $y$ is in $D$, there exists an open neighborhood $B\subset U$ of $y$ in $\C^n$ 
such that:
\[|W\cap X \cap B|=Y \cap B,\]
where $\vert Z\vert$ denotes the reduced space of $Z$.\\ In short, locally the intersection with $X$ of a general subspace containing $Y$ and whose intersection with $X$ is of dimension $t$, is $Y$ and nothing more. Any other component of this intersection would "bring more multiplicity" to $X$ at the origin.
\end{itemize}
\end{proposition}
\begin{proof}$\;$\\
See \cite[\S 6]{Hi}, \cite[Appendix III, theorem 2.2.2]{He-Or}, and for iii) see \cite[Chapter I, 5.5]{Te3}. The meaning of the last statement is that if the equality is not satisfied, there are $t$-dimensional components of the intersection $|W\cap X \cap B|$, distinct from $Y\cap B$, meeting $Y$ at the point $y$; the multiplicity of $X$ at $y$ must then be larger than at general nearby points of $Y$.
\end{proof} 

\begin{ex}
To illustrate the equivalence of $i)$ and $iii)$, let us look again at Whitney's umbrella $(X,0) \subset (\C^3,0)$ defined by $x^2-y^2z=0$,
and let $Y$ be the $y$-axis, along which $X$ is {\rm not} equimultiple at the origin.
Taking $W$ as the nonsingular 2-dimensional space defined by $z=ax$ gives for the intersection with $X$
\[z-ax=0,\ \ \ \ \ x(x-ay^2)=0,\]
so that whenever $a\neq 0$ the intersection $W\cap X$ has two irreducible components at the origin:
the $y$-axis and the curve defined by the equations $x=ay^2$, $z=ax$.\par\medskip
We can do the same taking for $Y$ the $z$-axis,  along which $X$ {\rm is} equimultiple at the origin, and $W$ defined by $y=ax$. We obtain \[y-ax=0,\ \ \ \ \ x^2(1-az^2)=0,\]
which locally defines the $z$-axis.
\end{ex}
 
    In order to understand better the normal cone and Proposition \ref{propnc}, we are going
to introduce the blowing up of $X$ along $Y$. As in the case of the tangent cone, let us start
by a geometric description. \\

     Let $(Y,0)\subset (X,0)$ be germ of nonsingular subspace of dimension $t$ as before. 
Choose a local analytic retraction $\rho:\C^n \to Y$, a decomposition $\C^n\simeq Y\times\C^{n-t}$ such that $\rho$ coincides with the first projection and use it to define a map:
\begin{align*}
\phi: X\setminus Y & \longrightarrow \P^{n-1-t}\\
         x   & \longmapsto         {\rm direction\  of }\; \ \overline{x\rho (x)}\subset\C^{n-t}.
\end{align*}
Then consider its graph in $(X\setminus Y) \times \P^{n-1-t}$. Note that, since we can assume
that $Y$ is a linear subspace, in a suitable set of coordinates the map $\rho$ is just the canonical linear
projection $\rho:\C^n \to \C^t$. Moreover, the map $\phi$ maps $x=(x_1,\ldots,x_n)\in X\setminus Y 
\mapsto (x_{t+1}:\cdots:x_n)\in \P^{n-1-t}$.

	Just as in the case of the blow-up of a point, the closure of the graph is a complex analytic space 
$E_{Y}X \subset X \times \P^{n-1-t}$ and the natural projection map $e_Y:= p \circ i$:
\[\xymatrix{ E_YX \ar@{^{(}->}[r]^i\ar[dr]_{e_Y}&  X\times \P^{n-1-t}\ar[d]^p\\
                              & X   }\]
is proper. Moreover the map $e_Y$ induces an isomorphism $E_YX\setminus e_Y^{-1}(Y) \to 
X\setminus Y$. Note that if we take an open cover of the complex space $X$, consisting only of local models, 
we can do an analogous construction in each local model (see the proof of Proposition \ref{blowup.eomyalg} below) and then paste them 
all up to obtain a global blow-up. There is an algebraic construction which will save us the effort of 
pasting by doing it all at once. Let $J\subset O_X$ be the ideal defining $Y\subset X$ as before. 

\begin{definition}
The \textbf{Rees algebra}, or  \textbf{blowing up algebra} of $J$ in $\Oo_X$ is the graded $\Oo_X$-algebra:
\[P(J)=\bigoplus_{i\geq0} J^i=\Oo_X \oplus J \oplus J^2 \oplus \cdots\ .\] 
\end{definition}
\noindent
Note that $P(J)/JP(J)\cong {\rm gr}_J\Oo_X $, the associated graded ring of $\Oo_X$ with respect to $J$. Moreover, 
since $J$ is locally generated by $n-t$ coordinates of $\C^n$ whose vanishing defines $Y\subset\C^n$, $P(J)$ is a locally finitely presented  graded $\Oo_X$-algebra,
generated in degree 1, and as such, it has locally a presentation, for suitable open sets $U\subset X$:
\[\frac{\Oo_{X\vert U}[z_1,\ldots,z_{n-t}]}{\left<g_1,\ldots,g_m\right>}\cong P(J)\vert U,\]
where the $g_i$ are homogenous polynomials in $z_1,\ldots,z_{n-t}$ with coefficients in $\Oo_X\vert U$.\\

Defining $\tilde{E}_YX$ as the \textbf{projective analytic spectrum} of $P(J)$, $\tilde{E}_YX= 
{\rm Projan} P(J)$ (see \cite[Appendix III, 1.2.8]{He-Or}), we can view this as defining a family of projective 
varieties parametrized by X, as a result of the  $\Oo_X$-algebra structure.
\[\xymatrix{\ \ \ \ \ \ \ \ \ \ \ \ \ \ \ \ \ \ \tilde{E}_YX \subset X \times \P^{n-1-t}\ar[d]\\ X }\]

To check that these two spaces are the same it is enough to check that they are the same locally for each open set of an appropriate open cover of $X$, and this is where the next proposition comes
into play:
 
\begin{proposition}\label{blowup.eomyalg}
Take a point $x \in X$ and a sufficiently small neighborhood $U \subset X$ of $x$ such that the ideal 
$J(U)\subset \Oo_X(U)$ is finitely generated. Then choosing a system of generators $J=\left< h_1,\ldots,h_s \right>$
gives an embedding $E_YX \subset X \times \P^{s-1}$ and an embedding $\tilde{E}_YX \subset 
X \times \P^{s-1}$. Their images are equal.
\end{proposition}

\begin{proof}$\;$\newline
Let $Y\subset X$ be the subspace defined by $J$, which in the following will mean $Y\cap U \subset U \subset X$ 
to avoid complicated and unnecessary notation, but always keeping in mind that we are 
working in a special open set $U$ of $X$ which allows us to use the finiteness properties of analytic geometry. 
Now consider the map:
\begin{align*}
 \ \lambda : X\setminus Y &\longrightarrow \P^{s-1}\\
                 x \ \ \ \ \ \   &\longmapsto (h_1(x):\cdots:h_s(x)),
\end{align*}
and as before let $E_YX \subset X\times \P^{s-1}$ be the closure of the graph of $\lambda$. 

	On the other hand, consider the presentation $\Oo_X[z_1,\ldots,z_s]/(g_1,\ldots,g_m)\cong P(J)$,
where the isomorphism is defined by $z_i\mapsto h_i$.\par\noindent Note that the $g_i\in \Oo_X[z_1,\ldots,z_s], 
\ i=1,\ldots,m$, generate the ideal of all homogeneous relations $g(h_1,\ldots,h_s)=0, \; g\in 
\Oo_X[z_1,\ldots,z_s]$. Those are exactly the equations for the closure of the graph. To see why this
last statement is true, recall that:
\[{\rm graph}(\lambda)=\{(x,z_1:\cdots:z_s)\subset X \times \P^{s-1}| \, (z_1:\cdots:z_s)=(h_1(x):\cdots:
h_s(x))\}.\]
and remember that the elements $g \in \left<g_1,\ldots,g_m\right>$ are homogeneous polynomials in 
$z_1,\ldots,z_s$ with coefficients in $\Oo_X$, so they define analytic functions in $X\times \C^s$ 
such that the homogeneity in the $z$'s allow us to look at their zeros in $X\times \P^{s-1}$. Moreover, 
if  $(x,z_1:\cdots:z_s)\in {\rm graph}(\lambda)$, then $[z]=[h(x)]$ and thus $g(z)=0$. Since the $g_i$ generate the ideal of elements of $\Oo_X[z_1,\ldots,z_s]$ such that $g(h_1,\ldots,h_s)=0$ they are the 
equations defining the graph, and consequently its closure.
\end{proof}
 
Finally, to relate all this to the normal cone, note that in the map:
\[e_Y:E_YX \longrightarrow X\]
the inverse image of $Y$ is the projective family associated to the family of cones
\[C_{X,Y}\longrightarrow Y.\]
This is clear, set-theoretically, in the geometric description. In the algebraic description, it 
follows from the identity:
\[\big(\bigoplus_{i\geq0}J^i\big)\bigotimes_{\Oo_X}\Oo_X/J \cong \bigoplus_{i\geq0}J^i/J^{i+1}={\rm gr}_J\Oo_X\]
and the fact that fiber product corresponds germ-wise to tensor product.\\

\[\xymatrix{\Oo_X\ar@{^{(}->}[r]\ar[d] & \bigoplus_{i\geq0}J^i \\
            \Oo_X/J &} \;\;\;\;\;\;\;\;\;\;\;\;\;\;\;\;\;\;
  \xymatrix{ X & E_YX \ar[l]_{e_Y}\\
            Y\ar@{^{(}->}[u] & }\] 

  The real trick comes when, in the analytic setting, we want to build the specialization to the normal cone 
in a global scenario. We will describe a geometric construction for this. Consider the complex space
$X \times \C$, and the closed nonsingular complex subspace $Y \times \{0\} \subset X \times \C$ defined 
by the coherent sheaf of ideals $\left<J,v\right>$. \\

   Let $\pi\colon Z \to X \times \C$ denote the blowing up of $X \times \C$ along $Y \times \{0\}$. Since $v$ is one of the generators
of the ideal $J^e=\left<J,v\right>$ defining the blown-up subspace, there is an open set $U \subset Z$ where $v$ generates the
pullback of the ideal $J^e \subset \Oo_{X \times \C}$, which is the ideal defining the exceptional 
divisor of the map $\pi$. One can verify by a direct computation that our old acquaintance, the sheaf of $\Oo_X$-algebras $\mathcal{R}$, can be 
identified with the sheaf of analytic functions, algebraic in $v$, over $U$. Moreover, the composed analytic map:
\[Z\supset U \stackrel{\pi\vert U}\longrightarrow  X \times \C \stackrel{{\rm pr}_2}\longrightarrow \C\]
is precisely the map which gives us the \textbf{specialization to the normal cone}.\par
 In the next  section we prove another specialization result which is very useful to prove theorem \ref{limtan1}
 and its generalization in the section of relative duality. It relates the specialization of $(X,0)$ to  its tangent cone to the specialization of $T^*_X\C^n$ to the normal cone of the fiber $\kappa^{-1}(x)$ of $x$ in $T^*_X\C^n$.
\subsection{Specialization to the normal cone of  $\kappa^{-1}(x)\subset C(X)$}
  \begin{proposition}\label{Specializationprop}
	    Let $X \subset \C^n$ be a reduced analytic subspace of dimension $d$ and for $x \in X$, let 
	    $\varphi:\mathfrak{X} \to \C$ be the specialization
	    of $X$ to the tangent cone $C_{X,x}$. Let $\kappa=\kappa_X\colon T^*_X\C^n\rightarrow X$ denote the conormal space of $X$ 
      in $\C^n \times \check{\C}^n$. Then the relative conormal space
	    \[ q: T^*_\mathfrak{X}(\C^n \times \C/\C )\to \mathfrak{X} \to \C\]
	    is isomorphic to the specialization space of $T^*_X\C^n$ to the normal cone $C_{T^*_X\C^n, \kappa^{-1}(x)}$
	    of $\kappa^{-1}(x)$ in $T^*_X\C^n$. In particular, the fibre $q^{-1}(0)$ is isomorphic
     	to this normal cone. 	
     \end{proposition}    
\begin{proof} We shall see a proof in a more general situation below in subsection \ref{reldu}.

\end{proof}

\begin{corollary}\label{SpecLagRel}
   The relative conormal space $\kappa_\varphi: T^*_\mathfrak{X} (\C^ n\times \C)/ \C) \to \mathfrak{X}$ is $\varphi-$Lagrangian.
\end{corollary}
\begin{proof}
   We will use the notation of the proof of Proposition \ref{Specializationproprelative}. From definition 
 \ref{RELA} we need to prove that every fiber $q^{-1}(s)$ is a Lagrangian subvariety
 of $\{s\} \times \C^n \times \check{\C}^n$. By Proposition \ref{Specializationprop}
 we know that for $s \neq 0$, the fiber $q^{-1}(s)$ is isomorphic to $T^*_X\C^n$ and so it is Lagrangian.
 Thus, by Proposition \ref{SpecLag} all we need to prove is that the special fiber $q^{-1}(0)$ has 
 the right dimension, which in this case is equal to $n$. \\
Proposition \ref{Specializationprop} also tells us that the fiber $q^{-1}(0)$ is isomorphic
to the normal cone 
\[C_{T^*_X\C^n, T^*_{\{x\}}\C^n \cap T^*_X\C^n}=C_{T^*_X\C^n, \kappa^{-1}(x)}.\] 
Finally, since the projectivized normal cone $\P C_{T^*_X\C^n, T^*_{\{x\}}\C^n \cap T^*_X\C^n}$ is 
obtained as the exceptional divisor of the blowing up of $T^*_X\C^n$ along $\kappa_X^{-1}(x)$, 
it has dimension $n-1$ and so the cone over this projective variety has dimension $n$, which finishes the proof.
\end{proof}
\subsection{Local Polar Varieties}\label{LOCPOL}

  In this section we introduce the local polar varieties of a germ of a reduced equidimensional complex analytic space $(X,0)\subset (\C^n,0)$. The dimension of $X$ is generally denoted by $d$ but to make the comparison with the case of projective varieties $V$ of dimension $d$ mentioned in the introduction we must think of $X$ as the cone  with vertex $0\in\C^n$ over $V$, which is of dimension $d+1$.\par Local polar varieties were first constructed, using the Semple-Nash modification and special Schubert cycles of the Grassmannian, in \cite{L-T1}. The description of the local polar varieties in terms of the conormal space used here is a contribution of Henry-Merle which appears in \cite{HM},  \cite{HMS} and \cite{Te3}. In this subsection we shall get a first glimpse into a special case of what will be called the 
normal-conormal diagram. Let us denote by $C(X)$ and $E_0X$ respectively the conormal space of $X$
and the blowing up of 0 in $X$ as before, then we have the diagram:
\[\xymatrix{E_0C(X)\ar[r]^{\hat{e}_0}\ar[dd]^{\kappa'}\ar[ddr]^\xi & C(X)\ar@{^{(}->}[r]\ar[dd]^\kappa
            \ar[dr]^{\lambda}& X \times \check{\P}^{n-1}\ar[d]^{pr_2} \\
             & &  \check{\P}^{n-1} \\
             E_0X\ar[r]_{e_0}  & X &}\]
where $E_0C(X)$ is the blowing up of the subspace $\kappa^{-1}(0)$ in $C(X)$, and $\kappa'$ is obtained from the universal
property of the blowing up, with respect to $E_0X$ and the map $\xi$.\par\noindent It is worth mentioning
that $E_0C(X)$ lives inside the fiber product \linebreak $C(X) \times_X E_0X$ and can be described in the following 
way: take the inverse image of $E_0X\setminus e_0^{-1}(0)$ in $C(X) \times_X E_0X$ and close it, thus obtaining
$\kappa'$ as the restriction of the second projection to this space.\\

Let $D_{d-k+1}\subset \C^n$ be a linear subspace of codimension $d-k+1$, for $0\leq k \leq d-1$,
and let $L^{d-k}\subset \check{\P}^{n-1}$ the dual space of $D_{d-k+1}$, which is the linear subspace of $ \check{\P}^{n-1}$ consisting of hyperplanes of $\C^n$ that contain $D_{d-k+1}$.\par\medskip The next proposition provides the relation between the intuitive definition of local polar varieties as closures of sets of critical points on $X^0$ of linear projections and the conormal definition, which is useful for proofs.

\begin{proposition}\label{polarv}
For a sufficiently general $D_{d-k+1}$, the image $\kappa(\lambda^{-1}(L^{d-k}))$ is the closure in 
$X$ of the set of points of $X^0$ which are critical for the projection $\pi|_{X^0}:X^0 \to \C^{d-k+1}$ induced
by the projection $\C^n \to \C^{d-k+1}$ with kernel $D_{d-k+1}=(L^{d-k})^{\check{}}$.
\end{proposition}
\begin{proof}$\;$ \newline
Note that $x \in X^0$ is critical for $\pi$ if and only if
the tangent map  $d_x\pi:T_{X,x}^0\longrightarrow \C^{d-k+1}$ is not onto, which means ${\rm dimker}  d_x \pi \geq k$ 
since 
$\mathrm{dim}T_{X,x}^0=d$, and ${\rm ker} d_x{\pi}=D_{d-k+1}\cap T_{X,x}^0$.\\
Note that the conormal space $C(X^0)$ of the nonsingular part of $X$ is equal to $\kappa^{-1}(X^0)$
so by definition:
\[\lambda^{-1}(L^{d-k})\cap C(X^0) =\{(x,H) \in C(X) |x\in X^0,\; H \in L^{d-k}, \; T_{X,x}^0 \subset H \}\]
equivalently:
\[\lambda^{-1}(L^{d-k}) \cap C(X^0) =\{(x,H), \in C(X)| x\in X^0,\  H \in \check{D},\; H \in (T_{X,x}^0)^{ \check{}}\ \} \]
thus $H \in \check{D}\cap (T_{X,x}^0)^{ \check{}}$, and from the equality $\check{D}\cap (T_{X,x}^0)^{ \check{}}=
(D+T_{X,x}^0)^{\check{}}$ we deduce that the intersection is not empty if and only if $D+T_{X,x}^0\neq \C^n$, which
implies that $\mathrm{dim} D\cap T_{X,x}^0\geq k$, and consequently $\kappa(H)=x$ is a critical point.

	According to \cite[Chapter IV, 1.3]{Te3}, there exists an open dense set $U_k$ in the grasmannian of 
 $n-d+k-1$-planes of $\C^n$ such that if $D \in U_k$, the intersection $\lambda^{-1}(L^{d-k})\cap C(X^0)$
 is dense in $\lambda^{-1}(L^{d-k})$. So, for any $D \in U$, since $\kappa$ is a proper map and thus closed,
we have that $\kappa(\lambda^{-1}(L^{d-k}))= \kappa\left(\overline{\lambda^{-1}(L^{d-k})\cap C(X^0)}\right)=
\overline{ \kappa(\lambda^{-1}(L^{d-k}))}$, which finishes the proof.
\noindent
See \cite[Chap. 4, 4.1.1]{Te3} for a complete proof of a more general statement.
\end{proof}
 \begin{remark}\label{Factsofpolarvarieties} It is important to have in mind the following easily verifiable facts:
  \begin{enumerate}
 \item [a)] As we have seen before, the fiber $\kappa^{-1}(x)$ over a regular point $x \in X^0$ in the 
       (projectivized) conormal space $C(X)$ is a $\P^{n-d-1}$, so by semicontinuity of fiber dimension 
       we have that $\mathrm{dim}\kappa^{-1}(0) \geq n-d-1$.
 \item[b)] The analytic set $\lambda^{-1}(L^{d-k})$ is nothing but the intersection of $C(X)$ 
       and $\C^n \times L^{d-k}$ in $\C^n \times \check{\P}^{n-1}$. The space $\C^n \times L^{d-k}$ is 
       ``linear", defined by $n-d+k-1$ linear equations. For a general $L^{d-k}$, this intersection is of pure dimension $n-1-n+d-k+1=d-k$ if it is not empty. \par
       The proof of this is not immediate because we are working over an open neighborhood of a point $x\in X$, so we cannot assume that $C(X)$ is compact. However (see \cite[Chap. IV]{Te3}) we can take a Whitney stratification of $C(X)$ such that the closed algebraic subset $\kappa^{-1}(0)\subset \check \P^{n-1}$, which is compact. is a union of strata. By general transversality theorems in algebraic geometry (see \cite{Kl1}) a sufficiently general $L^{d-k}$ will be transversal to all the strata of $\kappa^{-1}(0)$ in $\check \P^{n-1}$ and then because of the Whitney conditions $\C^n \times L^{d-k}$ will be transversal in a neighborhood of $\kappa^{-1}(0)$ to all the strata of $C(X)$, which will imply in particular the statement on the dimension. Since $\kappa$ is proper, the neighborhood of $\kappa^{-1}(0)$ can be taken to be the inverse image by $\kappa$ of a neighborhood of $0$ in $X$. The meaning of "general" in Proposition \ref{polarv} is that of Kleiman's transversality theorem.
Moreover, since $C(X)$ is a reduced equidimensional analytic space, for a general $L^{d-k}$, the intersection of $C(X)$ 
and $\C^n \times L^{d-k}$ in $\C^n \times \check{\P}^{n-1}$ is generically reduced and since according to our general rule we remove embedded components when intersecting with linear spaces, $\lambda^{-1}(L^{d-k})$ is a reduced equidimensional complex analytic space.\par\noindent Note that the {\rm existence} of Whitney stratifications does not depend on the existence of polar varieties. In \cite[Chap. III, Proposition 2.2.2]{Te3} it is deduced from the idealistic Bertini theorem. \item[c)]  The fact that $\lambda^{-1}(L^{d-k})\cap C(X^0)$
 is dense in $\lambda^{-1}(L^{d-k})$ means that if a limit of tangent hyperplanes at points of $X^0$ contains $D_{d-k+1}$, it is a limit of tangent hyperplanes which also contain $D_{d-k+1}$. This equality holds because transversal intersections preserve the frontier condition; see \cite{Ch}, \cite[Remarque 4.2.3]{Te3}.
      \item[d)] Note that for a fixed $L^{d-k}$, the germ $(P_k(X;L^{d-k}),0)$ is empty if and only if the
      intersection $\kappa^{-1}(0) \cap \lambda^{-1}(L^{d-k}) $ is empty. From a) we know that 
      $\mathrm{dim}\kappa^{-1}(0)=n-d-1+r$ with $r \geq 0$. Thus, by the same
      argument as in b), this implies that the polar variety $(P_k(X;L^{d-k}),0)$ is not empty if and only if $\mathrm {dim}(\kappa^{-1}(0) \cap
      \lambda^{-1}(L^{d-k})) \geq 0$ and if and only if  $r\geq k$.
\end{enumerate}
\end{remark}
\begin{definition}\label{defpolarv}
   With the notations and hypotheses of Proposition \ref{polarv}, define for $0\leq k\leq d-1$
   the \textbf{local polar variety}.
\[P_k(X;L^{d-k})= \kappa(\lambda^{-1}(L^{d-k})) \]
\end{definition}

  A priori, we have just defined $\polar$ set-theoretically, but since $\lambda^{-1}(L^{d-k}) $ is empty or reduced and $\kappa$ is a projective fibration over the smooth part of $X$ we have
 the following result, for which a proof can be found in \cite[Chapter IV, 1.3.2]{Te3}.
 
 \begin{figure}[!ht]
    \begin{center}
              \includegraphics*{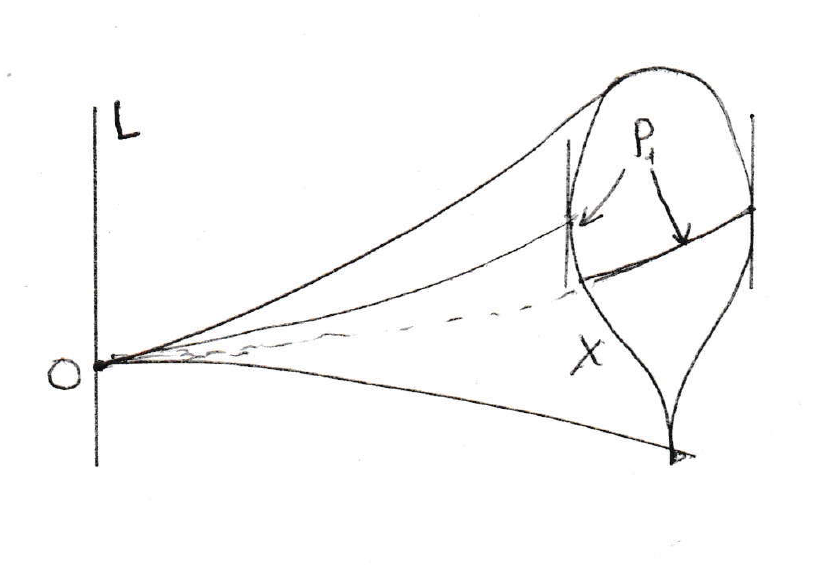}
       \end{center}
\end{figure}

 \begin{proposition}
   The local polar variety $\polar \subseteq X $ is a reduced closed analytic subspace of $X$, either of pure codimension
   $k$ in $X$ or empty.   
 \end{proposition}

      We have thus far defined a local polar variety that depends on both the choice
  of the embedding $(X,0) \subset (\C^n,0)$ and the choice of the linear space $D_{d-k+1}$. However,
  an important information we will extract from these polar varieties is their multiplicities at $0$,
  and these numbers are analytic invariants provided the linear spaces used to define them are general enough. This generalizes the invariance of the degrees of Todd's polar loci which we saw in the introduction.

  \begin{proposition}{{\rm (Teissier)}}\label{multpolv}
    Let $(X,0) \subset (\C^n,0)$ be as before, then for every $0\leq k \leq d-1$ and a sufficiently
    general linear space $D_{d-k+1} \subset \C^n$ the multiplicity of the polar variety $\polar$ at
    $0$ depends only on the analytic type of $(X,0)$.
  \end{proposition}
  
  \begin{proof}$\;$\\
    See \cite[Chapter IV, Th\'eor\`eme 3.1]{Te3}. The idea is to construct for a given local embedding $X\subset\C^n$ a map $\pi\colon Z\to G$ where $G$ is the space of linear projections $\C^n\to \C^{d-k+1}$ such that for general $g\in G$ the fiber is the corresponding polar variety, and then to use the analytic semicontinuity of multiplicity. Given two different embeddings, one puts them in a common third embedding and uses a similar method.   \end{proof}
  
   	This last result allows us to associate to any reduced, pure $d$-dimensional, analytic local algebra
   $O_{X,x}$ a sequence of $d$ integers $(m_0,\ldots,m_{d-1})$, where $m_k$ is the multiplicity at $x$ of the 
   polar variety $\polar$ calculated from any given embedding $(X,x) \subset (\C^n,0)$, and
   a general choice of $D_{d-k+1}$. Note that in practice such a choice is not always easy to determine.

 \begin{remark}\label{generalflag}
      Since for a linear space $L^{d-k}$ to be ``sufficiently general'' means
  that it belongs to an open dense subset specified by certain conditions, we can just as well take
  a sufficiently general flag
 \[ L^1 \subset L^2 \subset \cdots \subset L^{d-2} \subset L^{d-1} \subset L^d\subset \check{\P}^{n-1}\]
  which by definition of a polar variety and Proposition \ref{multpolv}, gives us a chain
 \[ P_{d-1}(X;L^1) \subset P_{d-2}(X;L^2) \subset \cdots \subset P_{1}(X;L^{d-1}) \subset P_{0}(X;L^d)=X, \]  
  of polar varieties, each with generic multiplicity at the origin. This implies that if the 
 germ of a general polar variety $(\polar,0)$ is empty for a fixed $k$, then it will be empty for all
 $l\in\{k,\dots ,d-1\}$. This fact can also be deduced from \ref{Factsofpolarvarieties} d) by counting
 dimensions.
\end{remark}
\begin{definition}\label{projpol}{\rm (Definition of polar varieties for singular projective varieties)}  Let $V\subset\P^{n-1}$ be a reduced equidimensional projective variety of dimension $d$. Let $(X,0)\subset (\C^n,0)$ be the germ at $0$ of the cone over $V$. The polar varieties $P_k(X,L^{d-k+1}),\ 0\leq k\leq d$ are cones because tangent spaces are constant along the generating lines (see Lemma \ref{cone}). The associated projective subvarieties of $V$ are the polar varieties of $V$ and are denoted by $P_k(V)$ or $P_k(V,L^{d-k+1})$ or $P_k(V, D_{d-k+2})$ with $L^{d-k+1}=(D_{d-k+2})^{\check{}}\subset\check\P^{n-1}$.
\end{definition}
\noindent If $V$ is nonsingular this definition coincides with the definition of $P_k(V, D_{d-k+2})$ given in the introduction. It suffices to take the linear subspace $L^{d-k+1}\subset \check\P^{n-1}$ to be the dual of the subspace $D_{d-k+2}\subset \P^{n-1}$ of codimension $d-k+2$ which appears in that definition. 
\begin{ex}\label{sur} $\;$\\ Let $X:= y^2 - x^3 -t^2x^2=0 \subset \C^3$, so $\mathrm{dim}X=2$, and thus $k=0,1$. An easy calculation
 shows that the singular locus of $X$ is the $t-$axis, and $m_0(X)=2$. \\

 \begin{figure}[!ht]
    \begin{center}
    \label{schematic1}
              \includegraphics*{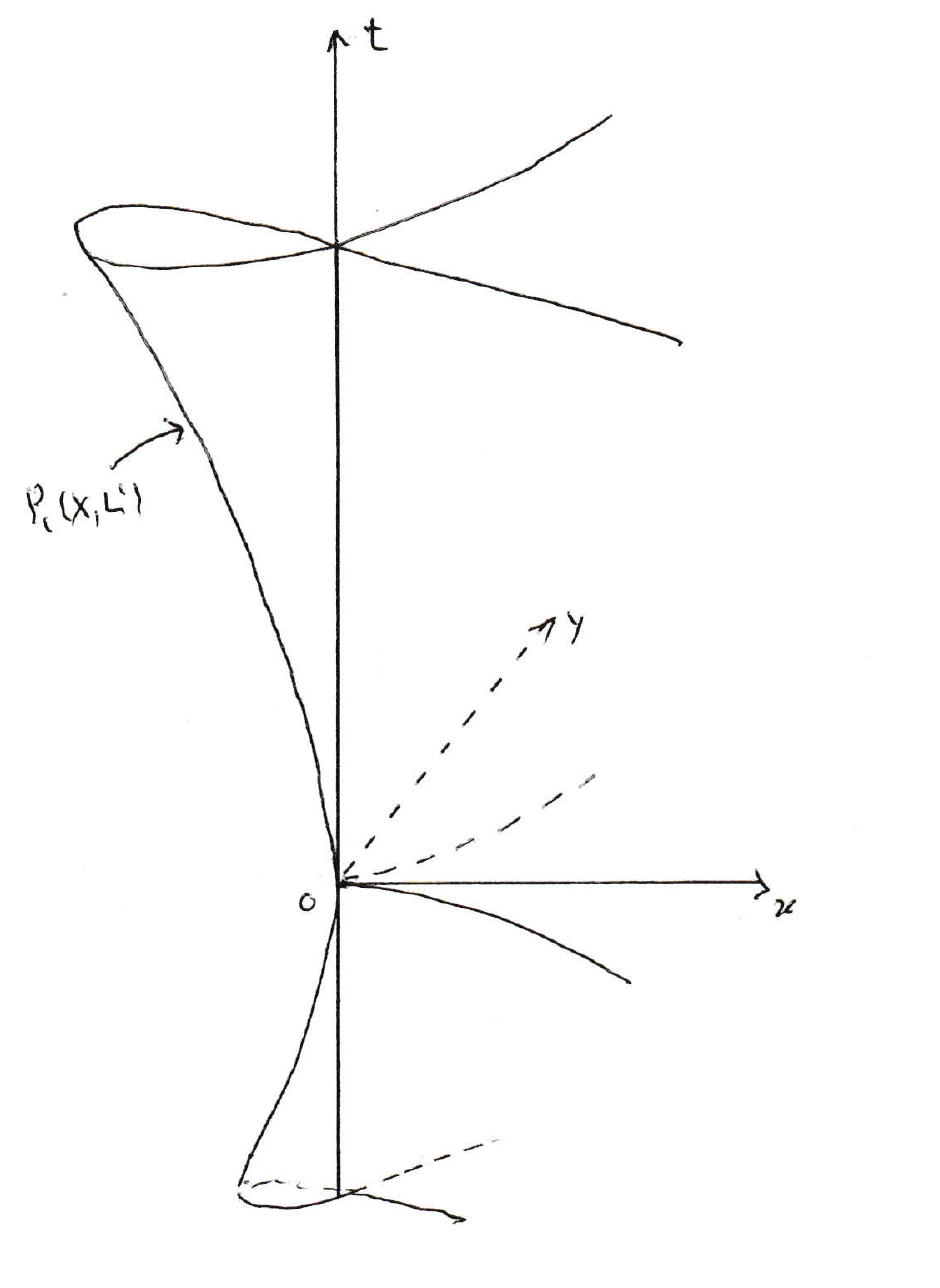}
    % \caption{ Sing $X=$ eje z; $C_{X,0}=$ plano yz}
    \end{center}
\end{figure}

 Note that for $k=0$, $D_3$ is just the origin in $\C^3$, so the the projection  
 \[ \pi:X^0 \to \C^3\] 
 with kernel $D_3$ is the restriction to $X^0$ of the identity map, which is of rank $2$ and we get that
 the whole $X^0$ is the critical set of such a map. Thus, \[P_0(X,L^2)=X.\]
 
 For $k=1$, $D_2$ is of dimension 1. So let us take for instance $D_2=y-$axis, so we get the projection
 \[\pi:X^0 \to \C^2 \;\;\; (x,y,t) \mapsto (x,t),\] 
 and we obtain that the set of critical points of the projection is given by
 \[P_1(X,L^1)= \left\{ \begin{array}{l} x= -t^2 \\ y=0 \end{array} \right. \]
 If we had taken for $D_2$ the line $t=0,\  \alpha x+\beta y=0$, we would have found that the polar curve is a nonsingular component of the intersection of our surface with the surface $2\alpha y=\beta x(3x+2t^2)$. For $\alpha\neq 0$ all these polar curves are tangent to the $t$-axis. As we shall see in the next subsection, this means that the $t$-axis is an ``exceptional cone" in the tangent cone $y^2=0$ of our surface at the origin, and therefore all the $2$-planes containing it are limits at the origin of tangent planes at nonsingular points of our surface.
 \end{ex}
\subsection{Limits of tangent spaces}
      
      Weith the help of the 
normal/conormal diagram and the polar varieties we will be able to obtain information on the limits of tangent spaces
to $X$ at 0, assuming that $(X,0)$ is reduced and purely $d$-dimensional. This method is based 
on Whitney's lemma and the two results which follow it:

\begin{lemma}\textbf{Whitney's lemma.-}\label{wl}
 Let $(X,0)$ be a pure-dimensional germ of analytic subspace of $\C^n$, choose a representative $X$ and let 
 $\{x_n\}\subset X^0$ be a sequence of points tending to $0$, such that 
 \[\lim_{n\to \infty}[0x_n]=l \;\; and \;\;
 \lim_{n\to \infty}T_{x_n}X=T,\] then $l \subset T$.
\end{lemma} 
      
     This lemma originally appeared in \cite[Theorem 22.1]{Whi1}, and you can also find a proof due to
Hironaka in \cite{L1} and yet another below in assertion a) of theorem \ref{limtan1}.\\

\begin{theorem}{\rm (L\^e-Teissier, see \cite{L-T2})}\label{limtan1}$\;$\\

\textbf{I)} In the normal/conormal diagram
\[\xymatrix{&X \times\P^{n-1}\times \check{\P}^{n-1}&\supset &E_0C(X)\ar[r]^{\hat{e}_0}\ar[dd]^{\kappa'}\ar[ddr]^\xi & C(X)\ar@{^{(}->}[r]\ar[dd]^\kappa
            \ar[dr]^{\lambda}& X \times \check{\P}^{n-1}\ar[d]^{pr_2} \\
            &&& & &  \check{\P}^{n-1} \\
            &X \times\P^{n-1}&\supset  &E_0X\ar[r]_{e_0}  & X &}\]
consider the irreducible components $\{D_{\alpha}\}$ of $D=|\xi^{-1}(0)|$. Then:
\begin{itemize}
\item[a)] Each $D_{\alpha}\subset \P^{n-1} \times \check{\P}^{n-1}$ is in fact contained
          in the incidence variety $I\subset \P^{n-1} \times \check{\P}^{n-1}$.\\
    %    \textnormal{  Note that, using the interpretation of $E_0C(X)$ as the fibered product $C(X)\times_X E_0X$,
    %      what we are stating here is immediate from Whitney's Lemma \ref{wl} since a point $(0,l,H) \in \xi^{-1}(0)$
    %      means that both $l$ and $H$ are limits obtained from the same sequence $\{x_n\}\subset X^0$, and 
    %      as such this implies $l\subset H$.}

\item[b)] Each $D_\alpha$ is Lagrangian in $I$ and therefore establishes a projective duality 
          of its images:
\[\xymatrix{ D_\alpha \ar[rr]\ar[dd]& &  W_\alpha \subset \check{\P}^{n-1}\\
                                    & & \\
             V_\alpha \subset \P^{n-1} & & }\]
\end{itemize}

     \textnormal{Note that, from commutativity of the diagram we obtain  $\kappa^{-1}(0)=\bigcup_\alpha W_{\alpha}$, and 
                $e_0^{-1}(0)=\bigcup_\alpha V_{\alpha}$. It is important to notice that these expressions are not necessarily the 
                irreducible decompositions of $\kappa^{-1}(0)$ and $e_0^{-1}(0)$ respectively, since there may be repetitions; it is the case for the surface of example \ref{sur}, where the dual of the tangent cone, a point in $\check\P^2$, is contained in the projective line dual to the exceptional tangent. However, it is true that they contain
                the respective irreducible decompositions.} \\

                 \textnormal{In particular, note that if dim$V_{\alpha_0}=d-1$, then the cone $O(V_{\alpha_0})\subset \C^n$ is 
                 an irreducible component of the tangent cone $C_{X,0}$ and its projective dual $W_{\alpha_0}=
                 \check V_{\alpha_0}$ is contained in $\kappa^{-1}(0)$. That is, any tangent hyperplane to the tangent cone is 
                 a limit of tangent hyperplanes to $X$ at $0$. The converse is very far from true and we shall see more about this below.} \\

\textbf{II)} For any integer $k$, $0\leq k \leq d-1$ and sufficiently general $L^{d-k}\subset \check{\P}^{n-1}$
             the tangent cone $C_{P_k(X,L),0}$ of the polar variety $P_k(X,L)$ at the origin consists of:
\begin{itemize}
\item[$\bullet$]The union of the cones $O(V_\alpha)$ which are of dimension $d-k$ (= $\mathrm {dim}P_k(X,L)$).
\item[$\bullet$]The polar varieties $P_j(O(V_\beta),L)$ of dimension $d-k$, for the projection $p$ associated to $L$, of the
              cones $O(V_\beta)$, for $\mathrm{dim}O(V_\beta)= d-k+j$.
\end{itemize}

\noindent\textnormal{Note that $P_k(X,L)$ is not unique, since it varies with $L$, but we are saying that their tangent cones have things 
            in common. The $V_\alpha$'s are fixed, so the first part is the fixed part of $C_{P_k(X,L),0}$ because it is independent 
            of $L$, the second part is the mobile part, since we are talking of polar varieties of certain cones, which by definition move with $L$.}\\
\end{theorem}
\begin{proof}$\;$\\
The proof of \textbf{I)}, which can be found in \cite{L-T1}, is essentially a strengthening of Whitney's lemma (Lemma \ref{wl}) using the normal/conormal diagram and the fact that the vanishing of a differential form (the symplectic form in our case) is a closed condition.\par\noindent
The proof of \textbf{II)}, also found in \cite{L-T1}, is somewhat easier to explain geometrically:\par\noindent
     Using our normal/conormal diagram, remember that we can obtain the blowing up $E_0(P_k(X,L))$ of the polar variety
  $P_k(X,L)$ by taking its strict transform under the morphism $e_0$, and as such we will get the projectivized
  tangent cone $\P C_{P_k(X,L),0}$ as the fiber over the origin.

  The first step is to prove that set-theoretically the projectivized tangent cone can also be expressed as 
  \[|\P C_{P_k(X,L),0}|= \bigcup_\alpha \kappa'( \hat{e}_0^{-1} ( \lambda^{-1}(L) \cap W_\alpha))
                       = \bigcup_\alpha \kappa' (D_\alpha \cap (\P^{n-1} \times L))\]
 
  	Now recall that the intersection $P_k(X,L) \cap X^0$ is dense in $P_k(X,L)$, so for any point $(0,[l]) \in 
 \P C_{P_k(X,L),0} $ there exists a sequence of points $\{x_n\} \subset X^0$ such that the directions of the secants $\overline{0x_n}$ converge to it. So, by definition
 of a polar variety, if $D_{d-k+1}=\check{L}$ and $T_n=T_{x_n}X^0$ then by \ref{polarv} we know that $\mathrm{dim}T_n \cap 
 D_{d-k+1}\geq k$ which is a closed condition. In particular if $T$ is a limit of tangent spaces obtained
 from the sequence $\{T_n\}$, then $T \cap D_{d-k+1}\geq k$ also. But if this is the case, since the 
 dimension of $T$ is $d$, there exists a limit of tangent hyperplanes $H \in \kappa^{-1}(0)$ such that 
  $T + D_{d-k+1} \subset H $ which is equivalent to $H \in \kappa^{-1}(0)\cap \lambda^{-1}(L) \neq \emptyset$.
 Therefore the point $(0,[l],H)$ is in $\bigcup_\alpha \hat{e}_0^{-1} ( \lambda^{-1}(L) \cap W_\alpha)$,
 and so we have the inclusion:
\[|\P C_{P_k(X,L),0}| \subset \bigcup_\alpha \kappa'( \hat{e}_0^{-1} ( \lambda^{-1}(L) \cap W_\alpha))\]

 	For the other inclusion, recall that $\lambda^{-1}(L)\setminus \kappa^{-1}(0)$ is dense in $\lambda^{-1}(L)$
 and so $\hat{e}_0^{-1} ( \lambda^{-1}(L))$ is equal set theoretically to the closure in $E_0C(X)$ of 
 $\hat{e}_0^{-1} ( \lambda^{-1}(L) \setminus \kappa^{-1}(0))$. Then for any point $(0,[l],H) \in 
 \hat{e}_0^{-1} ( \lambda^{-1}(L) \cap \kappa^{-1}(0))$ there exists a sequence $\{(x_n,[x_n],H_n)\}$ 
 in $\hat{e}_0^{-1} ( \lambda^{-1}(L) \setminus \kappa^{-1}(0))$ converging to it. Now by commutativity of 
 the diagram, we get that the sequence $\{(x_n,H_n)\}\subset \lambda^{-1}(L)$ and as such the sequence of 
 points $\{x_n\}$ lies in the polar variety $P_k(X,L)$. This implies in particular, that the sequence 
 $\{(x_n,[0x_n])\}$ is contained in $e_0^{-1}(P_k(X,L)\setminus \{0\})$ and the point $(0,[l])$ is in the 
 projectivized tangent cone $|\P C_{P_k(X,L),0}|$. \\

 	The second and final step of the proof is to use that from \textit{a)} and \textit{b)} it follows that
 each $D_\alpha \subset I \subset \P^{n-1} \times \check{\P}^{n-1}$ is the conormal space of $V_\alpha$ in $\P^{n-1}$, 
 with the restriction of $\kappa'$ to $D_\alpha$ being its conormal morphism.\\

  	Note that $D_\alpha$ is of dimension $n-2$, and since all the maps involved are just projections, we can 
 take the cones over the $V_\alpha$'s and proceed as in section \ref{du}. In this setting we get that
 since $L$ is sufficiently general, by Proposition \ref{polarv} and definition \ref{defpolarv}: 

\begin{itemize}
\item[$\bullet$] For the $D_\alpha$'s corresponding to cones $O(V_\alpha)$ of dimension $d-k$ (= $\mathrm{dim} P_k(X,L)$), 
        the intersection $D_\alpha \cap \C^n \times L$ is not empty and as such its image is a polar variety
        $P_0(O(V_\alpha),L)=O(V_\alpha)$.
\item[$\bullet$] For the $D_\alpha$'s corresponding to cones $O(V_\alpha)$ of dimension $d-k+j$, the intersection 
        $D_\alpha \cap (\C^n \times L)$ is either empty or of dimension $d-k$ and as such its image is a polar variety
        of dimension $d-k$, which is $P_j(O(V_\alpha),L)$.
\end{itemize}

You can find a proof of these results in \cite{L-T1}, \cite[Chap. IV]{Te3} and \cite{Te4}.
\end{proof}
       So for any reduced and purely $d-$dimensional complex analytic germ $(X,0)$, we have a method to ``compute'' or rather
       describe, the set of limiting positions of tangent hyperplanes. Between parenthesis are the types of computations involved:
\begin{itemize}
\item[1)] For all integers $k$, $0\leq k \leq d-1$, compute the ``general'' polar varieties 
          $P_k(X,L)$, leaving in the computation the coefficients of the equations of L as indeterminates.
          (Partial derivatives, Jacobian minors and residual ideals with respect to the Jacobian ideal);  
\item[2)] Compute the tangent cones  $C_{P_k(X,L),0}$. (computation of a standard basis with parameters; see remark \ref{initial});
\item[3)] Sort out those irreducible components of the tangent cone of each $P_k(X,L)$ which are independent of L (decomposition into irreducible components with parameters);
\item[4)] Take the projective duals of the corresponding projective varieties  (Elimination).  
\end{itemize}

       We have noticed, that among the $V_\alpha$'s, there are those which are irreducible components of ${\rm Proj} C_{X,0}$
       and those that are of lower dimension.

\begin{definition}
       The cones $O(V_{\alpha})$'s such that 
       \[ {\rm dim}\; V_\alpha <{\rm dim}\ {\rm Proj} C_{X,0}\]
       are called exceptional cones. 
\end{definition}
 
\begin{remark}\label{excone} 1) We repeat the remark on p. 567 of \cite{L-T2} to the effect that when $(X,0)$ is analytically isomorphic to the germ at the vertex of a cone the polar varieties are themselves isomorphic to cones so that the families of tangent cones of polar varieties have no fixed components except when $k=0$. Therefore in this case $(X,0)$ has no exceptional cones. \par\noindent 2) The fact that the cone $X$ over a nonsingular projective variety has no exceptional cones is thus related to the fact that the critical locus $P_1(X,0)$ of the projection $\pi\colon X\to\C^d$, which is purely of codimension one in $X$ if it is not empty,  actually moves with the projection $\pi$; in the language of algebraic geometry, the {\rm ramification divisor} of the projection  is {\rm ample} (see \cite[Chap. I, cor. 2.14]{Zak}) and even {\rm very ample} (see \cite{Ein}).\par\noindent 3) 
The dimension of $\kappa^{-1}(0)$ can be large for a singularity $(X,0)$ which has no exceptional cones. This is the case for example if $X$ is the cone over a projective variety of dimension $d-1<n-2$ in $\P^{n-1}$ whose dual is a hypersurface. \end{remark}
      
       Now one may wonder whether having no exceptional tangents makes $X$ look like a cone. We will 
       give a partial answer to this question in section \ref{sec:specialization} in terms of the Whitney equisingularity along the axis of parameters
       of the flat family specializing $X$ to its tangent cone.\par\medskip
We are now going to discuss the relation
between the conormal space of $(X,0) \subset (\C^n,0)$ and its Semple-Nash modification, or rather between their fibers over a singular point. It is convenient here to use the notations of projective duality of linear spaces.\par Given a vector subspace $T\subset \C^n$ we denote by $\P T$ its projectivization, i.e., the image of $T\setminus \{0\}$ by the projection $\C^n \setminus \{0\}\to \P^{n-1}$ and by $\check T\subset \check\P^{n-1}$ the projective dual of $\P T\subset\P^{n-1}$, which is a $\P^{n-d-1}\subset \check\P^{n-1}$, the set of all hyperplanes $H$ of $\P^{n-1}$ containing $\P T$.\par  We denote by  $\check\Xi \subset G(d, n) \times \check\P^{n-1}$ the cotautological $\P^{n-d-1}$-bundle over $G(d,n)$, that is
$ \check\Xi = \{(T,H)\; | \; T \in G(d,n), \; H \in \check T \subset \check\P^{n-1} \}$, and consider the intersection 
\[ \xymatrix{ 
  E:=(X \times \check\Xi)\cap ( NX \times \check\P^{n-1} ) \ar @{^{(}->}[r] \ar[dr]^{p_2} \ar[d]_{p_1}&  
  X \times G(d,n)\times \check{\P}^{n-1} \ar[d]\\
     NX & X \times \check{\P}^{n-1}}\]
and the morphism $p_2$ induced on $E$ by the projection onto $ X \times \check{\P}^{n-1}$.
We then have the following:

\begin{proposition}\label{ConormalvsNash}
   Let $p_2:E \to X \times \check{\P}^{n-1}$ be as before. The set-theoretical image $p_2(E)$ of the morphism 
$p_2$ coincides with the conormal space of $X$ in $\C^n$ 
\[ C(X) \subset X \times \check{\P}^{n-1}.\]
\end{proposition}
\begin{proof}
  By definition, the conormal space of $X$ in $\C^n$ is an analytic space $C(X) \subset X \times 
  \check{\P}^{n-1}$, together with a proper analytic map $\kappa_X: C(X)\to X$, where the fiber over a smooth point 
  $x \in X^0$ is the set of tangent hyperplanes, that is the hyperplanes $H$ containing the direction of the 
  tangent space $T_{X,x}$. That is, if we define $E^0=\{(x,T_{X,x},H) \in E \, | \, x \in X^0,H\in\check T_{X,x}\}$, then by 
  construction $E^0= p_1^{-1}(\nu_X^{-1}(X^0))$, and $p_2(E^0)= C(X^0)$. Since the morphism $p_2$ is proper it is closed, which finishes the proof.
\end{proof}

\begin{corollary}\label{NashvsConormalCoro}
   A hyperplane $H \in \check{\P}^{n-1}$ is a limit of tangent hyperplanes to $X$ at $0$, i.e., 
$H \in \kappa_X^{-1}(0)$, if and only if there exists a $d$-plane $(0,T) \in \nu_X^{-1}(0)$ such that $T \subset H$.
\end{corollary}
\begin{proof}
   Let $(0,T) \in \nu_X^{-1}(0)$ be a limit of tangent spaces to $X$ at $0$. By construction of $E$ and proposition 
 \ref{ConormalvsNash}, every hyperplane $H$ containing $T$ is in the fiber $\kappa_X^{-1}(0)$, and so is
 a limit at $0$ of tangent hyperplanes to $X^0$.\\ 
	On the other hand, by construction, for any hyperplane $H \in \kappa_X^{-1}(0)$ there is a sequence of points
$\{(x_i,H_i)\}_{i \in \N}$ in $\kappa_X^{-1}(X^0)$ converging to $p=(0,H)$. Since the map $p_2$ is surjective, by definition of $E$, we have a sequence  $(x_i,T_i,H_i)\in E^0$ with $T_i=T_{x_i}X^0\subset H_i$. By compactness of Grassmannians and projective spaces, this sequence has to converge, up to taking a subsequence, to $(x,T,H)$ with $T$ a limit at $x$ of tangent spaces to $X$. Since inclusion is a closed condition, we have $T\subset H$.
\end{proof}
\begin{corollary}\label{Nashbundle} The morphism $p_1: E \to  NX$ is a locally analytically trivial fiber bundle with fiber $\P^{n-d-1}$.
\end{corollary}
\begin{proof}By definition of $E$, the fiber of the projection $p_1$ over a point $(x,T)\in NX$ is the set of all hyperplanes in $\P^{n-1}$ containing $\P T$. In fact, the tangent bundle $T_{X^0}$, lifted to $NX$ by the isomorphism $NX^0\simeq X^0$, extends to a fiber bundle over $NX$, called the Nash tangent bundle of $X$. It is the pull-back by $\gamma_X$ of the tautological bundle of $G(d,n)$, and $E$ is the total space of the  $\P^{n-d-1}$-bundle of the projective duals of the projectivized fibers of the Nash bundle.
\end{proof}
By definition of $E$, the map $p_2$ is an isomorphism over $C(X^0)$ since a tangent hyperplane at a nonsingular point contains only the tangent space at that point. Therefore the map $p_2\colon E\to C(X)$ is a modification.\par\noindent In general the fiber of $p_2$ over a point $(x,H)\in C(X)$ is the set of limit directions at $x$ of tangent spaces to $X$ that are contained in $H$. If $X$ is a hypersurface, the conormal map coincides with the Semple-Nash 
modification. In general, the manner in which the geometric 
structure of the inclusion
$\kappa_X^{-1}(x)\subset \check{\P}^{n-1}$ determines the set of limit 
positions of tangent spaces, i.e., the fiber $\nu_X^{-1}(x)$ of the 
Semple-Nash modification, is
not so simple: by Proposition \ref{ConormalvsNash} and its corollary, 
the points of $\nu_X^{-1}(x)$ correspond to \textit{some of} the projective subspaces 
${\P}^{n-d-1}$ of $\check{\P}^{n-1}$ contained in $\kappa_X^{-1}(x)$. The assertion made in the observation on page 553 of \cite{L-T3} is wrong, as the next example shows.
\begin{example}\label{jawad} We use for our purpose the following example due to J. Snoussi in \cite{Sn2} of a surface singularity $(X,0)$ where the tangent cone has no linear component and the normalized Semple-Nash modification is a singular surface. We warn the reader that what follows requires some knowledge of resolution of singularities of surfaces. We recommend \cite{Sp} and the very informative paper \cite{B} of R. Bondil. \par\noindent Snoussi's example is a germ of a normal surface $(X,0)\subset (\C^5,0)$. In its minimal resolution of singularities $\pi\colon W\to X$ the inverse image of $0$, the exceptional divisor, has five irreducible components $E_i$ and if one represents each by a point and connect those points by an edge when the corresponding components intersect, one obtains the following diagram, called the \textit{dual graph}.  \par\bigskip
\vspace{.1in}
\begin{picture}(2,1)\setlength{\unitlength}{.4cm}
\put(7,0){$\bullet$}\put(6.5,1){$-3$}\put(6.8,-1){$E_1$}\put(7,.2) {\line(1,0){3}}\put(10,0){$\bullet$}\put(9.5,1){$-2$}\put(9.8,-1){$E_2$}\put(10,.2) {\line(1,0){3}}\put(13,0){$\bullet$}\put(12.5,1){$-2$}\put(12.8,-1){$E_3$}\put(13,.2) {\line(1,0){3}}\put(16,0){$\bullet$}\put(15.5,1){$-2$}\put(15.8,-1){$E_4$}\put(16,.2) {\line(1,0){3}}\put(19,0){$\bullet$}\put(18.5,1){$-3$}\put(18.8,-1){$E_5$}
\end{picture}
\vspace{.2in}

\noindent Each component $E_i$ is in fact a projective line and the numbers above each point represent the self intersection in the surface $W$ of the corresponding component.
\noindent The pull-back in $W$ of the maximal ideal $m_{X,0}$ is in this case an invertible sheaf which defines the divisor $Z=\sum_{i=1}^5E_i$. From the theory of resolution of rational surface singularities, one knows that the blowing-up $e_0\colon X'\to X$ of the maximal ideal is obtained by contracting in $W$ the $E_i$ such that their intersection number in $W$ with the cycle $Z$ is zero, in this case $E_2,E_3,E_4$. This contraction gives a map $\pi'\colon W\to X'$ such that $\pi=e_0\circ \pi'$, while the images in $X'$ of the components $E_i$ such that $E_i.Z<0$, in this case $E_1,E_5$, give the components of the projectivized tangent cone $\P C_{X,0}\subset \P^4$. This projectivized tangent cone is the union of two conics meeting at one point (compare with \cite[8.2]{B}). This point, which we denote by $\ell$,  is also the intersection point in $\P^4$ of the projective planes $\P^2$ containing each conic. By work of Spivakovsky in \cite[\S 5, Theorem 5.4]{Sp}, we know that the strict transform of a general polar curve of $(X,0)$ intersects $E_1,E_3$ and $E_5$. So the tangents to the polar curves are not separated in $X'$ and in fact, the strict transforms in $X'$ of some irreducible component of the polar curves of $(X,0)$ all go through the point $\ell$ which is therefore, by theorem \ref{limtan1}, an exceptional tangent. This means that \textit{all the hyperplanes in $\C^5$ containing the line $\ell$ are limits of tangent hyperplanes} at nonsingular points of $X$.\par
Those hyperplanes in $\C^5$ which contain $\ell$ form a $\P^3$ in the space $\check\P^4$ of hyperplanes of $\C^5$. On the other hand, to each limit $T$ at $0$ of tangent planes to $X^0$ corresponds a $\P^2(T)\subset \check\P^4$ of tangent hyperplanes. If $\ell\subset T$, then $\P^2(T)\subset\P^3$. But the space of projective planes contained in $\P^3$ is $3$-dimensional, while the space of limits of tangent planes, which is the fiber of the Semple-Nash modification, has to be of dimension $\leq 1$ for a surface.\par
In order to understand which projective planes in $\P^3$ are actually of the form $\P^2(T)$, let us first consider the analogue for the Semple-Nash modification  $\nu_X\colon NX\to X$ of the normal/conormal diagram:
\[\xymatrix{E_0NX\ar[r]^{\tilde{e}_0}\ar[dd]^{\nu'_X}\ar[ddr]^\eta & NX\ar@{^{(}->}[r]\ar[dd]^{\nu_X}
            \ar[dr]^{\gamma}& X \times G(d,n)\ar[d]^{pr_2} \\
             & &  G(d,n) \\
             E_0X\ar[r]_{e_0}  & X &}\] 
    We know that the strict transforms of the polar curves of $X$ are not separated by $e_0$ since some of their branches all go through the point $\ell\in e_0^{-1}(0)$. On the contrary the strict transforms of the polar curves are separated in $NX$ because if there was a base point $(0,T)\in\nu^{-1}(0)$ for the system of strict transforms of polar curves, considering the duals $\P^1(\rho)$ of the kernels of linear projections $\rho\colon \C^n\to\C^2$, in view of Corollary \ref{NashvsConormalCoro} it would mean that $\P^2(T)\cap\P^1(\rho)\neq\emptyset$ in $\P^4$ for general $\rho$, which cannot be for dimension reasons.\par The same proof generalizes to any dimension to show that there can be no base point for the system of strict transforms on $NX$ of polar varieties of any dimension. See also \cite[Chap. III, Theorem 1.2]{Sp}.\par\noindent This implies that the limits of tangent planes to $X$ which contain $\ell$ are limits of tangent planes to $X$ along some branches of the polar curves which all have the same tangent $\ell$ at the origin.\par                
 Consider now the space of linear projections $\pi\colon \C^5\to\C^3$. At least for a general projection, the image of $X$ will be a hypersurface $X_\pi\in\C^3$ having $\ell_\pi=\pi(\ell)$ as an exceptional tangent. To the projection $\pi$ corresponds an injection $\check\P^2(\pi)\subset \check\P^4$ which maps a hyperplane in $\C^3$ to its inverse image by $\pi$.\par\noindent
The line $\P^1(\ell_\pi)\subset \check\P^2(\pi)\subset\check\P^4$ consisting of the hyperplanes of $\C^3$ containing $\ell_\pi$ is the intersection in $\check\P^4$ of $\check\P^2(\pi)$ with the $\P^3$ of hyperplanes containing $\ell$. Since $\ell_\pi$ is an exceptional line in the tangent cone of $X_\pi$, each point $z$ of $\P^1(\ell_\pi)$ corresponds to a limit tangent plane $T_{\pi,z}$ of $X_\pi$ which has to be the image by $\pi$ of at least one limit tangent plane $T_z$ of $X$ containing $\ell$.\par
Since the images of general polar curves of $X$ are general polar curves of $X_\pi$ we can lift tangent planes to $X_\pi$ containing $\ell_\pi$ to tangent planes of $X$ containing $\ell$. Conversely, any limit tangent plane $T$ of $X$ has an image by $\pi$, at least for general $\pi$, which is a limit tangent plane to $X_\pi$.\par The point $z$ is the only intersection point of $\P^2(T_z)$ with $\P^1(\ell_\pi)$ in $\P^3$. Since $X$ is a rational singularity, the one dimensional components of the fiber $\nu_X^{-1}(0)$ must be projective lines so that the tangent planes which contain $\ell$ are parametrized by a $\P^1$. Thus, the $\P^2(T_z)$ must form themselves a linear system in $\P^3$, which means that they are the linear system of hyperplanes of $\P^3$ which contain a $\P^1\subset\P^3$. This in turn means that they are the duals of a system of lines in $\P^4$ containing $\ell$ and contained in a $\P^2\subset\P^4$.\par Since the set of limits at $0$ of tangent planes to $X^0$ contains the limits of tangent spaces to the (reduced) tangent cone at its nonsingular points, this $\P^2$ has to contain the point $\ell$ and the two lines through $\ell$ representing the tangents at $\ell$ to the two conics. It has to be the $\P^2$ spanned by the tangents at $\ell$ to the two conics which form the projectivization of the tangent cone of $X$ at $0$.\par\noindent
In conclusion, the $\P^2(T)\subset \P^3\subset\kappa^{-1}(0)$ which correspond to limits $T$ of tangent planes to $X$ containing the exceptional tangent $\ell$ are those which contain the $\P^1\subset\check\P^4$ dual to the $\P^2\subset\P^4$ spanned by the tangents at $\ell$ to the two conics. Those $\P^2(T)$ do constitute a $\P^1$, which is part of the fiber $\nu_X^{-1}(0)$ of the Semple-Nash modification of $X$.
\end{example}
%%%%%%%%%%%%%%%%%%%%%%%%%%%Fin Modification %%%%%%%%%%%%%%%\P^{n-1}
 \noindent\textbf{Problem:} For a general reduced $d$-equidimensional germ $(X,0)\subset(\C^n,0)$ of a complex analytic space, find a way to characterize those $\P^{n-d-1}\subset\kappa^{-1}(0)$ which correspond to limits at $0$ of tangent spaces to $X^0$. Perhaps one can use the family of the general projections $\pi\colon\C^n\to\C^{d+1}$, which map limits of tangent spaces to $X$ to limits of tangent spaces to the hypersurface $\pi(X)$ and behave well on the tangent cone. Except in the case of curves (see \cite[Chap.I, \S 6]{Te3}) the equisingularity properties of the family of general hypersurface projections of a given singularity are largely unexplored (see \cite[Chap.VI, \S6]{Te3}). For generic projections to $\C^d$ one can consult \cite{Ak}.
\section{Whitney Stratifications}

       Whitney had observed, as we can see from the statement of Lemma \ref{wl}, that ``asymptotically'' near $0$ a germ $(X,0)\subset (\C^n,0)$ 
       behaves like a cone with vertex $0$,  in the sense that for any sequence $(x_i)_{i\in \N}$ of nonsingular points of $X$ tending to zero, the limit (up to restriction to a subsequence) of the tangent spaces $T_{x_i}X^0$ contains the limit of the secants $\overline{0x_i}$. Suppose now that we replace $0$ by a nonsingular subspace
       $Y\subset X$, and we want to understand what it means for $X$ to ``behave asymptotically like a cone with vertex $Y$''.\par\noindent First let us define what we mean by a cone with vertex $Y$.

\begin{definition}
       A cone with vertex $Y$ is a space $C$ equipped with a map \[\pi\colon C \longrightarrow Y\] and homotheties in the fibers, i.e., a morphism $\eta\colon C\times \C^*\to C$ with $\pi\circ\eta=\pi\circ{\rm pr}_1$, inducing an action of the multiplicative group $\C^*$ in the fibers of $\pi$ which has as fixed set the image of a section $\sigma\colon Y\rightarrow C$ of $\pi$; $\pi\circ \sigma={\rm Id}_Y$ and $\sigma (Y)=\{c\in C\vert \eta(c,\lambda)=c\ \forall \lambda\in\C^*\}$. We shall consider only homogeneous cones, which means that $Y$ is covered by open sets $U$ such that there is an embedding $\pi^{-1}(U)\subset U\times\C^k$ with $\sigma(U)=U\times\{0\}$, the map $\pi$ is induced by the first projection and $\eta$ is induced by the homotheties of $\C^k$.\end{definition}

      Let us take a look at the basic example we have thus far constructed.

\begin{example}\label{conelike}$\;$\\
 The reduced normal cone $|C_{X,Y}|\longrightarrow Y$, with the canonical analytic projection mentioned after definition
   \ref{defnc}.
\end{example}

     Now we can state Whitney's answer to the problem posed at the beginning of this section, again in terms of tangent spaces and secants:

\subsection{Whitney's conditions}

     Let $X$ be a reduced, pure dimensional analytic space of dimension $d$, let $Y\subset X$  be a nonsingular 
     analytic subspace of dimension $t$ containing $0$. Choose a local embedding $(X,0)\subset (\C^n,0)$ around $0$, and a local holomorphic
     retraction $\rho:(\C^n,0)\longrightarrow (Y,0)$. Note that, since $Y$ is nonsingular we can assume it is an open
     subset of $\C^t$, $(X,0)$ is embedded in an open subset of $\C^t \times \C^{n-t}$ and the retraction $\rho$ coincides with the
     first projection.
 
     We say that $X^0$ satisfies Whitney's conditions along $Y$ at $0$ if for any sequence of pairs of points
     $\{(x_i,y_i)\}_{i\in \N} \subset  X^0 \times Y$ tending to $(0,0)$ we have:
     \[\lim_{i\to \infty} [x_iy_i] \subset \lim_{i\to \infty} T_{X,x_i}\]
     where $[x_iy_i]$ denotes the line passing throught these two points. If we compare this to Whitney's lemma,
     it is just spreading out along $Y$ the fact observed when $Y=\{0\}$.
     
     In fact, Whitney stated two conditions, which together are equivalent to the above. Letting $\rho$ be the aforementioned 
     retraction, here we state the two conditions:

\begin{itemize}
     \item[a)] For any sequence $\{x_i\}_{i \in \N}\subset X^0$, tending to $0$ we have:
               \[ T_{Y,0} \subset \lim_{i\to \infty} T_{X,x_i} .\]
     \item[b)] For any sequence $\{x_i\}_{i \in \N}\subset X^0$, tending to $0$ we have:
               \[ \lim_{i\to \infty} [x_i\rho(x_i)] \subset \lim_{i\to \infty}  T_{X,x_i}.\]        
\end{itemize}
We leave it as an interesting exercise for the reader to verify that a homogeneous cone with vertex $Y$ satisfies these conditions. One uses the fact that locally, the equations defining $\pi^{-1}(U)$ in $ U\times\C^k$ are homogeneous polynomials in $z_1,\ldots ,z_k$ whose coefficients are analytic functions on $U$ and then takes a close look at the consequences of Euler's identity in terms of relative sizes of partial derivatives with respect to the coordinates $z_i$ and the coordinates on $Y$ and of the angle, or distance (see subsection \ref{defw}), between the secant line $\overline{x\pi(x)}$ and the tangent space to $X$ at $x\in X^0$. The shortest detailed proof involves integral dependence on ideals and cannot find its place here.\par\medskip
  Whitney's conditions can also be characterized in terms of the conormal space and the normal 
 /conormal diagram, as we shall see later.\par\noindent
 The fact that the Whitney conditions are independent of the embedding is not obvious from these definitions. It follows from the algebraic characterization explained in section \ref{total} below. 
 Note that condition $a)$ has the important consequence that a linear space transversal to $Y$ in $\C^n$ at a point $y$ will remain transversal to $X^0$ in a neighborhood of $y$.
 \begin{remark}\label{noemdim} It is convenient, for example when studying resolutions of singularities, to be able to consider as "similar" singularities which do not have the same embedding dimension. We point out that the fact that $X$ satisfies the Whitney conditions along $Y$ does not imply that the embedding dimension of $X$ is constant along $Y$. The surface defined in $\C^4$ by the equations $u_2^2-u_1^3-vu_3=0, u_3^2-u_1^5u_2-\frac{1}{16}v^2u_1^7=0$ satisfies the Whitney conditions along the $v$-axis but its embedding dimension is $3$ at the points of this axis where $v\neq 0$ and $4$ at the origin. Its multiplicity is $4$ at each point of the $v$-axis. The slice by $v=v_0\neq 0$ is isomorphic to the plane branch $(u_2^2-u_1^3)^2-u_1^5u_2-\frac{1}{16}u_1^7=0$ while the slice by $v=0$ is the monomial curve  given parametrically by $u_1=t^4,u_2=t^6, u_3=t^{13}$, which has embedding dimension $3$.\par
 It is not difficult to verify that our surface $X$ can be parametrized as follows: $u_1=t^4, u_2=t^6+\frac{1}{2}vt^7, u_3=t^{13}+\frac{1}{4}vt^{14}$. If we compose this parametrization with a sufficently general linear projection $p\colon X\to \C^2$, given by $u_1+\lambda u_3, u_2+\mu u_3$, with $\lambda,\mu \in\C^2$, $\lambda\neq 0$, we see that the composed map has rank two outside of $t=0$. This implies that the general polar curve of $X$ is empty and therefore equimultiple (of multiplicity $0$) along the $v$-axis $Y$ in a neighborhood of $0$. Since the tangent cone of $X$ at every point of $Y$ is defined by the ideal $(u_2^2, u_3^2)$, the multiplicity of $X$ is constantly equal to $4$ along $Y$ (see subsection \ref{mul}, e)). By theorem \ref{PE} below, the pair $(X^0,Y)$ satisfies Whitney's conditions at $0$. 
 \end{remark}

%\scalebox{0.4}{\includegraphics{Strat.pdf}}

      Recall that our objective is to ``stratify'' $X$. What exactly do we mean by stratify, and how do Whitney
     conditions relate to this? What follows in this section consists mostly of material from \cite[\S 2]{Lip} 
     and \cite[Chapitre III]{Te3}.

\begin{definition}\label{stratif}
     A \textbf{stratification} of $X$ is a decomposition into a locally finite disjoint union $X=\bigcup X_\alpha$, of non-empty, 
     connected, locally closed subvarietes called \emph{strata}, satisfying:
\begin{itemize}
   \item [(i)] Every stratum $X_\alpha$ is nonsingular (and therefore an analytic manifold).   
   \item [(ii)] For any stratum $X_\alpha$, with closure $\overline{X_\alpha}$, both the frontier $\partial X_\alpha:= 
           \overline{X_\alpha} \setminus X_\alpha$ and the closure $\overline{X_\alpha}$ are closed analytic in $X$.
   \item [(iii)] For any stratum $X_\alpha$, the frontier $\partial X_\alpha:= \overline{X_\alpha} \setminus X_\alpha$
             is a union of strata.             
\end{itemize}
\end{definition}

    Stratifications can be determined by \emph{local stratifying conditions} as follows. 
    We consider conditions $C=C(W_1,W_2,x)$ defined for all $x \in X$ and all pairs $(W_1,x)\supset (W_2,x)$ of
    subgerms of $(X,x)$ with $(W_1,x)$ equidimensional and $(W_2,x)$ smooth. For example, $C(W_1,W_2,x)$ could signify
    that the Whitney conditions hold at x.

    For such a $C$ and for any subvarieties $W_1,W_2$ of $X$ with $W_1$ closed and locally equidimensional, and $W_2$ 
    locally closed, set
    \begin{align*}
         C(W_1,W_2) & :=\{x \in W_2 | \, W_2 \; \mathit{is\;smooth \;at \; x, \; and \; if} \; x \in W_1 \; \mathit{then} 
         \; (W_1,x)\supset (W_2,x) \; \mathit{and}\\ & \; C(W_1,W_2,x)\}, \\
         \widetilde{C}(W_1,W_2) & := W_2 \setminus C(W_1,W_2).
    \end{align*}  

    The condition C is called \emph{stratifying} if for any such $W_1$ and $W_2$, the set $\widetilde{C}(W_1,W_2)$ is contained in a 
    nowhere dense closed analytic subset of $W_2$. In fact, it suffices that this be so whenever $W_2$ is smooth, connected,
    and contained in $W_1$.\\

    %A stratification is a \emph{$C$-stratification} if for any strata $X_\alpha$, $X_\beta$ with 
    %$X_\beta \subset \overline{X_\alpha}$ it holds that $\widetilde{C}(\overline{X_\alpha},\overline{X_\beta})^{Zar}$
    %is a union of strata.\\

    Going back to our case, it is true that Whitney's conditions are stratifying. See \cite[Lemma 19.3, p. 540]{Whi1}. 
    The key point is to prove, given $Y\subset X$ as in section 4.1, that the set of points of $Y$ where the pair 
    $(X^0,Y)$ satisfies Whitney's conditions contains the complement of a strict closed analytic subspace of $Y$. 
  A proof of this different from Whitney's is given below as a consequence of theorem \ref{prin}.\\

    \begin{definition}\label{defwhstr} Let $X$ be as above, then by a \textbf{Whitney stratification} of $X$,
     we mean a stratification $X=\bigcup X_\alpha$, such that for any pair of strata $X_{\beta}, X_{\alpha}$
     with $X_{\alpha}\subset \overline{X_{\beta}}$, the pair $(\overline{X_{\beta}},X_{\alpha})$ satisfies
     the Whitney conditions at every point $x \in X_{\alpha}$.
    \end{definition}
    \begin{figure}[!ht]
    \begin{center}
    \label{schematic2}
              \includegraphics*{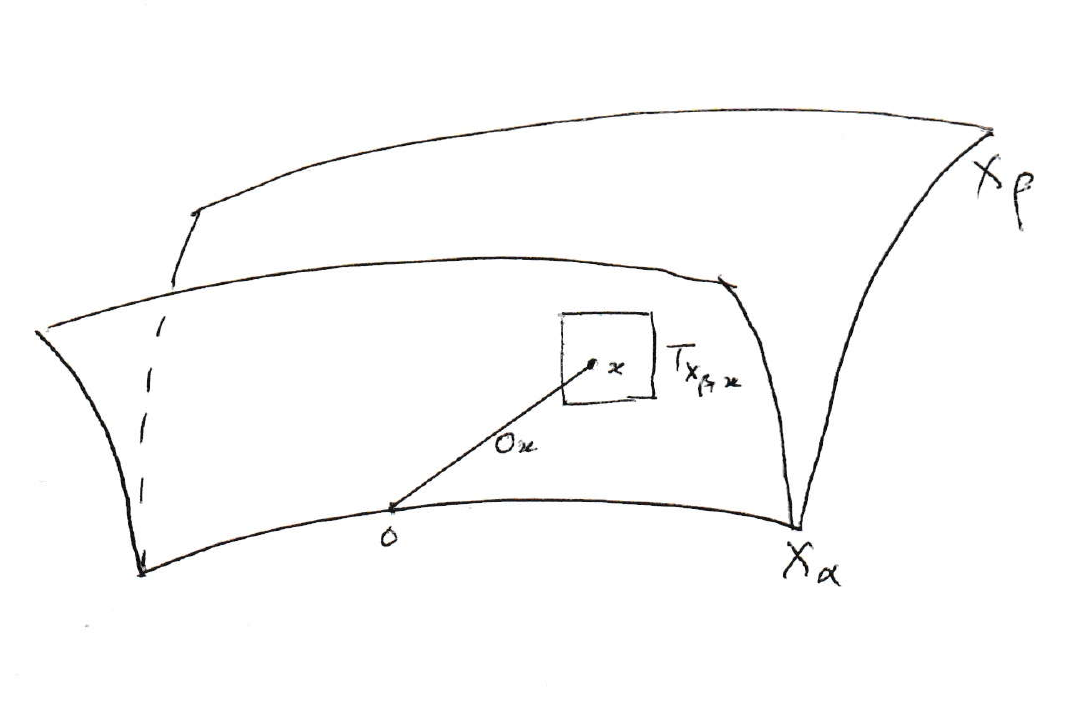}
    % \caption{ Sing $X=$ eje z; $C_{X,0}=$ plano yz}
    \end{center}
\end{figure}
\subsection{Stratifications}\label{stratifications}

    We will now state two fundamental theorems concerning Whitney's conditions, the first of which was proved by
    Whitney himself and the second by R. Thom and J. Mather. The proofs can be found in \cite{Whi1}, and \cite{Ma1} 
    respectively.

\begin{theorem}{\rm (Whitney)}
    Let $M$ be a reduced complex analytic space and let $X\subset M$ be a locally closed analytic subspace of $M$.
    Then, there exists a Whitney stratification $M=\bigcup M_\alpha$ of $M$ such that:
    \begin{itemize}
     \item[i)] $X$ is a union of strata. 
     \item[ii)] If $M_\beta \cap \overline{M_\alpha} \neq \emptyset$ then $M_\beta \subset \overline{M_\alpha}$.
     %\item[iii)] If $M_\beta \subset \overline{M_\alpha}$ the pair $(\overline{M_\alpha}^\circ,M_\beta)$ satisfies
     %           Whitney's conditions at every point of $M_\beta$.
    \end{itemize}  
\end{theorem}
  In fact, one can prove that \textit{any stratifying condition gives rise to a locally finite stratification of any space $X$ such that all pairs of strata satisfy the given condition.} See (\cite[\S 2]{Lip}, \cite[p. 478-480]{Te3}).\par\medskip\noindent
  Given a germ of $t$-dimensional nonsingular subspace $(Y,0)\subset (\C^n,0)$, by a (germ of) local holomorphic retraction $\rho\colon  (\C^n,0)\to (Y,0)$ we mean the first projection of a product decomposition $(\C^n,0)\simeq (Y,0)\times (\C^{n-t},0)$. By the implicit function theorem, such retractions always exist. One applies to such retractions Thom's first isotopy lemma to prove:  
\begin{theorem}\label{T-M}\textnormal{(Thom-Mather)(see \cite[Th\'eor\`eme 1G1]{Th}, \cite{Ma2}, \cite{G-M}, Chap. 1, 1.5)}$\;$\\
    Taking $M=X$ in the previous statement, let $X=\bigcup_\alpha X_\alpha$ be a Whitney stratification of $X$, let $x \in X$ and let $X_\alpha\subset X$ be the stratum
    that contains $x$. Then, for any local embedding $(X,x)\subset (\C^n,0)$ and any local retraction
    $\rho:(\C^n,0)\to (X_\alpha,x)$ and a real number $\epsilon_0>0$ such that for all $0<\epsilon < \epsilon_0$ 
    there exists $\eta_\epsilon$ such that for any $0<\eta < \eta_\epsilon$ there is a homeomorphism h
    \[\xymatrix{\B(0,\epsilon) \cap \rho^{-1}(\B(0,\eta) \cap M_\alpha) \ar[rr]^h \ar[ddr]_\rho  & &
                (\rho^{-1}(x)\cap \B(0,\epsilon )) \times (X_\alpha \cap \B(0,\eta))\ar[ddl]^{pr_2} \\ \\
              & X_\alpha \cap \B(0,\eta) & }\]
    compatible with the retraction $\rho$, and inducing for each stratum $X_\beta$ such that $\overline{X_\beta}
    \supset X_\alpha$ a homeomorphism 
    \[\overline{X_\beta}\cap \B(0,\epsilon ) \cap \rho^{-1}(\B(0,\eta) \cap X_\alpha) \longrightarrow
               (\overline{X_\beta}\cap \rho^{-1}(x)\cap(\B(0,\epsilon))\times (X_\alpha \cap \B(0,\eta))\]
    where $\B(0,\epsilon )$ denotes the ball in $\C^n$ with center in the origin and radius $\epsilon$.
\end{theorem}

    In short, each $\overline{X_\beta}$, or if you prefer, the stratified set $X$, is locally topologically trivial
    along $X_\alpha$ at $x$. A natural question then arises: is the converse to the Thom-Mather theorem true? That
    is, does local topological triviality implies the Whitney conditions? The question was posed by the second author in \cite{Te1} for families of hypersurfaces with isolated singulaities.

    The answer is \emph{NO}. In [BS] Brian\c con and Speder showed that the family of surface germs
    \[z^5+ty^6z+xy^7+x^{15}=0\]
    (each member, for small t, having an isolated singularity at the origin) is locally topologically trivial, 
    but not Whitney in the sense that the nonsingular part of this hypersurface in $\C^4$ does not satisfy at the origin Whitney's condition along the singular locus, which is the $t$-axis.\\
    To explain the origin of this example we need to introduce the Milnor number of an isolated singularity of hypersurface. We shall meet it again in example \ref{smoothing} below. The Milnor number $\mu^{(d+1)}(X,x)$ of an isolated singularity of hypersurface $f(z_1,\ldots ,z_{d+1})=0$ as above is defined algebraically as the multiplicity in $\C\{z_1,\ldots ,z_{d+1}\}$ of the Jacobian ideal $j(f)=(\frac{\partial f}{\partial z_1},\ldots ,\frac{\partial f}{\partial z_{d+1}})$, which is also the dimension of the $\C$-vector space $\frac{\C\{z_1,\ldots ,z_{d+1}\}}{j(f)}$ since in this case the partial derivatives form a regular sequence.\par
    \begin{definition}\label{TT} We say that two germs $(X,0)\subset(\C^n,0)$ and $(X',0)\subset(\C^n,0)$ have the same embedded topological type if there exists a germ of homeomorphism $\phi\colon (\C^n,0)\to (\C^n,0)$ such that $\phi (X)=X'$.
    \end{definition} 
    By the Thom-Mather topological triviality theorem (Theorem \ref{T-M} above), if a family of germs of spaces with isolated singularities satisfies the Whitney conditions along its singular locus, all its members have the same embedded topological type.\par
     In \cite[Th\'eor\`eme 1.4]{Te1} it was shown that the Milnor number is an invariant of the embedded topological type of the germ of hypersurface and it was conjectured that the constancy of the Milnor number at the origin in an analytic family $X$ given by $F(t, z_1,\ldots ,z_{d+1})=0$, with $F(t, 0,\ldots ,0)= 0$, of hypersurfaces with isolated singularities was equivalent to the Whitney conditions for the smooth part of $X$ along the $t$-axis. Part of this conjecture was the statement that if the Milnor number is constant in an analytic family of hypersurfaces with isolated singularity, then the whole sequence of Milnor numbers $\mu^{(i)}(X_t,0)=\mu (X_t\cap H_{d+1-i},0)$ of general plane sections of all dimensions $i=1,\ldots, d+1$ (which also have isolated singularities) is constant. It was proved in \cite[Chap. II, \S 3]{Te1} that the constancy of all these Milnor numbers implies the Whitney conditions. The converse was proved later by Brian\c con and Speder in \cite{B-S2}. We shall see below in example \ref{smoothing} that this equivalence is now a special case of a general result. Note that the Milnor number of a general section by a line is the multiplicity minus one.\par\noindent  In the same period, L\^e and Ramanujam proved in \cite{L-R} that the constancy of the Milnor number implied the topological triviality of $X$ along the $t$ axis when $d\neq 2$ and L\^e proved in \cite{L2} that when $d=1$ the constancy of the Milnor number implies the constancy of multiplicity.\par\noindent Therefore, to prove that the local topological triviality does not imply the Whitney conditions for the nonsingular part of $X$ along the $t$-axis and thus give a counterexample to the conjecture of \cite{Te1} is the same as proving that the local topological type of the singularity does not determine the local topological type of a general hyperplane section through the origin, or that the family obtained by intersecting $X$ by a general hyperplane containing the $t$ -axis is not topologically trivial. This is what the example of Brian\c con-Speder does.
    
\begin{ex}\label{B-S}    Going back to this example, let us consider plane sections of the hypersurface $z^5+ty^6z+y^7x+x^{15}=0$ by the hyperplanes $y=ax+bz$ which form an Zariski open subset of the family of hyperplanes in $\C^4$ containing the $t$ axis. The resulting equation is the family of curves in the $(x,z)$ plane:
    \[z^5+tz(ax+bz)^6+x(ax+bz)^7+x^{15}=0\]
Using classical Newton polygon methods it is not difficult to see that if $a\neq 0$  for $t=0$ the germ of curve is irreducible, with a singularity of type $z^5-x^8=0$, while if $t\neq 0$ the curve has three irreducible components, two of type $z^2-x^3=0$ and one nonsingular component of type $z-x^2=0$.
\end{ex}
  However, a strengthened version of local topological triviality is equivalent to the Whitney conditions. This was proved by L{\^e} and Teissier (see \cite[\S 5]{L-T3}, and \cite[Chapitre VI]{Te3}). Let us refer to the 
    conclusion of the Thom-Mather theorem as condition $(TT)$ (local topological triviality), so we can restate theorem 
    \ref{T-M} as: Whitney implies $(TT)$.\par Let $X=\bigcup X_\alpha$ be a stratification of the complex analytic space $X$ and let $d_\alpha={\rm dim}X_\alpha$. We say 
    that a stratification satisfies the condition $(TT)^*$ (local topological triviality for the general sections) if in 
    addition to the condition $(TT)$, for every $x \in X_\alpha$, there exists for every $k> {\rm dim}X_\alpha$ a dense Zariski open
    set $\Omega$ in $G(k-d_\alpha, n-d_\alpha)$ such that for any nonsingular space $E$ containing $X_\alpha$ and such
    that $T_xE \in \Omega$, the (set-theoretic) intersection $\overline{X_\beta}\cap E$ satisfies $(TT)$ for all $X_\beta$ such that
    $\overline{X_\beta}\supset X_\alpha$. 

\begin{theorem}\label{TT}\textnormal{(L{\^e}-Teissier), see \cite[Th\'eor\`eme 5.3.1]{L-T3}}$\;$\\
    For a stratification  $X=\bigcup X_\alpha$ of a complex analytic space $X$, the following conditions are equivalent:
    \begin{itemize}
          \item [1)] $X=\bigcup X_\alpha$ is a Whitney stratification.
          \item [2)] $X=\bigcup X_\alpha$ satifies condition $(TT)^*$.
    \end{itemize}           
\end{theorem}
  \noindent We shall see more about this below in section \ref{total}. 
  \subsection{Whitney stratifications and polar varieties}  \label{WhitneyinNormalConormal}

     We now have all the ingredients so it is time to put them together. Let us fix a nonsingular subspace 
     $Y \subset X$ through $0$ of dimension $t$ as before, recall that we are assuming $X$ is a reduced, pure dimensional analytic 
     space of dimension $d$. Let us recall the notations of section 3.1  and take a look at the normal/conormal diagram:

     \[\xymatrix{E_YC(X)\ar[r]^{\hat{e}_Y}\ar[dd]^{\kappa'}\ar[ddr]^\xi & C(X)\ar[dd]^\kappa
             \\
             & &  \\
             E_YX\ar[r]_{e_Y}  & Y\subset X }\]
             
     Remember that $E_YC(X)$ is the blowing up of $\kappa^{-1}(Y)$ in $C(X)$, and $\kappa'$ is obtained
 from the universal property of the blowing up, with respect to $E_YX$ and the map $\xi$. Just as in the case where $Y=\{0\}$, it is 
 worth mentioning that $E_YC(X)$ lives inside the fiber product $C(X) \times_X E_YX \subset X  
 \times \P^{n-1-t}\times \check {\P}^{n-1}$ and can be described in the following 
 way: take the inverse image of $E_YX \setminus e_Y^{-1}(Y)$ in $C(X) \times_X E_YX$ and close it,  
 thus obtaining $\kappa'$ as the restriction of the second projection to this space.\\
    
     Looking at the definitions, it is not difficult to prove that, if we consider the divisor:
     \[D=|\xi^{-1}(Y)| \subset E_YC(X), \;\; D \subset Y \times \P^{n-1-t}\times \check {\P}^{n-1},\]
     we have that, denoting by $\check{\P}^{n-1-t}$ the space of hyperplanes containing $T_0Y$:

     \begin{itemize}
     \item[$\bullet$)] The pair $(X^0,Y)$ satisfies Whitney's condition a) along $Y$ if and only if
     we have the set theoretical equality $C(X)\cap C(Y) = \kappa^{-1}(Y)$. It satisfies Whitney's condition a) at $0$ if and only if $\xi^{-1} (0) \subset
     \P^{n-1-t}\times \check {\P}^{n-1-t}$.
     \end{itemize}
     
      Note that we have the inclusion $C(X)\cap C(Y) \subset \kappa^{-1}(Y)$, so
     it all reduces to having the inclusion $\kappa^{-1}(Y)\subset C(Y)$, and since we have
     already seen that every limit of tangent hyperplanes $H$ contains a limit of
     tangent spaces $T$, we are just saying that every limit of tangent hyperplanes to $X$ at
     a point $y \in Y$, must be a tangent hyperplane to $Y$ at $y$. Following this line of thought,
     satisfying condition a) at $0$ is then equivalent to the inclusion $\kappa^{-1}(0) \subset
     \{0\} \times \check{\P}^{n-1-t}$ which implies $\xi^{-1} (0) \subset \P^{n-1-t}\times \check
     {\P}^{n-1-t}$.
         
     \begin{itemize}
     \item[$\bullet$)] The pair $(X^0,Y)$ satisfies Whitney's condition b) at $0$ if and only
      if $\xi^{-1}(0)$ is contained in the incidence variety $I\subset \P^{n-1-t}\times \check
      {\P}^{n-1-t}$.
     \end{itemize}
     
        This is immediate from the relation between limits of tangent hyperplanes and limits of  
 tangent spaces and the interpretation of $E_YC(X)$ as the closure of the inverse image
 of $E_YX \setminus e_Y^{-1}(Y)$ in $C(X) \times_X E_YX$ since we are basically taking limits as $x\to Y$ of 
 couples $(l,H)$ where $l$ is the direction in $\P^{n-1-t}$ of a secant line $\overline{yx}$ with $x\in X^0\setminus Y, y=\rho (x)\in Y$, where $\rho$ is some local retraction of the ambient space to the nonsingular subspace $Y$, and $H$ is a tangent hyperplane to $X$ at $x$. So, in order to verify the Whitney conditions, it is important to control the geometry of the projection $D\to Y$ of the divisor
 $D \subset E_YC(X)$.\\
 \begin{remark}\label{Intdep} Although it is beyond the scope of these notes, we point out to the interested reader that there is an algebraic definition of the Whitney conditions for $X^0$ along $Y\subset X$ solely in terms of the ideals defining $C(X)\cap C(Y)$ and $\kappa^{-1}(Y)$ in $C(X)$. Indeed, the inclusion $C(X)\cap C(Y) \subset \kappa^{-1}(Y)$ follows from the fact that the sheaf of ideals $\Jj_{C(X)\cap C(Y)}$ defining $C(X)\cap C(Y)$ in $C(X)$ contains the sheaf of ideals $\Jj_{\kappa^{-1}(Y)}$ defining $\kappa^{-1}(Y)$, which is generated by the pull-back by $\kappa$ of the equations of $Y$ in $X$. What was said above means that condition a) is equivalent to the second inclusion in:
     \[\Jj_{\kappa^{-1}(Y)}\subseteq \Jj_{C(X)\cap C(Y)}\subseteq \sqrt{\Jj_{\kappa^{-1}(Y)}}.\]
     It is proved in \cite{L-T2}, proposition 1.3.8 that having both Whitney conditions is equivalent to having the second inclusion in:
     \[\Jj_{\kappa^{-1}(Y)}\subseteq \Jj_{C(X)\cap C(Y)}\subseteq \overline{\Jj_{\kappa^{-1}(Y)}}.\]
     where the bar denotes the {\rm integral closure} of the sheaf of ideals, which is contained in the radical and is in general much closer to the ideal than the radical. The second inclusion is an algebraic expression of the fact that locally near every point of the common zero set the modules of local generators of the ideal $\Jj_{C(X)\cap C(Y)}$ are bounded, up to a multiplicative constant depending only on the chosen neighborhood of the common zero, by the supremum of the modules of generators of $\Jj_{\kappa^{-1}(Y)}$. See \cite[Th\'eor\`eme 7.2]{Lej-Te}. We shall see more about it in section \ref{sec:specialization}.\par\noindent
In the case where $Y$ is a point, the ideal defining $ C(X)\cap C(\{y\})$ in $C(X)$ is just the pull-back by $\kappa$ of the maximal ideal $m_{X,y}$, so it coincides with $\Jj_{\kappa^{-1}(Y)}$ and Whitney's lemma follows.
     \end{remark}

     \begin{definition} Let $Y \subset X$ as before. Then we say that the local polar variety $\polar$ is
      equimultiple along $Y$ at a point $x \in Y$ if the map $y\mapsto m_y(\polar)$ is constant for $y\in Y$ in a neighborhood of $x$. \par\noindent Note that this implies that if  $(\polar ,x)\neq \emptyset$, then $\polar \supset Y$ in a neighborhood of $x$ since the emptyness of a germ is equivalent to multiplicity zero.
     \end{definition}
     
     Now we can state the main theorem of these notes, a complete proof of which can be found in
     \cite[Chapter V, Thm 1.2, p. 455]{Te3}.
          
     \begin{theorem}\label{PE}{\rm (Teissier; see also \cite{HM} for another proof)}\label{prin}
            Given $0 \in Y \subset X$ as before, the following conditions are equivalent, where $\xi$ is the diagonal map in the normal/conormal diagram above:
     \begin{itemize}
        \item[1)] The pair $(X^0,Y)$ satisfies Whitney's conditions at $0$. 
        \item[2)] The local polar varieties $P_k(X,L)$, $0\leq k \leq d-1$, are equimultiple
                  along $Y$ (at $0$), for general $L$.
        \item[3)] ${\rm dim} \; \xi^{-1}(0)=n-2-t.$
     \end{itemize}        
     \end{theorem}  

      Note that since ${\rm dim\  D}=n-2$ condition 3) is open and the theorem implies that
      $(X^0,Y)$ satisfies Whitney's conditions at $0$ if and only if it satisfies Whitney's 
      conditions in a neighborhood of $0$.\\

      Note also that by analytic semicontinuity of fiber dimension (see \cite[Chap. 3, 3.6]{Fi}, \cite[\S\   49]{K-K}), condition 3) is satisfied outside of a closed
      analytic subspace of $Y$, which shows that Whitney's conditions give a stratifying condition.\\

 Moreover, since a blowing up does not lower dimension, the condition \break${\rm dim} \;\xi^{-1}(0)=n-2-t$ implies 
      ${\rm dim}\; \kappa^{-1}(0)\leq n-2-t$. So that, in particular $\kappa^{-1}(0) \not\supset \check{\P}^{n-1-t}$,
      where $\check{\P}^{n-1-t}$ denotes as before the space of hyperplanes containing $T_0Y$. This tells us 
      that \textit{a general hyperplane containing $T_0Y$ is not a limit of tangent hyperplanes to $X$}. This fact is crucial in the proof that Whitney conditions are equivalent to the equimultiplicity of polar varieties since it allows the start of an inductive process. In the actual proof of \cite{Te3}, one reduces to the case where ${\rm dim}Y=1$ and shows by a geometric argument that the Whitney conditions imply that the polar curve has to be empty, which gives a bound on the dimension of $\kappa^{-1}(0)$. Conversely, the equimultiplicity condition on polar varieties gives bounds on the dimension of $\kappa^{-1}(0)$ by implying the emptiness of the polar curve and on the dimension of $e_Y^{-1}(0)$ by Hironaka's result, hence a bound on the dimension of $\xi^{-1}(0)$. \\
      
      It should be noted that Hironaka had proved in \cite[Corollary 6.2]{Hi} that the Whitney conditions for $X^0$ along $Y$ imply equimultiplicity of $X$ along $Y$.\\

      Finally, a consequence of the theorem is that \textit{given a complex analytic space $X$, there is a unique minimal
      (coarsest) Whitney stratification}; any other Whitney stratification of $X$ is obtained by adding strata
      inside the strata of the minimal one. A detailed explanation of how to construct this ``canonical" Whitney
      stratification using theorem \ref{prin}, and a proof that this is in fact the coarsest one can 
      be found in \cite[Chap. VI, \S\ 3]{Te3}. The connected components of the strata of the minimal Whitney stratification give a minimal "Whitney stratification with connected strata"\\

\subsection{Relative Duality}\label{reldu}
  
      There still is another result which can be expressed in terms of the relative conormal space
      and therefore in terms of relative duality. We first need the:\\

           \begin{proposition}{{\rm (Versions of this appear in \cite{La2}, \cite{Sa}.)}}\label{Specializationproprelative}
	    Let $X \subset \C^n$ be a reduced analytic subspace of dimension $d$ and let $Y \subset X$ be a 
	nonsingular analytic proper subspace of dimension $t$. Let $\varphi:\mathfrak{X} \to \C$ be the specialization
	of $X$ to the normal cone $C_{X,Y}$ of $Y$ in $X$, and  let $C(X), C(Y)$ denote the conormal spaces of $X$ and $Y$
	respectively, in $\C^n \times \check{\C}^n$. Then the relative conormal space
	\[ \kappa_\varphi\circ\varphi:=q: C_\varphi(\mathfrak{X}) \to \mathfrak{X} \to \C\]
	is isomorphic, as an analytic space over $\C$, to the specialization space of $C(X)$ to the normal cone $C_{C(X), C(Y) \cap C(X)}$
	of $C(Y) \cap C(X)$ in $C(X)$. In particular, the fibre $q^{-1}(0)$ is isomorphic
	to this normal cone. 	
\end{proposition}    

\begin{proof}$\;$\\
       Let $I\subset J$ be the coherent ideals of the structure sheaf of $\C^n$ that define the analytic subspaces
  $X$ and $Y$ respectively, and let $p:\mathfrak{D}\to \C$ be the specialization space of $C(X)$ to the normal cone of  
    $C(Y) \cap C(X)$ in $C(X)$.
   Note that, in this context, both spaces $\mathfrak{D}$ and $C_\varphi(\mathfrak{X})$ are analytic subspaces of $\C \times
   \C^n \times \check{\C}^n$. Let us consider a local chart, in such a way that $Y \subset X \subset \C^n$, locally  
   becomes $ \C^t \subset X \subset \C^n$ with associated local coordinates:
   \[(v,y_1\ldots,y_t,z_{t+1},\ldots,z_n,a_1,\ldots,a_t,b_{t+1},\ldots,b_n) \]  
   in $\C \times \C^n \times \check{\C}^n$.\\
   
       Let $J:=\left< z_{t+1},\ldots, z_n\right>$ be the ideal
  defining $Y$ in $\C^n$.  One can verify that, just as in the case of the tangent cone (see exercise \ref{ejercontan} b)),
  if $f_1, \ldots, f_r$, are local equations for $X$ in $\C^n$ such that their initial forms $\mathrm{in}_J f_i$ generate the
  ideal of $\mathrm{gr}_J O_X $ defining the normal cone of $X$ along $Y$, the equations $F_i:= v^{-k_i}f_i(y,vz), 
  \; i=1,\ldots, r$, where $k_i= sup \{k | f_i \in J^k\}$  locally define in $\C\times\C^n$ the specialization space $\varphi:\mathfrak{X} \to \C$ of $X$ to the normal cone 
  $C_{X,Y}$. Furthermore, if you look closely at the equations, you will easily verify that the open set $\mathfrak{X}
  \setminus \varphi^{-1}(0)$ is isomorphic over $\C^*$ to $\C^*\times X $, via the morphism $\Phi$ defined by the map
  $(v,y,z) \mapsto (v,y,vz)$. \\
  
  We can now consider the relative conormal space,
   \[q:C_\varphi(\mathfrak{X}) \to \mathfrak{X} \to \C,\] 
  and thanks to the fact that $\mathfrak{X} \setminus \varphi^{-1}(0)$ is an open subset with fibers $\mathfrak{X}(v)$ isomorphic to $X$, 
  the previous isomorphism $\Phi$ implies that $C_\varphi(\mathfrak{X}) \setminus q^{-1}(0)$ is isomorphic over $\C^*$ to $\C^*\times C(X)$.\\
 
     On the other hand, note that, since $J=\left< z_{t+1},\ldots, z_n\right> $ in $\Oo_{\C^n}$, 
  the conormal space $C(Y)$ is defined in $\C^n\times\check\P^{n-1}$ by the sheaf of ideals $J^C$ generated (in $\Oo_{\C^n\times\check\C^n}$) by  $(z_{t+1},\ldots,z_n,a_1,\ldots,a_t)$.\par\noindent Thus, if we chose local generators 
  $(g_1, \ldots , g_s)$ for the sheaf of ideals defining $C(X)\subset \C^n\times \check\P^{n-1}$, whose $J^C\Oo_{C(X)}$-initial forms generate the initial ideal,  the equations 
  $G_i(v,y,z,a,b)=v^{-l_i}g_i(v,y,vz,va,b)$ locally define a subspace $\D\subset \C\times\C^n\times\check \P^{n-1}$ with a faithfully flat projection $\D 
  \stackrel{p}{\rightarrow} \C$, where the fiber $p^{-1}(0)$ is the normal cone 
  $C_{C(X),C(Y)\cap C(X)}$. Note that
  in this case the $l_i$'s are defined with respect to the ideal $J^C\Oo_{C(X)}$.\par\noindent  The open set $\D \setminus p^{-1}(0)$ is isomorphic
  to $\C^*\times C(X)$ via the morphism defined by $(v,y,z,a,b) \mapsto (v,y,vz,va,b)$. \par\noindent 
 
  This last morphism is a morphism of the ambient space to itself over $\C$
  \begin{align*}
    \psi: \C \times \C^n \times \check{\C}^n &\longrightarrow \C \times \C^n \times \check{\C}^n\\
                               (v,y,z,a,b) &\longmapsto (v,y,vz,va,b)
   \end{align*}
  which turns out to be an isomorphism when restricted to the open dense set
   $\C^* \times \C^n \times \check{\C}^n$. So, if we take the analytic subspace  $\C^*\times C(X)$ 
   in the image, as a result of what we just said, we have the equality
  $\psi^{-1}(\C^*\times C(X)) = \D \setminus p^{-1}(0)$.\\
  
   Finally, recall that both morphisms defining $q$, are induced by the natural projections
  \[ \C \times \C^n \times \check{\C}^n \rightarrow \C \times \C^n \rightarrow \C ,\]
  and therefore we have a commutative diagram:
  \[\xymatrix @C=.2in{
    C_\varphi(\mathfrak{X}) \ar @{^{(}->}[r] \ar @/_6.5pc/[ddrr]_q \ar[d]_{\kappa_\varphi} 
         & \C \times \C^n \times \check{\C}^n \ar[rr]^\psi \ar[d]^\pi  & &
            \C \times \C^n \times \check{\C}^n  \ar[d]^\pi\\ 
       \mathfrak{X}  \ar @{^{(}->}[r] \ar @/_/[drr]_\varphi
       & \C \times \C^n  \ar[rr]^\phi \ar[dr] & &  \C \times \C^n  \ar[dl] \\
       &  & \C &         
  }\]
     To finish the proof, it is enough to check that the image  by $\psi$ of $C_\varphi(\mathfrak{X})\setminus
  q^{-1}(0)$ is equal to $\C^*\times C(X)$, since we already know that $\psi^{-1}(\C^*\times C(X) 
 ) = \D \setminus p^{-1}(0)$  and so we will find an open dense set common to both spaces, which are faithfully flat over $\C$, 
  and consequently the closures will be equal.
  
     Let $(y,z) \in X$ be a smooth point. The vectors
     \[\nabla f_i (y,z):= \left( \frac{\partial f_i}{\partial y_1}(y,z), \cdots ,\frac{\partial f_i}{\partial y_t}(y,z)
                   ,\frac{\partial f_i}{\partial z_{t+1}}(y,z),\cdots, \frac{\partial f_i}{\partial z_n}(y,z)\right) ,\]
    representing the $1$-forms $df_i$ in the basis $dy_j,dz_i$ , generate the linear subspace of $\check{\C}^n$ encoding all the 1-forms that vanish on the tangent space
     $T_{(y,z)}X^0$, i.e. the fiber over the point $(y,z)$ in $C(X)$. Analogously, let $(v,y,z) \in \mathfrak{X}$ be a smooth
     point in  $\mathfrak{X} \setminus \varphi^{-1}(0)$. Then the vectors
 \[\nabla F_i (v,y,z):= \left( \frac{\partial F_i}{\partial y_1}(v,y,z), \cdots ,\frac{\partial F_i}{\partial y_t}(v,y,z)
           ,\frac{\partial F_i}{\partial z_{t+1}}(v,y,z),\cdots, \frac{\partial F_i}{\partial z_n}(v,y,z)\right) \] 
     generate the linear subspace of $\check{\C}^n$ encoding of all the 1-forms that vanish on the tangent space
     $T_{(v,y,z)}\mathfrak{X}(v)^0$, i.e. the fiber over the point $(v,y,z)$ in $C_\varphi(\mathfrak{X})$. According to our choice of 
     $(v,y,z)$, we know that $\phi((v,y,z))=(v,y,vz)$ is a smooth point of $\C^*\times X$ and in particular $(y,vz)$
     is a smooth point of $X$. Moreover, notice that:
     \begin{align*}
            \frac{\partial F_i}{\partial y_j}(v,y,z) &= v^{-n_i}\frac{\partial f_i}{\partial y_j}(y,vz)\\
            \frac{\partial F_i}{\partial z_k}(v,y,z) &= v^{-n_i+1}\frac{\partial f_i}{\partial z_k}(y,vz)    
     \end{align*} 
     and therefore the image of the corresponding point
      \begin{align*}
      \psi(v,y,z,\frac{\partial F_i}{\partial y_j}(v,y,z),\frac{\partial F_i}{\partial z_k}(v,y,z)) & = 
         (v,y,vz,v^{-n_i+1}\frac{\partial f_i}{\partial y_j}(y,vz),v^{-n_i+1}\frac{\partial f_i}{\partial z_k}(y,vz)) \\
        & = (v,y,vz,v^{-n_i+1}\nabla f_i(y,vz))                              
       \end{align*}  
         is actually a point in $\C^*\times C(X)$. Since $v \neq 0$, the $v^{-n_i+1}\nabla f_i(y,vz)$ also 
        generate the fiber over the point $(v,y,vz)\in\C^*\times X$ by the map $\C^*\times C(X) \to \C^*\times X$ induced by $\kappa_\varphi$ and the isomorphism $\varphi^{-1}(\C^*)\simeq \C^*\times X$, which implies that 
        $\psi$ sends $C_\varphi(X)\setminus q^{-1}(0)$ onto $\C^*\times C(X)$.
\end{proof}
      
 Going back to our normal-conormal diagram:

      \[\xymatrix{E_YC(X)\ar[r]^{\hat{e}_Y}\ar[dd]^{\kappa'}\ar[ddr]^\xi & C(X)\ar[dd]^\kappa
             \\
             & &  \\
             E_YX\ar[r]_{e_Y}  & Y\subset X }\]

      Consider the irreducible components $D_\alpha\subset Y\times \P^{n-1-t} \times \check{\P}^{n-1}$ of \newline
      $D=|\xi^{-1}(Y)|$, that is $D=\bigcup D_\alpha$, and its images:
      \begin{align*}
      V_\alpha & = \kappa'(D_\alpha)\subset Y \times \P^{n-1-t},\\
      W_\alpha  & = \hat{e}_Y (D_\alpha)\subset Y \times \check{\P}^{n-1}.
      \end{align*}      
      
      We have:
      \begin{theorem}\label{dual}{\rm (L\^e-Teissier, see \cite[Thm 2.1.1]{L-T2})} The equivalent statements of theorem \ref{prin} are 
      also equivalent to:\par\noindent
For each $\alpha$, the irreducible divisor $D_\alpha$ is the relative conormal space of its image $V_\alpha\subset C_{X,Y}\subset Y\times \C^{n-t}$ with respect
      to the canonical analytic projection $Y\times \C^{n-t}\to Y$ restricted to $V_\alpha$, and all the fibers
      of the restriction $\xi:D_\alpha \to Y$ have the same dimension near $0$.
      \end{theorem}
      In particular, Whitney's conditions are equivalent to the equidimensionality of the fibers of the map $D_\alpha\to Y$, plus the fact that each $D_\alpha$ is contained in $Y\times\P^{n-1-t}\times \check \P^{n-1-t}$, where $\check\P^{n-1-t}$ is the space of hyperplanes containing the tangent space $T_{Y,0}$, and the contact form on $\P^{n-1-t}\times \check \P^{n-1-t}$ vanishes on the smooth points of $D_\alpha(y)$ for $y\in Y$. This means that each $D_\alpha$ is $Y$-Lagrangian and is equivalent to a relative
      (or fiberwise) duality:
      \[\xymatrix{ D_\alpha \ar[r]\ar[d] &   W_\alpha = Y{\rm- dual\:of}\:  V_\alpha\subset Y\times\check \P^{n-1-t}\\
                   Y\times\P^{n-1-t }\supset V_\alpha\ \ \ \ \ \ \ \ \ \ \ \ \ \ \ \ \ &}\]

      The proof uses that the Whitney conditions are stratifying, and that theorem \ref{PE} and the result of remark \ref{Intdep} imply\footnote{The proof of this in \cite{L-T2} uses a lemma, p.559, whose proof is incorrect, but easy to correct. There is an unfortunate mixup in notations. One needs to prove that $\sum_{t+1}^N\xi_kdz_k=0$ and use the fact that the same vector remains tangent after the homothety $\xi_k\mapsto \lambda\xi_k,\ t+1\leq k\leq N$. Since we want to prove that $L_1$ is $Y$-Lagrangian, we must take $dy_i=0$. } that $D_\alpha$ is the conormal of its image over a dense open set of $Y$. The condition ${\rm dim} \; \xi^{-1}(0)=n-2-t$ then gives exactly
      what is needed, in view of Proposition \ref{SpecLag}, for $D_\alpha$ to be $Y$-Lagrangian.\par\noindent  Finally, we want to state another result relating
      Whitney's conditions to the dimension of the fibers of some related maps. A complete proof of 
      this result can be found in \cite[Prop. 2.1.5 and Cor.2.2.4.1]{L-T2}.
      
      \begin{corollary} Using the notations above we have:
      \begin{itemize}
            \item[1)] The pair $(X^0,Y)$ satisfies Whitney's conditions at $0$ is and only if for each
                      $\alpha$ the dimension of the fibers of the projection $W_\alpha \to Y$ is locally constant 
		      near $0$.          
            \item[2)] The pair $(X^0,Y)$ satisfies Whitney's conditions at $0$ is and only if for each
                      $\alpha$ the dimension of the fibers of the projection $V_\alpha \to Y$ is locally constant 
		      near $0$. 
      \end{itemize}           
      \end{corollary}   
      \begin{remark} The fact that the Whitney conditions, whose original definition translates as the fact that $\xi^{-1}(Y)$ is in $Y\times\P^{n-1-t}\times\check{\P}^{n-1-t}$ and not just $Y\times\P^{n-1-t}\times\check{\P}^{n-1}$ (condition a) and moreover lies in the product $Y\times I$ of $Y$ with the incidence variety $I\subset \P^{n-1-t}\times\check{\P}^{n-1-t}$ (condition b)), are in fact of a Lagrangian, or Legendrian, nature, explains their stability by general sections (by nonsingular subspaces containing $Y$) and linear projections.
      \end{remark}
\noindent      \textbf{Problem:} The fact that the Whitney conditions are of an algebraic nature, since they can be translated as an equimultiplicity condition for polar varieties by theorem \ref{PE} leads to the following question: given a germ $(X,x)\subset (\C^n,0)$ of a reduced complex analytic space, endowed with its minimal Whitney stratification, does there exist a germ $({\mathcal Y},0)\subset (\C^N,0)$ of an \textit{algebraic} variety and a germ $({\mathcal H},0)\subset (\C^N,0)$ of a nonsingular \textit{analytic variety} transversal to the stratum of $0$ in the minimal Whitney stratification of $({\mathcal Y},0)$ such that $(X,0)$ with its minimal Whitney stratification is analytically isomorphic to the intersection of  
 $({\mathcal Y},0)$, with its minimal Whitney stratification, with $({\mathcal H},0)$ in $(\C^N,0)$? 
\section{Whitney stratifications and tubular neighborhoods}\label{tub}
In differential geometry, a very useful tool is the existence of a tubular neighborhood of a closed submanifold $X$ of a differentiable manifold $Z$. It is a diffeomorphism, inducing the identity on $X$, from a neighborhood of $X$ in its normal bundle in $Z$ to a neighborhood of $X$ in $Z$; here $X$ is viewed as the zero section of its normal bundle. If $X$ is a point $x$, it is just a diffeomorphism from a neighborhood of the origin in the tangent space $T_{Z,x}$ to a neighborhood of $x$ in $Z$. In this sense a tubular neighborhood of $X$ in $Z$ is a linearization "transversally to $X$" of a neighborhood of $X$ in $Z$. Since the normal bundle $T_{Z,X}$ is a fiber bundle one can choose a positive definite quadratic form $g_x(u,u)$, or metric, on its fibers depending differentiably on the points $x\in X$ and if one chooses some differentiable function $\epsilon (x)$ on $X$ which is everywhere $>0$ one can carry over via the diffeomorphism the tube in the normal bundle defined by $g_x(u,u)\leq \epsilon(x)$, to get a tube $T_\epsilon \subset Z$ with core $X$. The natural projection to $X$ in the normal bundle carries over to a retraction $\rho\colon T_\epsilon\to X$ and $\epsilon(y)$ defines a radius of the tube, or distance to $X$ from the frontier of the tube. All this carries over to the complex analytic case, replacing metric by hermitian metric. \par
When $X$ is singular the situation becomes more complicated, but Thom and Mather (see \cite{Th} and the excellent exposition in \cite[\S 6]{Ma2}) discovered that a Whitney stratification of a nonsingular space $Z$ such that $X$ is a union of strata allows one to build an adapted version of tubular neighborhoods of $X$ in $Z$. \par\noindent
Let $X=\bigcup_{\alpha\in A} X_\alpha$ be a closed Whitney stratified subset of a chart $\R^n$ of $Z$, where the $X_\alpha$ are differentiable submanifolds. There exists a family of triplets $(T_\alpha,\pi_\alpha,\rho_\alpha)$ such that:
\begin{itemize}
\item $T_\alpha$ is the intersection with $X$ of a tubular neighborhood of $X_\alpha$ in $\R^n$.
\item The map $\rho_\alpha\colon T_\alpha\to X_\alpha$ is a $C^\infty$ retraction; in particular $\rho_\alpha(x)=x$ for $x\in X_\alpha$.
\item The function $\delta_\alpha\colon T_\alpha\to \R_{\geq 0}$ is a $C^\infty$ function (the distance to the stratum) such that $\delta_\alpha^{-1}(0)=X_\alpha$.
\item Whenever $X_\alpha\subset\overline X_\beta$, the restriction of $(\rho_\beta,\delta_\beta)$ to $T_\alpha\cap X_\beta$ is a $C^\infty$ submersion $T_\alpha\cap X_\beta\rightarrow X_\beta\times\R_{\geq 0}$.
\item  Whenever $X_\alpha\subset\overline X_\beta$, we have for all $x\in T_\alpha\cap T_\beta$ the inclusion $\rho_\beta (x)\in T_\alpha$ and the equalities  $\rho_\alpha ( \rho_\beta (x))=\rho_\alpha (x)$ and $\delta_\alpha (\rho_\beta (x))=\delta_\alpha(x)$.
\end{itemize}
All this says that you have a tubular neighborhood of each stratum in such a way that when you approach the frontier of a stratum, you enter the tubular neighborhood of a frontier stratum in a way which is compatible with the tubular neighborhood of that stratum.\par A careful description of this in the complex analytic case can be found in \cite{Sc}, \cite{Sc2} if you think of radial vector fields as transversal to the boundaries of tubular neighborhoods. By taking the viewpoint of "fundamental systems of good neighborhoods" as in \cite[Definition 2.2.9]{L-T3}, one should obtains a more analytic version; in particular the function $\delta$ could be taken to be subanalytic. However, it seems the last condition for tubular neighborhoods may be too strict to be realized in the complex analytic case and one might think of weaker conditions such as the existence, locally on $X$, of constants $C_1,C_2$ such that $\vert\delta_\alpha(\rho_\beta (x))-\delta_\alpha(x)\vert<C_1\delta_\beta(x)$ and $\Vert \rho_\alpha(\rho_\beta(x))-\rho_\alpha(x)\Vert<C_2\delta_\beta(x)$. This means that one requests that the last condition for tubular neighborhoods should only be satisfied asymptotically as one approaches each $X_\beta$. \par\noindent
After Example \ref{conelike} we stated that Whitney conditions mean that $X$ is locally "cone like" along each stratum $X_\alpha$. Conicity indeed suggests the existence of tubular neighborhoods as above, since we can expect that there are "tubes" everywhere transversal to the cones, but nevertheless the existence of tubular neighborhoods of a Whitney stratified set, due to Thom and Mather in the differentiable framework, is quite delicate to prove. Again, see \cite[\S 6]{Ma2}. By methods also due to Thom and Mather, the tubular neighborhoods provide the local topological triviality along the strata which we shall see in the next section. We note that, compared to the distance in $\R^n$ or $\C^n$ to the stratum $X_\alpha$, the radius of the tubular neighborhood $T_\beta$ must in general tend to zero as we approach the frontier of $X_\beta$; see \cite{Sc}. Recently, strong equisingularity theorems in the same general direction have been proved by Parusi\~nski and Paunescu in the complex analytic case; see \cite{P-P}.
\begin{exercise} 1) Describe the minimal Whitney stratification and associated tubular neighborhoods for the singularity $x_1.\ldots x_k=0$ in $\C^n$.\par\noindent
2) Describe the minimal Whitney stratification (in the real and in the complex case) and associated tubular neighborhoods for the singularity of example \ref{sur}.\par\noindent
3) Do the same for the surface $y^2-tx^2=0$. What is the difference?
\end{exercise}
\noindent The purpose of this section is to introduce the following:\par\medskip\noindent
\textbf{Problem:} By Proposition \ref{propnc}, iii) we know that for $X=\bigcup_\alpha X_\alpha\subset\C^n$, if some polar variety $P_k(\overline X_\beta)$ is not equimultiple along a stratum $X_\alpha$ at a point $x\in  X_\alpha$ then for a general subspace $W$ of codimension $d_\beta-d_\alpha-k$ locally containing $X_\beta$, the intersection with the polar variety $P_k(\overline X_\beta)$ is of the same dimension as $X_\alpha$ but not set theoretically equal to $X_\alpha$ near $x$. Let us assume that we have embedded $X$ in $\C^n$ in such a way that $X_\alpha$ is a linear subspace and we consider linear subspaces $W$ of $\C^n$ of codimension $d_\beta-d_\alpha-k$ containing $X_\alpha$. If in a neighnorhood of $x$ we take a tube $T_\epsilon $ around $X_\alpha$ in $\C^n$ whose radius (distance of the frontier to $X_\alpha$ in $\C^n$)  tends to zero fast enough as we approach $x$, it will not meet the extra components of the intersection $W\cap P_k(\overline X_\beta)$, i.e., those which do not coincide with $X_\alpha$.\par
The problem is to determine whether this condition for tubes of not meeting the extra components of the intersections of general linear spaces $W$ of the right codimension with the non-equimultiple  polar varieties, for each pair of strata $X_\alpha\subset\overline X_\beta$, plus the requirement of not meeting the closures $\overline X_\gamma$ of the strata such that $X_\alpha$ is \textit{not} contained in $ \overline X_\gamma$, plus some adjustment of the retractions, is sufficient to provide a system of tubular neighborhoods in the weaker sense mentioned above. A possible approach is to use the part of \cite{L-T3} already mentioned. Other probably useful references are the books of Marie-H\'el\`ene Schwartz \cite{Sc} (where tubular neighborhoods appear as mentioned above) and \cite{Sc2}.\par\medskip\noindent
This problem is related to another one. An embedding $X\subset U\subset\C^n$, with $U$ open in $\C^n$, determines a metric on $X$, called the \textit{outer metric} $d(x,y)$, which is the restriction of the ambient metric. A homeomorphism $F\colon X'\to X$ between two metrized spaces is called bilipschitz if there exists a constant $C>0$ such that $\frac{1}{C}d(x,y)\leq d(F(x),F(y))\leq Cd(x,y)$. Following Neumann-Pichon in \cite{N-P}, we say that the \textit{Lipschitz geometry} of $X$ is its geometry up to bilipschitz homeomorphism; it is independent of the embedding $X\subset\C^n$. See \cite{N-P}, \cite{Ga7}.\par\medskip\noindent
\textbf{Problem:}\textit{ Does the Lipschitz geometry of a complex analytic space $X$ determine its minimal Whitney stratification in the sense that a bilipschitz homeomorphism between complex analytic spaces endowed with their minimal Whitney stratifications must carry strata to strata?} And if that is true, with what additional structure does one need to enrich the minimal Whitney stratification of $X$ in order to determine its Lipschitz geometry? For example optimal shrinking rates of the radii of tubular neighborhoods $T_\beta$ as functions of $\delta_\alpha$ whenever $X_\alpha\subset \overline X_\beta$, plus the local Lipschitz geometry of sections of the $X_\beta$ by nonsingular spaces transversal to the $X_\alpha\subset \overline X_\beta$? The results of \cite{B-H} and  \cite{N-P} indicate that one may have to refine the stratification in order for this transversal Lipschitz geometry to be constant along the strata.\par\noindent
There are encouraging results in the direction of the first question: by \cite{S} a bilipschitz homeomorphism must send nonsingular points to nonsingular points so that the first stratum (or strata if we insist that strata should be connected) of the minimal Whitney stratification is (are) preserved. It must induce bilipschitz homeomorphisms of the tangent cones by \cite{BFLS}, and preserve multiplicities at least in the case of hypersurfaces by \cite{F-S}. The general question of whether a bilipschitz homeomorphism preserves the Lipschitz geometry of general hyperplane sections seems to be open.\par
One way of understanding how much geometric information is lost in the singular case by taking homology or cohomology classes of local polar varieties is to make precise the idea that as the linear projections vary, some branches of the polar curves of normal surfaces having a fixed tangent, for example, span special regions of the surface in the sense of the geometric decomposition of \cite{N-P}. The generic contact of the branches with their common fixed tangent, or at least some weaker version of this contact, is an invariant of the Lipschitz geometry of the surface. This follows from the work of Neumann-Pichon in \cite{N-P} and gives hints for the solution of the problems just mentioned.

      \section{Whitney stratifications and the local total topological type}\label{total}
      \textbf{Warning} \textit{In this section and section \ref{duality}, we modify the notation for polar varieties; the general linear space defining each polar variety becomes implicit, while the point at which the polar variety is defined appears in the notation $P_k(X,x)$.}\par\medskip
      We have seen how to associate to a reduced equidimensional germ $(X,x)$ of a $d$-dimensional complex analytic space a generalized multiplicity (recall that $(X,x)=P_0(X,x)$):
      $$(X,x)\mapsto \left(m_x(X,x),m_x(P_1(X,x)),\ldots ,m_x(P_{d-1}(X,x))\right).$$
      We know from subsection \ref{mul} that the multiplicity $m_x(X)$ of a reduced germ $(X,x)$ of a $d$-dimensional complex analytic space has a geometric interpretation as follows: given a local embedding $(X,x)\subset (\C^n,0)$ there is a dense Zariski open set $U$ of the Grassmannian of $n-d$-dimensional linear subspaces $L\subset \C^n$ such that for $L\in U$, with equation $\ell (z)=0$, there exists $\epsilon >0$ and $\eta (\epsilon ,\ell)  >0$ such that the affine linear space $L_{t'}=\ell ^{-1}(t')$ intersects $X$ transversally in $m_x(X)$ points inside the ball $\B (0,\epsilon)$ whenever $0<\vert t'\vert <\eta (\epsilon ,\ell) $. Taking $t\in \B (0,\epsilon)$ such that $\ell (t)=t'$, we can write $L_{t'}$ as $L+t$.
      
%\scalebox{0.4}{\includegraphics{Mult.pdf}}
 \begin{figure}[!ht]
    \begin{center}
    \label{schematic3}
              \includegraphics*{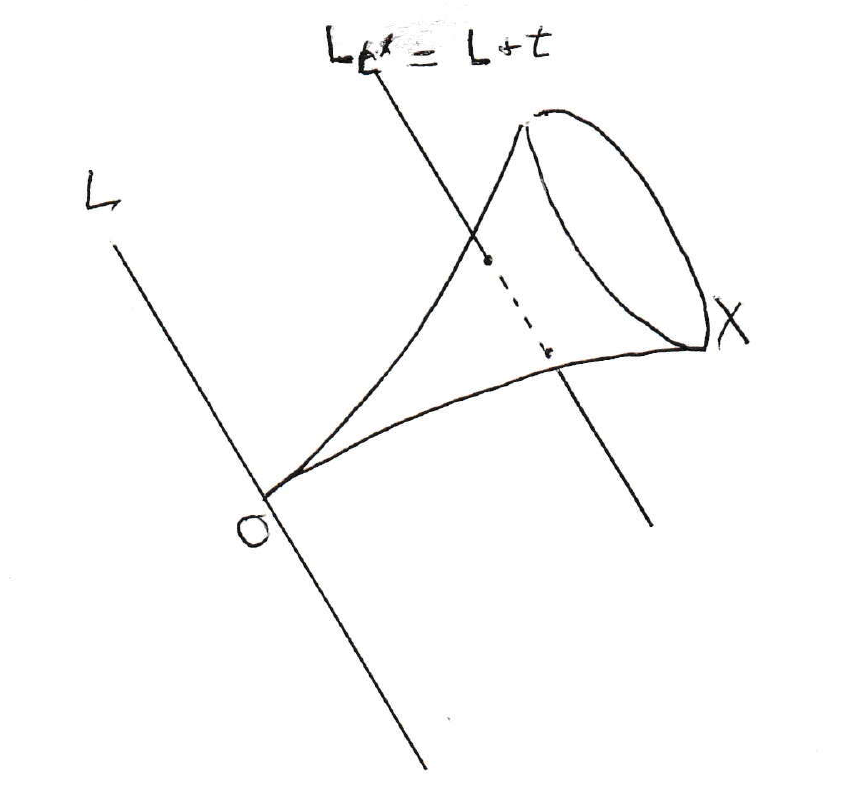}
    % \caption{ Sing $X=$ eje z; $C_{X,0}=$ plano yz}
    \end{center}
\end{figure}
      \par
      
      \newpage
      
      We may ask whether there is a similar interpretation of the other polar multiplicities in terms of the local geometry of $(X,x)\subset (\C^n,0)$. The idea, as in many other instances in geometry, is to generalize the number of intersection points ${\rm card}\{L_t\cap X\}$ by the Euler-Poincar\'e characteristic $\chi (L_t\cap X)$ when the dimension of the intersection is $>0$ because the dimension of $L_t$ is $>n-d$.\par\noindent
      
 \begin{proposition}{{\rm (L\^e-Teissier, see \cite[\S 3]{L-T3})}}\label{HOT} Let $X=\bigcup_\alpha X_\alpha$ be a Whitney stratified complex analytic set of dimension $d$. Given $x\in X_\alpha$, choose a local embedding $(X,x)\subset (\C^n,0)$. Set $d_\alpha={\rm dim} X_\alpha$. For each integer $i\in [d_\alpha+1,d]$ there exists a Zariski open dense subset $W_{\alpha,i}$ in the Grassmannian $G(n-i,n)$ and for each $L_i\in W_{\alpha,i}$ a semi-analytic subset $E_{L_i}$ of the first quadrant of $\R^2$, of the form $\{(\epsilon, \eta)\vert  0<\epsilon <\epsilon_0, 0<\eta <\phi(\epsilon)\}$ with $\phi (\epsilon)$ a certain Puiseux series in $\epsilon$, such that the homotopy type of the intersection $X\cap (L_i+t)\cap \B (0,\epsilon)$ for $t\in \C^n$ is independent of $L_i\in W_{\alpha,i}$ and $(\epsilon, t)$ provided that $(\epsilon, \vert t\vert )\in E_{L_i}$.
Moreover, this homotopy type depends only on the stratified set $X$ and not on the choice of $x\in X_\alpha$ or the local embedding.   In particular the Euler-Poincar\'e characteristics $\chi_i(X,X_\alpha)$ of these homotopy types are invariants of the stratified analytic set $X$.\end{proposition}
\begin{definition}\label{EUVA} The Euler-Poincar\'e characteristics $\chi_i(X,X_\alpha),\ i\in [d_\alpha+1,d]$ are called the local vanishing Euler-Poincar\'e characteristics of $X$ along $X_\alpha$.	
\end{definition}

\begin{corollary}{{\rm (Kashiwara; see \cite{K1}, \cite{K2})}}\label{Euler} The Euler-Poincar\'e characteristics $\chi(X,X_\alpha )=\chi_{d_\alpha+1}(X,X_\alpha)$ of the corresponding homotopy types when $i=d_\alpha+1$ depend only on the stratified set $X$ and the stratum $X_\alpha$.\end{corollary}
The invariants  $\chi(X,X_\alpha )$ appeared for the first time in \cite{K1}, in connection with Kashiwara's index theorem for maximally overdetermined systems of linear differential equations.\par\noindent Note that if the codimension of the affine spaces is $\leq d_\alpha$ they meet $X_\alpha$ so that the intersection we study is contractible by Whitney's Lemma \ref{wl}, and if the codimension is $>d$ the intersection with $ X$ is empty.

\begin{ex}\label{smoothing}\em\begin{itemize}\par\noindent
\item Let $d$ be the dimension of $X$. Taking $X_\alpha=\{x\}$, which is permissible by Whitney's lemma (Lemma \ref{wl}), and $i=d$ gives $\chi_d (X,\{x\})=m_x(X) $, as we saw above.
\item Assume that $(X,x)\subset (\C^{d+1},0)$ is a hypersurface with isolated singularity at the point $x$ (taken as origin in $\C^{d+1}$), defined by $f(z_1,\ldots ,z_{d+1})=0$. By Whitney's lemma (Lemma \ref{wl}), in a sufficiently small neighborhood of $x$, the minimal Whitney stratification (see the end of section \ref{WhitneyinNormalConormal}) is $(X\setminus\{x\})\cup \{x\}$, and we have $$\chi_i(X,\{x\})=1+(-1)^{d-i}\mu^{(d+1-i)}(X,x),\eqno{(*)}$$ where $\mu^{(k)}(X,x)$ is the Milnor number of the restriction of the function $f$ to a general linear space of dimension $k$ through $x$.\par
Let us recall that the Milnor number $\mu^{(d+1)}(X,x)$ of an isolated singularity of hypersurface as above is defined algebraically as the multiplicity in $\C\{z_1,\ldots ,z_{d+1}\}$ of the Jacobian ideal $j(f)=(\frac{\partial f}{\partial z_1},\ldots ,\frac{\partial f}{\partial z_{d+1}})$, which is also the dimension of the $\C$-vector space $\frac{\C\{z_1,\ldots ,z_{d+1}\}}{j(f)}$ since in this case the partial derivatives form a regular sequence. Topologically it is defined by the fact that for $0<\vert \lambda\vert<<\epsilon<<1$ the \textit{Milnor fiber} $f^{-1}(\lambda)\cap \B(0,\epsilon)$ has the homotopy type of a bouquet of $\mu^{(d+1)}(X,x)$ spheres of dimension $d$. In fact this is true of any \textit{smoothing} of $(X,x)$ that is, any nonsingular fiber in an analytic family $F(v,z_1,\ldots ,z_{d+1})=0$ with $F(0,z_1,\ldots ,z_{d+1})=f(z_1,\ldots ,z_{d+1})$, within a ball $\B(0,\epsilon)$ and for $0<\vert v\vert<< \epsilon<<1$. This is a consequence of the fact that the basis of the miniversal deformation of an isolated singularity of hypersurface (or more generally, complete intersection) is nonsingular, and thus irreducible, and the smooth fibers are the fibers of a locally trivial fibration over the (connected) complement of the discriminant; see \cite[ \S 4]{Te2}. Since $f^{-1}(0)\cap \B(0,\epsilon)$ is contractible the Milnor fiber has $\mu^{(d+1)}(X,x)$ \textit{vanishing cycles} of dimension $d$. For all this, see \cite{Mi}.\par\noindent
Moreover (see \cite[Chap. I]{Te1}), the Milnor number of the restriction of the function $f$ to a general $i$-dimensional linear space through $0$ is well defined and does not depend on the choice of the embedding $(X,x)\subset (\C^{d+1},0)$ or the general linear space in $\C^{d+1}$ but only on the analytic algebra $\Oo_{X,x}$. It is denoted by $\mu^{(i)}(X,x)$. Note that $\mu^{(1)}(X,x)$ is the multiplicity of $(X,x)$ minus one, and $\mu^{(0)}(X,x)=1$. \par
These numbers are related to limits at $x$ of tangent hyperplanes to the hypersurface $X$ by the following result found in \cite[Chap.II, 1.6]{Te1}: For all hyperplanes $(H,0)\subset (\C^{d+1},0)$ we have $\mu (X\cap H,x)\geq \mu^{(d)}(X,x)$ and equality holds if and only if $H$ is not a limit at $x$ of tangent hyperplanes to $X$. Here $\mu (X\cap H,x)=\infty$ if $(X\cap H,x)$ is not an isolated singularity. This result has been generalized by Gaffney to isolated singularities of complete intersections in \cite[Proposition 2.6]{Ga3} and more general situations in \cite[Theorem 3.3]{Ga4}, \cite{Ga5}, \cite[pp.129-130]{Ga6}. For non isolated singularities a criterion in terms of Segre numbers is given in \cite[Theorem 4.13]{Ga-Ga}.  Another generalization,  in a more topological framework, to a large class of Whitney stratified complex analytic spaces containing isolated singularities, is due to M. Tib\u ar in \cite{Ti}.\par\medskip
Let $p\colon E\to\check \P^{n-1}$, with $E\subset \check \P^{n-1}\times\C^n$, be the tautological bundle of the projective space $\check \P^{n-1}$ of hyperplanes in $\C^n$; given $H\in \check \P^{n-1}$, the fiber $p^{-1}(H)\subset\C^n$ is the hyperplane $H\subset \C^n$. Starting with our germ of hypersurface $(X,0)\subset (\C^n,0)$ with isolated singularity, let us consider the intersection $\mathcal{H}=(\check \P^{n-1}\times X)\cap E$. The germ of $\mathcal{H}$ along $Y=\check \P^{n-1}\times\{0\}$, endowed with the projection $p_1\colon \mathcal{H}\to \check \P^{n-1}$ induced by $p$, is the family of hyperplanes sections of $(X,0)$. It is shown in \cite[Appendice]{Te8} that the open subset of $Y$ where the Milnor number of the corresponding fiber of $p_1$ is minimal, and thus equal to $\mu^{(d)}(X,0)$, coincides with the subset where $\mathcal{H}^0$ satisfies the Whitney conditions along $Y$. This shows that the family of hyperplane sections is quite special since in general the constancy of the Milnor number in a family of isolated hypersurface singularities does not imply the Whitney conditions (see example \ref{B-S} above).
\par\medskip
Let us now prove the equality $(*)$. By the results of \cite[Chap. II]{Te1}, it suffices to prove the equality for $i=1$. We know by Proposition \ref{dimC} that a general hyperplane $L_1$ through $x$ is not a limit of tangent hyperplanes to $X$ at nonsingular points. Thus, if $0<\vert t\vert<<\epsilon$, the intersection $X\cap (L_1+t)\cap \B(0,\epsilon)$ is nonsingular because it is a transversal intersection of nonsingular varieties. For the same reason, the intersection $L_1\cap X\cap \B(0,\epsilon)$ is nonsingular outside of the origin, which means that the hypersurface $f(0,z_2,\ldots ,z_{d+1})=0$ has an isolated singularity at the origin. Choosing coordinates so that $L_1$ is given by $z_1=0$, we see that the intersections with a sufficiently small ball $\B(0,\epsilon)$ around $x$ of $f(t,z_2,\ldots ,z_{d+1})=0$ and $f(0,z_2,\ldots ,z_{d+1})=\lambda$, for small $\vert t\vert,\ \vert\lambda\vert$, are two smoothings of the hypersurface with isolated singularity $f(0,z_2,\ldots ,z_{d+1})=0$. They are therefore diffeomorphic and thus have the same Euler characteristic. The first one is our $\chi_1(X,\{x\})$ and the second one is the Euler characteristic of a Milnor fiber of $f(0,z_2,\ldots ,z_{d+1})$, which is $1+(-1)^{d-1}\mu^{(d)}(X,x)$ in view of the bouquet description recalled above.
\end{itemize}
\end{ex}
It is known from \cite[4.1.8]{L-T1} (see also just after theorem \ref{formula} below) that the image of a general polar variety $P_k(X,x)$ by the projection $p\colon (\C^n,0)\to (\C^{d-k+1},0)$ which defines it has at the point $p(x)$ the same multiplicity as $P_k(X,x)$ at $x$. This is because for a general projection $p$ the kernel of $p$ is transversal to the tangent cone $C_{P_k(X,x),x}$ of the corresponding polar variety.  Using this in the case of isolated singularities of hypersurfaces, it is known from (\cite[Chap. II, proposition 1.2 and cor. 1.4]{Te1} or \cite[corollary p.610]{Te7} that the multiplicities of the polar varieties can be computed from the $\mu^{(k)}(X,x)$; we have the equalities\footnote{The fact that the constancy of the numbers $\mu^{(i)}(X_t,0)$ in a family $(X_t,0)_{t\in \D}$ of germs of hypersurfaces with isolated singularities is equivalent to the Whitney conditions along the singular locus follows from these equalities and Theorem \ref{PE}. See \cite[Chap. VI]{Te3}. It has been stressed, in particular by T. Gaffney in \cite{Ga6} and Gaffney-Kleiman in \cite{G-K}, that this is a condition bearing only on the fibers of the family, and not its total space.} \[m_x(P_k(X,x))= \mu^{(k+1)}(X,x)+\mu^{(k)}(X,x).\]
At this point it is important to note that the equality $m_x(P_{d-1}(X,x))=\mu^{(d)}(X,x)+\mu^{(d-1)}(X,x)$ which, by what we have just seen, implies the equality \[ \chi_1(X,\{x\})-\chi_2(X,\{x\})=
 (-1)^{d-1}m_x(P_{d-1}(X,x)),\]
 implies the general formula 
\[ \chi_{d-k}(X,\{x\})-\chi_{d-k+1}(X,\{x\})=
 (-1)^k m_x(P_k(X,x)),\]
simply because an affine space $L_{d-k}+t$ can be viewed as the intersection of an $L_1+t$ for a general $L_1$ with a general vector subspace $L_{d-k-1}$ of codimension $d-k-1$ through the point $x$ taken as origin of $\C^n$, and $$m_x(P_k(X,x))=m_x(P_k(X,x)\cap L_{d-k-1})=m_x(P_k(X\cap L_{d-k-1}),x)).$$ The  first equality follows from general results on multiplicities since $L_{d-k-1}$ is general, and the second from general results on local polar varieties found in (\cite[ 5.4]{Te3}, \cite[4.18]{L-T1}). This sort of argument is used repeatedly in the proofs.\par\medskip The formula for a general stratified set is the following:

\begin{theorem}\label{formula}{\rm (L\^e-Teissier, see \cite[th\'eor\`eme 6.1.9]{L-T1}, \cite[ 4.11]{L-T3})} With the conventions just stated, and for any Whitney stratified complex analytic set $X=\bigcup_\alpha X_\alpha\subset \C^n$, we have for $x\in X_\alpha$ the equality
\begin{align*} \chi_{d_\alpha +1}(X,X_\alpha)&-\chi_{d_\alpha +2}(X,X_\alpha)=\\
&\sum_{\beta\neq\alpha} (-1)^{d_\beta-d_\alpha-1}m_x(P_{d_\beta-d_\alpha-1}(\overline{X_\beta},x))(1-\chi_{d_\beta+1}(X,X_\beta)),
\end{align*}
where it is understood that $m_x(P_{d_\beta-d_\alpha-1}(\overline{X_\beta},x))=0$ if $x\notin P_{d_\beta-d_\alpha-1}(\overline{X_\beta},x)$.
\end{theorem}
The main ingredients of the proof are the topological properties of \textit{descriptible} maps between stratified spaces (see \cite[\S 2]{L-T3}) and the transversality theorem already mentioned above which states that the kernel of the projection defining a polar variety $P_k(X,L)$ is transversal to the tangent cone $C_{P_k(X.L),0}$ at the origin provided that the projection is general enough. Thus, the image of that polar variety by this projection, a hypersurface called the polar image, has the same multiplicity as the polar variety (see \cite{L-T1}, 4.1.8).\par\noindent This is useful because one considers the intersections  $X\cap (L_i+t)\cap \B (0,\epsilon)$ as intersections with $X\cap \B (0,\epsilon)$ of the fibers of linear projections $\C^n\to \C^i$ over a ``general" point close to the image of the point $x\in X_\alpha$. Because we are in complex analytic geometry the variations of Euler-Poincar\'e characteristics can be computed as the number of intersection points of a general line with the polar image, which is its multiplicity. This is where the "descriptible" character of \textit{general} projections from the stratified space $X$ to $\C^i$, which lies beyond the scope of these notes, plays a key role in the computation of Euler-Poincar\'e characteristics; see \cite[Proposition 2.1.3]{L-T3}. The basic fact here is that in complex analytic geometry the complement of a closed union of strata in its "tubular neighborhood" as provided by the Whitney conditions (see section \ref{tub}), has zero Euler-Poincar\'e characteristic. In addition, the existence of fundamental systems of good neighborhoods of a point of $\C^n$ relative to a Whitney stratification also plays an important role. 
\begin{remark} A computation of vanishing Euler characteristics for isolated determinantal singularities is provided in \cite{NOT}. 
\end{remark}
Let us now go back to the definitions of stratifications and stratifying conditions (see definition \ref{stratif}). Given a complex analytic stratification $X=
\bigcup_\alpha X_\alpha$ of a complex analytic space, we can consider the following incidence conditions:
\begin{enumerate}
\item ``Punctual Whitney conditions", the incidence condition $\hat W_x(X_\alpha,X_\beta)$: For any $\alpha$, any point $x\in X_\alpha$, any stratum $X_\beta$ such that $\overline{X_\beta}$ contains $x$ and any local embedding $(X,x)\subset (\C^n,0)$, the pair of strata $(X_\beta,X_\alpha)$ satisfies the Whitney conditions at $x$.
\item ``Local Whitney conditions",  the incidence condition $W_x(X_\alpha,X_\beta)$: same as above except that the Whitney conditions must be satisfied at every point of some open neighborhood of $x$ in $X_\alpha$.
\item ``(Local Whitney conditions)$^*$": For each $\alpha$, for every $x\in X_\alpha$ and every local embedding $(X,x)\subset (\C^n,0)$, for every $i\leq n-d_\alpha$ there exists a dense Zariski open set $U_i$ of the Grassmannian $G(n-i-d_\alpha,n-d_\alpha)$ of linear spaces of codimension $i$ of $\C^n$ containing  the tangent space $T_{X_\alpha,x}$ such that for every germ of nonsingular subspace $(H_i,x)\subset (\C^n,0)$ of codimension $i$ containing $(X_\alpha,x)$ and such that $T_{H_i,x}\in U_i$, we have $W_x(X_\alpha,X_\beta\cap H_i)$.
\item ``Local Topological equisingularity", the incidence condition $(TT)_x$: For any $\alpha$, any point $x\in X_\alpha$, any stratum $X_\beta$ such that $\overline{X_\beta}$ contains $x$ and any local embedding $(X,x)\subset (\C^n,0)$, there exist germs of retractions $\rho\colon (\C^n,0)\to (X_\alpha,x)$ and positive real numbers $\epsilon_0$ such that for all $\epsilon,\ 0<\epsilon\leq\epsilon_0$ there exists $\eta_\epsilon$ such that for all $\eta$,  $0<\eta\leq\eta_\epsilon$, there is an homeomorphism $\B(0,\epsilon)\cap \rho^{-1}(\B(0,\eta)\cap X_\alpha)\simeq (\rho^{-1}(x)\cap \B(0,\epsilon))\times (\B(0,\eta)\cap X_\alpha)$ which is compatible with the retraction $\rho$ and the projection to $\B(0,\eta)\cap X_\alpha$ and, for each stratum $X_\beta$ such that $\overline{X_\beta}$ contains $x$, induces an homeomorphism:
$$\overline{X_\beta}\cap\B(0,\epsilon)\cap \rho^{-1}(\B(0,\eta)\cap X_\alpha)\simeq (\overline{X_\beta}\cap\rho^{-1}(x)\cap \B(0,\epsilon))\times (\B(0,\eta)\cap X_\alpha).$$
This embedded local topological triviality, meaning that locally around $x$ each $\overline{X_\beta}$ is topologically a product of the nonsingular  $X_\alpha$ by the fiber $\rho^{-1}(x)$, in a way which is induced by a topological product structure of the ambient space, will be denoted by $TT_x(X_\alpha,X_\beta)$ for each specified $\overline{X_\beta}$.
\item ``(Local Topological equisingularity)$^*$", the incidence condition $(TT^*)_x$: For each $\alpha$, for every $x\in X_\alpha$ and every local embedding $(X,x)\subset (\C^n,0)$, for every $i\leq n-d_\alpha$ there exists a dense Zariski open set $U_i$ of the Grassmannian $G(n-i-d_\alpha,n-d_\alpha)$ of linear spaces of codimension $i$ of $\C^n$ containing  the tangent space $T_{X_\alpha,x}$ such that for every germ of nonsingular subspace $(H_i,x)\subset (\C^n,0)$ of codimension $i$ containing $(X_\alpha,x)$ and such that $T_{H_i,x}\in U_i$, we have $TT_x(X_\alpha,X_\beta\cap H_i)$.

\item ``$\chi^*$ constant": For each $\alpha$, for every $x\in X_\alpha$, every stratum $X_\beta$ such that $X_\alpha\subset\overline{X_\beta}$,  and every local embedding $(X,x)\subset (\C^n,0)$, the map which to every point $y\in X_\alpha$ in a neighborhood of $x$ associates the sequence $\chi^*( \overline{X_\beta},y)=(\chi_{d_\alpha +1}(\overline{X_\beta},\{y\}),\ldots ,\chi_{d_\beta}(\overline{X_\beta},\{y\}))$ is constant on $X_\alpha$ in a neighborhood of $x$. Recall from Proposition \ref{HOT} that $\chi_{i}(\overline{X_\beta},\{y\}))$ is the Euler characteristic of the intersection, within a small ball $\B(0,\epsilon)$ around $y$ in $\C^n$,  of $\overline{X_\beta}$ with an affine subspace of codimension $i$ of the form $L_i+t$, where $L_i$ is a vector subspace of codimension $i$ of general direction and $0<\vert t\vert<\eta$ for a small enough $\eta$, depending on $\epsilon$.

\item ``$M^*$ constant": For each $\alpha$, for every $x\in X_\alpha$, for every stratum $X_\beta$ such that $X_\alpha\subset\overline{X_\beta}$,  and every local embedding $(X,x)\subset (\C^n,0)$, the map which to every point $y\in X_\alpha$ in a neighborhood of $x$ associates the sequence $$M^*(\overline{X_\beta},y)=\big(m_y(\overline{X_\beta}), m_y(P_1(\overline{X_\beta},y)), \ldots ,m_y(P_{d_\beta -1}(\overline{X_\beta},y))\big)\in \N^{d_\beta}$$ is constant in a neigborhood of $x$.\par\noindent
This condition is equivalent to saying that the polar varieties $P_k(\overline{X_\beta},x)$ which are not empty contain $X_\alpha$ and are locally around $x$  equimultiple along $X_\alpha$.
  \end{enumerate}
\textbf{The main theorem of \cite{L-T3} (Th\'eor\`eme 5.3.1) is that for a stratification in the sense of definition \ref{stratif} all these conditions are equivalent, except 4., which we know to be weaker.}\par\noindent Theorem \ref{formula}, which relates the multiplicities of polar varieties with local topological invariants, plays a key role in the proof.\par
Recall that we saw in subsection \ref{stratifications} the result of Thom-Mather (see Theorem \ref{T-M}) that Whitney stratifications have the property 4. of local topological equisingularity defined above. We also mentioned that the converse is known to be false since Brian\c con-Speder gave in \cite{BS} a counterexample to a conjecture of \cite[Pr\'eambule]{Te1}. The result just mentioned provides among other things the correct converse.
 \section{Specialization to the Tangent Cone and Whitney equisingularity} \label{sec:specialization}
 
   Let us now re-examine the question of how much does a germ of singularity $(X,0)$ without exceptional cones resembles a cone. 
   The obvious choice is to compare it with its tangent cone $C_{X,0}$, assuming that it is reduced, and we can rephrase the question by asking
   does the absence of exceptional cones implies that $(X,0)$ is Whitney-equisingular to its tangent cone?\\
   
   	To be more precise, let $(X,0) \subset (\C^n,0)$ be a reduced germ of an analytic singularity of
  pure dimension $d$, and let $\varphi: (\X,0) \to (\C,0)$ denote the specialization of $X$ to its 
  tangent cone $C_{X,0}$. Let $\X^0$ denote the open set of smooth points of $\X$, and let $Y$ denote
  the smooth subspace $0 \times \C \subset \X$.  Our aim is to study the equisingularity of $\X$ along $Y$. More precisely, we want to determine whether the absence of exceptional cones will allow us to construct a  Whitney stratification of $\X$  
  in which the parameter axis $Y$ is a stratum.\\
  
    The first result in this direction was obtained by L\^e and Teissier in \cite[Thm 2.2.1]{L-T4} and says that for a surface $(S,0)\subset (\C^3,0)$ with a
    reduced tangent cone the absence of exceptional cones is equivalent to $\{\X^0, \mathrm{Sing} \X \setminus Y, Y\}$ being a 
    Whitney stratification of $\X$. In particular $(S,0)$ is Whitney equisingular to its tangent cone $(C_{S,0},0)$. \\

  In the general case, we only have a partial answer which we will now describe. The first step to find out if such a stratification is possible, is to verify that the pair 
  $(\X^0, Y)$ satisfies Whitney's conditions. Since $\X \setminus \X(0)$ is isomorphic to the product 
  $\C^*\times X$, Whitney's conditions are automatically verified everywhere in $\{0\} \times \C$ , with 
  the possible exception of the origin.
  
  \begin{theorem}\cite[Thm. 8.11]{Gi} \label{Equivalenciashipersuperficies}
      Let $(X,0)$ be a reduced and equidimensional germ of a complex analytic space. Suppose
   that its tangent cone $C_{X,0}$ is reduced. The following 
   statements are equivalent:
  \begin{enumerate}
 \item The germ $(X,0)$ does not have exceptional cones.
  \item The pair $(\X^0, Y)$ satisfies Whitney's condition a) at the origin.
 \item The pair $(\X^0, Y)$ satisfies Whitney's conditions a) and b) at the origin.
 \item The germ $(\X,0)$ does not have exceptional cones.
 \end{enumerate}
 \end{theorem}
  
     We would like to explain a little how one goes about proving this result.  To begin with, we know that  Whitney's condition b) is stronger than
     the condition a). The equivalence of statements 2) and 3) tells us that in this case they are equivalent for the pair of strata $(\X^0, Y)$ at the origin.
The special geometry of  $\X$ plays a crucial role in this result.
  
  \begin{proposition}\label{aimplicab} \cite[Proposition 6.1]{Gi}
     If the pair $(\X^0,Y)$ satisfies Whitney's condition a) at the origin, 
   it also satisfies Whitney's condition b) at the origin.  
  \end{proposition}
  
     \begin{remark}\label{puntosclaves} $\;$\\
          \begin{enumerate}
     \item For any point $y \in Y$ sufficiently close to $0$, the tangent cone $C_{\X,y}$ is isomorphic to $C_{X,0} \times Y$,
     and the isomorphism is uniquely determined once we have chosen a set of coordinates. The 
     reason is that for any $f(z)$ vanishing on $(X,0)$, the function 
      $F(z,v)=v^{-m}f(vz)=f_m+vf_{m+1}+ v^2f_{m+2}+ \ldots$, vanishes in $(\X,0)$ and so for any
      point $y=(\underline{0},v_0)$ with $\vert v_0\vert$ small enough the series $F(z,v-v_0)$ converges for $z,v-v_0$ small enough and the initial form of $F(z,v-v_0)$ in $\C\{z_1, \ldots, z_n,
      v-v_0\}$ with respect to the ideal $(z_1, \ldots, z_n)$ is equal to the initial form of $f$ at $0$. That is $\mathrm{in}_{(0,v_0)}F=\mathrm{in}_0f$, which is independent of $v$.
      
     \item The projectivized normal cone $\P C_{\X,Y}$ is isomorphic to $Y \times \P C_{X,0}$. 
            This can be seen from the equations used to define $\X$ (section 2, exercise \ref{ejercontan}), 
            where the initial form of $F_i$ with respect to $Y$, is equal to the initial form of $f_i$ at 
            the origin: we have $\mathrm{in}_YF_i=\mathrm{in}_0f_i$.
            
     \item  There exists a natural morphism $\omega: E_Y\X \to E_0X$, making the following diagram
         commute: 
      \[\xymatrix{ E_Y\X \ar[r]^\omega \ar[d]_{e_Y} &  E_0X \ar[d]^{e_o}\\
                   \X  \ar[r]_\phi & X }\]
        Moreover, when restricted to the exceptional divisor $e_Y^{-1}(Y)= \P C_{\X,Y}$ it induces
        the natural map $\P C_{\X,Y}= Y \times \P C_{X,0} \to \P C_{X,0}$. \\   
          Algebraically, this results from the universal property of the blowing up $E_0X$ and the following diagram:
           \[\xymatrix{ E_Y\X  \ar[d]_{e_Y} &  E_0X \ar[d]^{e_o}\\
                   \X  \ar[r]_\phi & X }\]
          Note that, for the diagram to be commutative the morphism $\omega$ must map
      the point $((v,z), [z]) \in E_Y\X \setminus \{Y \times \P^{n-1}\} \subset \X \times \P^{n-1}$ 
      to the point $((vz), [z])$ in $E_0X \subset X \times \P^{n-1} $.               
   \end{enumerate}
 \end{remark} 
     
     Now we can proceed to the proof of Proposition \ref{aimplicab}.
    
    \begin{proof} (\textit{of Proposition \ref{aimplicab}})
	
       We want to prove that the pair $(\X^0,Y)$ satisfies Whitney's condition b) at the origin. 
    We are assuming that it already satisfies condition a), so in particular we have that 
    $\zeta^{-1}(0)$ is contained in $\{0\} \times \P^{n-1} \times \check{\P}^{n-1}$. By the remarks made at the beginning 
    of section  \ref{WhitneyinNormalConormal} 
    it suffices to prove that any point $(0,l,H) \in \zeta^{-1}(0)$  is contained in the incidence variety 
    $I \subset \{0\} \times \P^{n-1} \times \check{\P}^{n-1}$. This is done by considering the normal/conormal diagram
    of $\X$ augmented by the map $\omega: E_Y\X \to E_0X$ of the remarks above and the map 
    $\psi: C(\X) \to C(X) \times \C$ defined by $((z_1,\ldots,z_n,v),(a_1:\ldots:a_n:b)) \mapsto ((vz_1,\ldots,vz_n),(a_1:\ldots:a_n),v) $

   \[\xymatrix{E_YC(\X)\ar[r]^{\hat{e}_Y}\ar[dd]^{\kappa'_\X}\ar[ddr]^\zeta & C(\X)\ar[dd]^{\kappa_\X}\ar[r]^{\psi} & 
             C(X) \times \C \\
             & &  \\
             E_Y\X\ar[r]_{e_Y} \ar[d]_{\omega}  &  \X \\
             E_0X & & }\]
     
\noindent By construction, there is a sequence $(z_m, v_m, l_m, H_m)$ in $E_YC(\X) \hookrightarrow\linebreak
   C(\X) \times_\X E_Y\X $ tending to $(0,l,H)$, where $(z_m,v_m)$ is not in $Y$. Through $\kappa'_\X$, we obtain a 
   sequence $(z_m,v_m,l_m)$ in $E_Y\X$ tending to $(0,l)$, and through $\hat{e}_Y$ a sequence $(z_m,v_m,H_m)$
   tending to $(0,H)$ in $C(\X)$. \\

   In this case the condition $a)$ means that $b=0$ and so through $\psi$ we obtain the sequence $(t_mz_m, \widetilde{H}_m)$  
   tending to $(0,\widetilde{H})$ in $C(X)$.  Analogously, both the sequence $(v_mz_m, l_m)$ obtained through the map $\omega$ and its limit $(0,l)$ are in $E_0X$. 
   Finally, Whitney's Lemma \ref{wl} tells us that in this situation we have that $l \subset \widetilde{H}$ and so the point $(0,l,H)$ is in the incidence variety.
   \end{proof}

    \begin{lemma}\label{GeneralizationNecessity}\cite[Lemma 6.4]{Gi}
   If the tangent cone $C_{X,0}$ is reduced and the pair $(\X^0, Y)$
  satisfies Whitney's condition a), the germ $(X,0)$ does not have exceptional cones.
    \end{lemma}
   \begin{proof}
       Since $(\X^0, Y)$ satisfies Whitney's condition a), by Proposition
      \ref{aimplicab} it also satisfies Whitney's condition b).
      Recall that the aur\'eole of $(\X,0)$ along $Y$ is a collection $\{V_\alpha\}$ of subcones of the normal 
      cone $C_{\X,Y}$ whose projective duals determine the set of limits of tangent hyperplanes to $\X$ at
      the points of $Y$ in the case that the pair $(\X^0,Y)$ satisfies Whitney conditions a) and b) at every point
      of $Y$ (see \cite[Thm. 2.1.1, Corollary 2.1.2, p. 559-561]{L-T2}). 
      Among the $V_\alpha$ there are the irreducible components of $|C_{\X,Y}|$. Moreover we have:
      \begin{enumerate}
	       \item By Remark \ref{puntosclaves} we have that $C_{\X,Y} = Y \times C_{X,0}$ so its 
	        irreducible components are of the form $ Y \times \widetilde{V}_\beta$ where 
	        $\widetilde{V}_\beta$ is an irreducible component of $|C_{X,0}|$.
	        \item For each $\alpha$ the projection $V_\alpha \to Y$ is surjective and all the fibers 
	         are of the same dimension (see \cite{L-T2}, Proposition 2.2.4.2, p. 570).
	        \item The hyperplane $H$ corresponding to the point $(0:0: \cdots:0:1) \in \check{\P}^{n+1}$, which is $v=0$, is transversal to $(\X,0)$ by
	         hypothesis, and so by \cite[Thm. 2.3.2, p. 572]{L-T2} the collection $\{V_\alpha \cap H\}$
	         is the aur\'eole of $\X \cap H$ along $Y \cap H$. 
	     \end{enumerate}
        
      Notice that $(\X \cap H,Y\cap H)$ is equal to $(\X(0),0)$, which is isomorphic to the tangent cone 
      $(C_{X,0},0)$ and therefore does not have exceptional cones. This means that for each $\alpha$ either
      $V_\alpha \cap H$ is an irreducible component of $C_{X,0}$ or it is empty. But the intersection cannot be 
      empty because the projections $V_\alpha \to Y$ are surjective. Finally since all the fibers of the
      projection are of the same dimension, the $V_\alpha$'s are only the irreducible components of 
      $C_{\X,Y}$.\par\noindent This means that if we define the affine hyperplane $H_v$ as the hyperplane 
      with the same direction as $H$ and passing through the point $y=(0,v) \in Y$ for $v$ small enough; 
      $H_v$ is transversal to $(\X,y)$. So we have again that the collection $\{V_\alpha \cap H_v\}$
      is the aur\'eole of $\X \cap H_v$ along $Y \cap H_v$, that is, the aur\'eole of $(X,0)$, so it does not have
      exceptional cones.
    \end{proof}
    
    At this point it is not too hard to prove the equivalence of  statements 3) and 4) of theorem \ref{Equivalenciashipersuperficies}, namely 
   that the pair $(\X^0,Y)$ satisfies both Whitney conditions at the origin if and only if the 
      germ $(\X,0)$ does not have exceptional cones (see \cite[Proposition 6.5]{Gi}).\par\noindent  The idea is that 
      on the one hand we have that the Whitney conditions imply that $(X,0)$ has no exceptional cones and $b=0$, 
      but this means that the map $\psi: C(\X) \to C(X)$ $((z,v), [a:b]) \mapsto ((vz),[a])$ is defined everywhere.
    Thus, the set of limits of tangent hyperplanes to $(\X,0)$ is just the dual of the tangent cone. On the other hand               
      since $C_{\X,0}= C_{X,0} \times \C $ the absence of exceptional cones implies $b=0$ which is equivalent to Whitney's condition a).\\
      
      The key idea to prove Whitney's condition $a)$ starting from the assumption that $(X,0)$ is without exceptional cones is to use its algebraic 
      characterization given by the second author in \cite{Te1} for the case of hypersurfaces and subsequently generalized by Gaffney in \cite{Ga1}
      in terms of integral dependence of modules.  To give an idea of how it is done let us look at the hypersurface case. \\

      If  $(X,0)\subset (\C^n,0)$ is a hypersurface  then $(\X,0) \subset (\C^{n+1},0)$ is also a hypersurface. Let us say that it is defined by $F=0, \ F \in \C\{z_1,\ldots, z_n,v\}$. 
      Note that in this case the conormal space $C(\X)$ coincides with the Semple-Nash modification and thus every arc
       \[\gamma:(\C, \C\setminus\{0\},0) \to (\X,\X^0,0)\] 
       lifts  uniquely to an arc
             \[ \tilde{\gamma}:(\C, \C\setminus\{0\},0) \to ( C(\X), C(\X^0),(0,T))\]
       given  by \[\tau \mapsto \left(\gamma(\tau) ,T_{\gamma(\tau)}\X :=
            \left(\frac{\partial F}{\partial z_0} (\gamma(\tau)): \cdots : 
              \frac{\partial F}{\partial z_n}(\gamma(\tau)) : \frac{\partial F}{\partial v} 
             (\gamma(\tau))\right) \right),\]
       and so the vertical hyperplane  $\{v=0\}$, or $(0:\cdots:0:1)$ in projective coordinates, is not a limit of tangent spaces to $\X$ at $0$     
      if and only if  $\frac{\partial F}{\partial v}$  tends  to zero at least as fast as the slowest of the other partials, that is 
           \[\mathrm{order} \, \frac{\partial F}{\partial v}(\gamma(\tau)) \geq
             \mathrm{min}_j\, \left \{ \mathrm{order} \frac{\partial F}{\partial z_j}(\gamma(\tau))\right \}, \]
    where here and below $\mathrm{"order"}$ means order as a series in $\tau$. The point is that this is equivalent  to $\frac{\partial F}{\partial v} $ being
    integrally dependent on the relative Jacobian ideal $J_\varphi:=\left< \frac{\partial F}{\partial z_j}\right>$ in the local ring $O_{\X,0}$ as 
    proved by Lejeune-Jalabert and the second author in (\cite{Lej-Te}, Thm 2.1).  True, this is not precisely what we want, but it is very close because 
    the pair $(\X^0,Y)$ satisfies Whitney's condition a) at the origin if and only if $\frac{\partial F}{\partial v}$  tends  to zero 
    faster than the slowest of the other partials,  that is :
           \[\mathrm{order} \, \frac{\partial F}{\partial v}(\gamma(\tau)) >
             \mathrm{min}_j\, \left \{ \mathrm{order} \frac{\partial F}{\partial z_j}(\gamma(\tau))\right \} \]
   and according to the definition of strict dependence stated by Gaffney and Kleiman in \cite [Section 3, p. 555]{G-K}, not only for ideals but more generally for modules, 
   this is what it means for $\frac{\partial F}{\partial v} $ to be strictly dependent on the relative Jacobian ideal $J_\varphi$ in $O_{\X,0}$.\\
   
     As for the proof, note that we already know that the pair $(\X^0,Y)$ satisfies Whitney's conditions at every point $y \in Y\setminus\{0\}$, that is,  $\frac{\partial F}{\partial v} $ 
    is strictly dependent on the relative Jacobian ideal $J_\varphi$ in $O_{\X,y}$ at all these points. That this condition carries over to the origin can be determined by 
    the principle of specialization of integral dependence (see \cite[Appendice 1]{Te6}, \cite[Chap.1, \S 5]{Te3}, \cite{G-K}) which in this case  amounts to proving that the exceptional divisor $E$ of the normalized blowing up of $\X$ along 
    the ideal $J_\varphi$ does not have irreducible components whose image in $\X$ is contained in the special fiber $\X(0):=\varphi^{-1}(0)$. Fortunately, this normalized blowing up is isomorphic to a space we know,
    namely the normalization of the relative conormal space $C_\varphi(\X)$ of \ref{Specializationproprelative}:
       \[\widetilde{\kappa_\varphi}: \widetilde{C_\varphi(\X)} \to \X,\]
     and  we are able to use the absence of exceptional cones to prove that $E$ has the desired property. \\
     This ends our sketch of proof of theorem \ref{Equivalenciashipersuperficies}.
      \begin{corollary} \cite[Corollaries 8.14 and 8.15]{Gi}\\
   Let $(X,0)$ satisfy the hypothesis of theorem \ref{Equivalenciashipersuperficies}. 
        \begin{itemize}
        \item If $(X,0)$ has an isolated singularity and its tangent cone is a complete intersection singularity, then
      the absence of exceptional cones implies that $C_{X,0}$ has an isolated singularity and  
                 $\{\X \setminus Y, Y \}$ is a Whitney stratification of $\X$.  
        \item  If the  tangent cone $(C_{X,0},0)$ has an isolated singularity at the origin, then $(X,0)$ has an 
                  isolated singularity and $\{\X \setminus Y, Y \}$ is a Whitney stratification of $\X$. 
        \end{itemize} 
      \end{corollary}

       We have verified that the absence of exceptional cones allows us to start building a Whitney 
      stratification of $\X$ having $Y$ as a stratum. The question now is how to continue. We can prove (\cite[Proposition 8.13]{Gi})
      that  in the complete intersection case, the singular locus of $\X$ coincides with the specialization space $Z$ of $|\mathrm{Sing}X|$ to its tangent cone. 

   Suppose now that the germ $(\vert\mathrm{Sing} X\vert,0)$ has a reduced tangent cone; then a stratum $\X_\lambda$ 
 containing a dense open set of $Z$ will satisfy Whitney's conditions along $Y$ if and only if the 
 germ $(|\mathrm{Sing} \, X|,0)$ does not have exceptional cones. 

   In view of this it seems reasonable to start by assuming the existence of a Whitney 
 stratification $\{X_\lambda\}$ of $(X,0)$ such that for every $\lambda$ the germ $(\overline{X_\lambda},0)$
 has a reduced tangent cone and no exceptional cones. In this case, the specialization space $\X_\lambda$ of 
 $(\overline{X_\lambda},0)$ is canonically embedded as a subspace of $\X$, and the partition of $\X$ associated
 to the filtration given by the $\X_\lambda$ is a good place to start looking for the desired Whitney
 stratification of $\X$ but this is to our knowledge still an open problem. A precise formulation is the following:\par\medskip\noindent
 \textbf{Questions:} 1) Let $(X,0)$ be a germ of reduced equidimensional complex analytic space, and let $X=\bigcup_{\lambda\in L} X_\lambda$ be the minimal (section \ref{WhitneyinNormalConormal}) Whitney stratification of a small representative of $(X,0)$. Is it true that the following conditions are equivalent?
 \begin{itemize}
 \item The tangent cones $C_{\overline X_\lambda ,0}$ are reduced and the $(\overline {X_\lambda} ,0)$ have no exceptional cones, for $\lambda\in L$.
 \item The specialization spaces $(\X_\lambda )_{\lambda\in L}$ are the closures of the strata of the minimal Whitney stratification of $\X$. If $\{0\}$ is a stratum in $X$, we understand its specialization space to be $Y=\{0\}\times\C\subset\X$. Indeed, in this case the algebra ${\mathcal R}$ of Proposition \ref{Reesalgebra} is $k[v]$.
 \end{itemize}
  If that is the case, for a sufficiently small representative $\X$ of $(\X,0)$, the spaces $(X,0)$ and $(C_{X,0},0)$ are isomorphic to the germs, at $(0,v_0)$ and $(0, 0)$ respectively, of  two transversal sections, $v=0$ and $v=v_0\neq 0$ of a Whitney stratification of $\X\subset \C^n\times \C$, and so are Whitney-equisingular. Conversely, if $Y$ is a stratum of a Whitney stratification of $\X$, it is contained in a stratum of the minimal Whitney stratification of $\X$, whose strata are the specialization spaces $\X_\lambda$ of the strata $X_\lambda$ of the minimal Whitney stratification of $X$. It follows from theorem \ref{Equivalenciashipersuperficies} that the $\overline{X_\lambda}$ have a reduced tangent cone and no exceptional cones.\par\medskip\noindent
  2) Given an algebraic cone $C$, reduced or not, which systems of irreducible closed subcones can be obtained as exceptional tangents for some complex analytic deformation of $C$ having $C$ as tangent cone?
  %%%%%%%%%%%____________
    \section{Polar varieties, Whitney stratifications, and projective duality} \label{duality}
    \textit{See the warning concerning notation at the beginning of section \ref{total}. In this section we go back and forth between a projective variety $V\subset \P^{n-1}$ of dimension $d$, the germ $(X,0)$ at $0$ of the cone $X\subset \C^n$ over $V$, and the germ $(V,v)$ of $V$ at a point $v\in V$, so that we also use the notations of section \ref{LOCPOL}. Note that $P_k(V),\ 0\leq k\leq d$ denotes the polar varieties in the sense of definition \ref{projpol}.}\par\medskip\noindent
 The formula of theorem \ref{formula} can be applied to the special singular point which is the vertex $0$ of the cone $X$ in $\C^n$ over a projective variety $V$ of dimension $d$ in $\P^{n-1}$, which we assume not to be contained in a hyperplane. The dual variety $\check V$ of $V$ was defined in subsection \ref{du}. Remember that every complex analytic space, and in particular $V$, has a minimal Whitney stratification (see the end of section \ref{WhitneyinNormalConormal}). We shall use the following facts, with the notations of Proposition \ref{HOT} and those introduced after Proposition \ref{dimC}:
    \begin{proposition}\label{projEuler}{\rm (Compare with \cite[end of \S 5]{Te4})} Let $V\subset\P^{n-1}$ be a projective variety of dimension $d$. 
    \begin{enumerate}
    \item If $V=\bigcup V_\alpha$ is a Whitney stratification of $V$, denoting by $X_\alpha\subset\C^n$ the cone over $V_\alpha$, we have that $X=\{0\}\bigcup X^*_\alpha$, where $X^*_\alpha=X_\alpha\setminus\{0\}$, is a Whitney stratification of $X$. It may be that $(V_\alpha)$ is the minimal  Whitney stratification of $V$ but $\{0\}\bigcup X^*_\alpha$ is not minimal, for example if $V$ is itself a cone.
    \item If $L_i+t$ is an $i$-codimensional affine space in $\C^n$ it can be written as $L_{i-1}\cap (L_1+t)$ with vector subspaces $L_i$ and for general directions of $L_i$ we have, denoting by $\B(0,\epsilon)$ the closed ball with center $0$ and radius $\epsilon$, for small $\epsilon$ and $0<\vert t\vert <<\epsilon$ :
    $$\chi_i (X,\{0\}):=\chi (X\cap(L_i+t)\cap \B(0,\epsilon))=\chi (V\cap H_{i-1})-\chi (V\cap H_{i-1}\cap H_1),$$ where $H_i=\P L_i\subset \P^{n-1}$.     \item For every stratum $X^*_\alpha$ of $X$, we have the equalities $\chi_i (X,X^*_\alpha)=\chi_i (V,V_\alpha)$.
    \item If the dual $\check V\subset \check\P^{n-1}$ is a hypersurface, its degree is equal to $m_0(P_d(X,0))$, which is the number of critical points of the restriction to $V$ of a general linear projection $\P^{n-1}\setminus L_2\rightarrow\P^1$. 
    \end{enumerate}
    \end{proposition}
    Note that we will apply statements 2) and 3) not only to the cone $X$ over $V$ but also to the cones $\overline{X_\beta}$ over the closed strata $\overline{V_\beta}$.

\begin{proof}The first statement follows from the product structure of the cones along their generating lines outside of the origin, and the fact that $\overline{V_\beta}\times \C$ satisfies the Whitney conditions along $V_\alpha\times\C$ at a point $(x,\lambda)\in V_\alpha\times\C^*$ if and only if $\overline{V_\beta}$ satisfies those conditions along $V_\alpha$ at the point $x$.\par To prove the second one, we first remark that it suffices to prove the result for $i=1$ since we can then apply it to $X\cap L_{i-1}$. Assuming that $i=1$ we may consider the minimal Whitney stratification of $V$ and by an appropriate choice of coordinates assume that the hyperplane of $\P ^{n-1}$ defined by $z_1=0$ is transversal to the strata. Then, we use an argument very similar to the proof of the existence of fundamental systems of good neighborhoods in \cite{L-T3}. In $\P^{n-1}$ with homogeneous coordinates $(z_1:\ldots :z_n)$, we choose the affine chart $\A^{n-1}\simeq\C^{n-1}\subset\P^{n-1}$ defined by $z_1\neq 0$. The distance function to $0\in \A^{n-1}$ is real analytic on the strata of $V$.\par\noindent  Let us denote by $\D(0,R)$ the ball centered at $0$ and with radius $R$ in $\A^{n-1}$. By Bertini-Sard's theorem and Thom's isotopy theorem, we obtain that there exists a radius $R_0$, the largest critical value of the distance function to the origin restricted to the strata of $V$,  such that the homotopy type of $V\cap \D(0,R)$ is constant for $R> R_0$ and equal to that of $V\setminus V\cap H$, where $H$ is the hyperplane $z_1=0$. Thus, $\chi(V\setminus V\cap H)=\chi(V)-\chi(V\cap H)=\chi (V\cap \D(0,R))$. In fact, by the proof of the Thom-Mather theorem, the intersection $V\cap \D(0,R)$ is then a deformation retract of $V\setminus V\cap H$.\par\noindent Since all that is required from our hyperplane $z_1=0$ is that it should be transversal to the strata of $V$, we may assume that the hyperplane $L_1$ is defined by $z_1=0$. Given $t\neq 0$ and $\epsilon$, the application $(z_1:\ldots :z_n)\mapsto (t, t\frac{z_2}{z_1},\ldots ,t\frac{z_n}{z_1})$ from $\A^{n-1}$ to $L_1+t$ maps isomorphically $V\cap \D(0,\frac{\epsilon}{\vert t\vert})$ onto $X\cap (L_1+t)\cap \B(0,\epsilon)$. It now suffices to take $\vert t\vert$ so small with respect to $\epsilon$ that $\frac{\epsilon}{\vert t\vert}>R_0$.\par The third statement follows from the fact that locally at any point of $X^*_\alpha$, the cone $X$, together with its stratification, is the product of $V$, together with its stratification, by the generating line through $x$ of the cone, and product by a disk does not change the Euler characteristic.\par
Finally, we saw in Lemma \ref{cone} that the fiber $\kappa^{-1}(0)$ of the conormal map $\kappa\colon C(X)\to X$ is the dual variety $\check V$. The last statement then follows from the very definition of polar varieties. Indeed, given a general line $L^1$ in $\check\P^{n-1}$, the corresponding polar curve in $X$ is the cone over the points of $V$ where a tangent hyperplane belongs to the pencil $L^1$; it is a finite union of lines and its multiplicity is the number of these lines, which is the number of corresponding points of $V$.\par 
\end{proof}
Using Proposition \ref{projEuler}, we can rewrite in this case the formula of theorem \ref{formula} as a generalized Pl\"ucker formula for any $d$-dimensional projective variety $V\subset\P^{n-1}$ whose dual is a hypersurface:
\begin{proposition}\label{projfor}{\rm (Teissier, see \cite[\S 5]{Te4})} Given the projective variety $V\subset \P^{n-1}$ equipped with a Whitney stratification $V=\bigcup_{\alpha \in A}V_\alpha$, denote by $d_\alpha$ the dimension of $V_\alpha$. We have, if the projective dual $\check V$ is a hypersurface in $\check\P^{n-1}$:
\begin{align*}& (-1)^d{\rm deg}\check V=\\ &\chi (V)-2\chi (V\cap H_1)+\chi (V\cap H_2)-\sum_{d_\alpha<d}  (-1)^{d_\alpha}{\rm deg}_{n-2}P_{d_\alpha}(\overline{V_\alpha})(1-\chi_{d_\alpha +1}(V,V_\alpha)),
\end{align*}
where $H_1,H_2$ denote general linear subspaces of $\P^{n-1}$ of codimension $1$ and $2$ respectively, ${\rm deg}_{n-2}P_{d_\alpha}(\overline{V_\alpha})$ is the number of nonsingular critical points of a general linear projection $\overline{V_\alpha}\to \P^1$, which is the degree of $\check{\overline{V_\alpha}}$ if it is a hypersurface and is set equal to zero otherwise. It is equal to $1$ if $d_\alpha=0$. 
\end{proposition}Here we remark that if $(V_\alpha)_{\alpha\in A}$ is the minimal Whitney stratification of the projective variety $V\subset \P^{n-1}$, and $H$ is a general hyperplane in $\P^{n-1}$, the $V_\alpha\cap H$ that are not empty constitute the minimal Whitney stratification of $V\cap H$; see \cite[Chap. III, lemma 4.2.2]{Te3} and use the fact that the minimal Whitney stratification is defined by equimultiplicity of polar varieties (see \cite[Chap. VI, \S3]{Te3}) and that the multiplicity of polar varieties of dimension $>1$ is preserved by general hyperplane sections as we saw before theorem \ref{formula}.\par
The formula, $ (-1)^d{\rm deg}\check V=\chi (V)-2\chi (V\cap H_1)+\chi (V\cap H_2)$ in the special case where $V$ is nonsingular, already appears in \cite[formula (IV, 72)]{Kl4}.\par\noindent The formula of Proposition \ref{projfor} is \textit{a priori} different in general from the very nice generalized Pl\"ucker formula given by Ernstr\"om in \cite{Er}, which also generalizes the formula (IV, 72) to the singular case, even when the dual variety is not a hypersurface:
\begin{theorem}{\rm (Ernstr\"om, see \cite{Er})}\label{ER} Let $V\subset \P^{n-1}$ be a projective variety and let $k$ be the codimension in $\check\P^{n-1}$ of the dual variey $\check V$. We have the following equality:
$$(-1)^d{\rm deg}\check V=k\chi(V,{\rm Eu}_V)-(k+1)\chi(V\cap H_1,{\rm Eu}_{V\cap H_1})+\chi(V\cap H_{k+1},{\rm Eu}_{V\cap H_{k+1}}),$$
where the $H_i$ are general linear subspaces of $\P^{n-1}$ of codimension $i$ and $\chi(V,{\rm Eu}_V)$ is a certain linear combination with coefficients in $\Z$ of Euler characteristics of subvarieties of $V$, which is built using the properties of the \textit{local Euler obstruction} ${\rm Eu}(V,v)\in\Z$ associated to any point $v$ of $V$, especially that it is \textit{constructible} i.e., constant on constructible subvarieties of $V$.\end{theorem} The local Euler obstruction is a local invariant of singularities which plays an important role in the theory of Chern classes for singular varieties, due to M-H. Schwartz and R. MacPherson (see \cite{Br}). Its definition is outside of the scope of these notes but we shall give an expression for it in terms of multiplicities of polar varieties below.

Coming back to our formula, if $\check V$ is not a hypersurface, the polar curve $P_{d-1}(X,L)$ is empty, but the degree of $\check V$ is still the multiplicity at the origin of a polar variety of the cone $X$ over $V$. We shall come back to this below.
\par\medskip\noindent
\textbf{The case where $V$ has isolated singularities}\par\medskip\noindent
Let us first treat the hypersurface case. Let $f(z_1,\ldots ,z_n)$ be a homogeneous polynomial  of degree $m$ defining a hypersurface $V\subset \P^{n-1}$ with isolated singularities, which is irreducible if $n>3$. The degree of $\check V$ is the number of points of $V$ where the tangent hyperplane contains a given general linear subspace $L$ of codimension $2$ in $\P^{n-1}$.  By Bertini's theorem we can deform $V$ into a nonsingular hypersurface $V'$ of the same degree, by considering the hypersurface defined by $F_{v_0}=f(z_1,\ldots ,z_n)+v_0z_1^m=0$, where the open set $z_1\neq 0$ contains all the singular points of $V$ and $v_0$ is small and non zero.\par\noindent
Taking coordinates such that $L$ is defined by $z_1=z_2=0$, the class of $V'$ is computed as the number of intersection points of $V'$ with the curve of $\P^{n-1}$ defined by the equations $\frac{\partial{F_{v_0}}}{\partial z_3}=\cdots =\frac{\partial F_{v_0}}{\partial z_n}=0$, which express that the tangent hyperplane to $V'$ at the point of intersection contains $L$. This is the \textit{relative polar curve} of \cite{Te3}\footnote{Or rather the projectivization of the relative polar surface of the homogeneous map $F_{v_0}\colon \C^n\to\C$. The distinction between absolute local polar varieties, which are defined as critical subsets of projections of a singular germ to a nonsingular one, and relative local polar varieties, whose nature is that of families of polar varieties of the fibers of a morphism, was established in \cite{Te3}. The idea was to extend to the local case the distinction between the polar curves used by Pl\"ucker and Poncelet to prove Pl\"ucker formulas, say for a projective plane curve, and the polar loci of Todd, which are a high dimensional generalization of the intersection of the polar curve with the given projective plane curve. This leads of course to the definition of relative Nash modification and relative conormal space for a morphism.}. For general $z_1$ this is a complete intersection and B\'ezout's theorem combined with Proposition \ref{projfor} gives $${\rm deg}\check V'=(-1)^{n-2}(\chi (V')-2\chi (V'\cap H_1)+\chi (V'\cap H_2))=m(m-1)^{n-2}.$$
\indent Now as $V'$ degenerates to $V$ when $v_0\to 0$,  by what we saw in example \ref{smoothing}, the topology changes only by $\mu^{(n-1)}(V,x_i)$ vanishing cycles in dimension $n-2$ attached to each of the isolated singular points $x_i\in V$ (example \ref{smoothing}). This gives $\chi (V)=\chi (V')-\sum_i (-1)^{n-2}\mu^{(n-1)}(V,x_i)$.\par\noindent We have $\chi (V\cap H_1)=\chi (V'\cap H_1)$ and $ \chi (V\cap H_2)=\chi (V'\cap H_2)$ since $H_1$ and $H_2$, being general, miss the singular points and are transversal to $V$ and $V'$ so that $V'\cap H_1$ (resp. $V'\cap H_2$) is diffeomorphic to $V\cap H_1$ (resp. $V\cap H_2$). It follows from a theorem of Ehresmann that all nonsingular projective hypersurfaces of the same degree are diffeomeorphic.\par\noindent We could have taken $H_2$ general in $H_1$ and $H_1$ to be $z_1=0$, and then $V\cap H_1=V'\cap H_1,\ V\cap H_2=V'\cap H_2$.\par On the other hand, in our formula the Whitney strata of dimension $<n-2$ are the $\{x_i\}$ so all the $d_{\{x_i\}}$ are equal to zero while the $\chi_1 (V,\{x_i\})$ are equal to $1+(-1)^{n-3}\mu^{(n-2)}(V,x_i)$, corresponding to the Milnor number of a generic hyperplane section of $V$ through $x_i$, as we saw in example \ref{smoothing}.\par
Substituting all this in our formula of Proposition \ref{projfor} gives:
\begin{align*}&(-1)^{n-2}{\rm deg}\check V=\\&(-1)^{n-2}m(m-1)^{n-2}-\sum_i(-1)^{n-2}\mu^{(n-1)}(V,x_i)\\&-\sum_i (1-(1+(-1)^{n-3}\mu^{(n-2)}(V,x_i)))
\end{align*}
Simplifying and rearranging we obtain:
$${\rm deg}\check V=m(m-1)^{n-2}-\sum_i(\mu^{(n-1)}(V,x_i)+\mu^{(n-2)}(V,x_i)).$$
This formula was previously established in \cite[Appendice II]{Te6} (see also \cite{La1}) by algebraic methods based on the fact that the multiplicity in the ring $\Oo_{V,x_i}$ of the Jacobian ideal is equal to $\mu^{(n-1)}(V,x_i)+\mu^{(n-2)}(V,x_i)$ (see \cite[Chap.II, \S 1]{Te1}).  This equality is called the \textit{restriction formula} because it shows that while the multiplicity of the jacobian ideal in the ambiant space at the point $x_i$ is $\mu^{(n-1)}(V,x_i)$, the multiplicity of its restriction to the hypersurface is $\mu^{(n-1)}(V,x_i)+\mu^{(n-2)}(V,x_i)$. This multiplicity is also the intersection multiplicity of the relative polar curve with the hypersurface, which counts the number of intersection points of the polar curve with a Milnor fiber of the hypersurface singularity. By the very definition of the relative polar curve, these points are those where the tangent hyperplane to the Milnor fiber is parallel to a given hyperplane of general direction or, in the projective case, belongs to a given general pencil of hyperplanes.\par\noindent This shows that the ``diminution of class"\footnote{That is, what one must add to the degree of the dual of $V$ to obtain the degree of the dual of a nonsingular variety of the same degree as $V$.}the due to the singularity is the number of hyperplanes containing $L$ and tangent to a smoothing of the singularity which are ``absorbed" by the singularity as the Milnor fiber specializes to the singular fiber; see \cite[Appendice II]{Te6}.\par\noindent See \cite[Corollary 3.8]{Pi1} for a proof in terms of characteristic classes, closer to the approach of Todd. Yet another proof, relating the degree of the dual variety to integrals of curvature and based on the relationship between (relative) polar curves and integrals of curvature brought to light by Langevin in \cite{Lan}, can be found in \cite[\S 5]{Gr}. \par
The same method works for complete intersections with isolated singula-rities, since they can also be smoothed in the same way, and the generalized Milnor numbers also behave in a similar way. Using the results of Navarro Aznar in \cite{N} on the computation of the Euler characteristics of nonsingular complete intersections and the results of L\^e in \cite{L1} on the computation of Milnor numbers of complete intersections, as well as a direct generalization of the small trick of example \ref{smoothing} for the computation of $\chi_1 (X,\{x\})$ (see also \cite{Ga2}),  one can produce a topological expression for the degrees of the duals of complete intersections with isolated singularities, in terms of the degrees of the equations and generalized Milnor numbers. We leave this to the reader as an interesting exercise. The answer, obtained by a different method, can be found in \cite[\S 2]{Kl3} and a proof inspired by \cite{Er} can also be found in \cite{M-T}. The correction term coming from the singularities has the same form as in the hypersurface case. \par In the general isolated singularities case (see \cite[\S 3]{Kl3}), both the computation of Euler characteristics and the topological interpretation of local invariants at the singularities offer new challenges.\par  \par\medskip\noindent
\textbf{Conclusion}: Given a projective variety $V$ of dimension $d$ endowed with its minimal Whitney stratification $V=\bigcup_{\alpha \in A}V_\alpha$, we can write the formula of Proposition \ref{projfor} as follows:
\begin{align*}& (-1)^d{\rm deg}\check V=\\ &\chi (V)-2\chi (V\cap H_1)+\chi (V\cap H_2)-\sum_{d_\alpha <d}  (-1)^{d_\alpha}{\rm deg}_{n-2}\check{\overline{V_\alpha}}(1-\chi_{d_\alpha +1}(V,V_\alpha)),
\end{align*} where we agree that ${\rm deg}_{n-2}\check{\overline{V_\alpha}}={\rm deg}\check{\overline{V_\alpha}}$ if ${\rm dim}\check V_\alpha=n-2$, and is $0$ if ${\rm dim}\check V_\alpha<n-2$.
Then we see by induction on the dimension that:
\begin{proposition}\label{degd} The degree of the dual variety, when it is a hypersurface, ultimately depends on the Euler characteristics of the $\overline{V_\alpha}$ (or the $V_\alpha$, since it amounts to the same by additivity of the Euler characteristic) and their general linear sections, and the local vanishing Euler characteristics $\chi_i (\overline{V_\beta},V_\alpha)$. 
\end{proposition}\noindent
\textbf{Problem:} Given $V$ as above with a defining homogeneous ideal, describe an algebraic method to produce an ideal defining the union of $V$ and the duals of the other strata of the minimal Whitney stratification of the dual $\check V$.\par\noindent
For example, the dual of a general plane algebraic curve has only cusps and double points as singularities. The construction described above adds to the curve all its "remarkable tangents", namely its double tangents and inflexion tangents.\par\medskip\noindent
Using the properties of polar varieties and theorem \ref{formula} one can prove a similar formula in the case where the dual $\check V$ is not a hypersurface, and thus extend Proposition \ref{degd} to all projective varieties. The degree of $\check V$ is then the multiplicity at the origin of the smallest polar variety of the cone $X$ over $V$ which is not empty. Now one uses the equalities $$m_x(P_k(X,x))=m_x(P_k(X,x)\cap L_{d-k-1})=m_x(P_k(X\cap L_{d-k-1}),x)),$$ which we have seen before theorem \ref{formula}. They tell us that the degree of $\check V$ is the degree of the dual of the intersection of $V$ with a linear space of the appropriate dimension for this dual to be a hypersurface.\par
More precisely, when $H$ is general hyperplane in $\P^{n-1}$, the following facts are consequences of the elementary properties of projective duality (see Remark \ref{Factsofpolarvarieties}, c)), and the property that tangent spaces are constant along the generating lines of a cone (see Lemma \ref{cone}):\begin{itemize} 
\item If $\check V$ is a hypersurface, the dual of $V\cap H$ is the cone with vertex $\check H$ over the polar variety $P_1(\check V, \check H)$, the closure in $\check V$ of  the critical locus of the restriction to $\check V^0$ of the projection $\pi\colon \check\P^{n-1}\to \check \P^{n-2}$ from the point $\check H\in \check\P^{n-1}$. Since we assume that $V$ is not contained in a hyperplane, the degree of the hypersurface $\check V$ is $\geq 2$, hence this critical locus is of dimension $n-3$ and the dual of $V\cap H$ is a hypersurface. In appropriate coordinates its equation is a factor of the discriminant of the equation of $\check V$.
\item  Otherwise, the dual of $V\cap H$ is the cone with vertex $\check H$ over $\check V$, i.e., the join $\check V*\check H$ in $\check\P^{n-1}$ of $\check V$ and the point $\check H$.
\end{itemize}Although they were suggested to us by the desire to extend Proposition \ref{degd} to the general case, these statements are not new. The authors are grateful to Steve Kleiman for providing the following references:  for the first statement, \cite[Lemma d, p.5]{Wa2}, and for the second one \cite[Thm. (4.10(a)), p.164]{HK}. One should also consult \cite{Kl5} and compare with \cite[Proposition 1.9]{Ho}, and \cite[Proposition 2.2]{A3}.\par Assuming that $H$ is general, it is transversal to the stratum $V^0$ and to verify these statements one may consider only what happens at nonsingular points of $V\cap H$, which are dense in $V\cap H$. At such a point $v\in V\cap H$, the space of hyperplanes containing the tangent space $T_{V,v}$ is of codimension one in the space of hyperplanes containing $T_{V\cap H,v}$ and does not contain the point $\check H$ since $H$ is general. Any hyperplane containing $T_{V\cap H,v}$ and distinct from $H$ determines with $H$ a pencil. Because of the codimension one, the line in $\check\P^{n-1}$ representing this pencil must contain a point representing a hyperplane tangent to $V$ at $v$. The closure in $\check\P^{n-1}$ of the union of the lines representing such pencils is the dual of $V\cap H$. It is a cone with vertex $\check H$ and because tangent spaces are constant along generating lines of cones, a tangent hyperplane to this cone must be tangent to $\check V$.\par\noindent
If ${\rm dim}\check V=n-3$ this cone is a hypersurface in $\check\P^{n-1}$. Otherwise we repeat the operation by intersecting $V\cap H$ with a new general hyperplane, and so on; we need to repeat this as many times as the \textit{dual defect} $\delta(V)={\rm codim}_{\check\P^{n-1}}\check V-1$. We can then apply Proposition \ref{degd} to $V\cap H_{\delta (V)}$ because its dual is a hypersurface. The cone $\check H*\check V$ on a projective variety $\check V$  from a point $\check H$ in $\check\P^{n-1}\setminus \check V$ has the same degree as $\check V$. To see this, remember that the degree is the number of points of intersection with a general (transversal) linear space of complementary dimension. If a general linear space $L$ of codimension ${\rm dim}\check V +1$ intersects  $\check H*\check V$ transversally in $m$ points, the cone $\check H*L$ will intersect transversally $\check V$ in $m$ points. Thus, the iterated cone construction does not change the degree so that the degree of the dual of $V\cap H_{\delta (V)}$ is the degree of $\check V$.\par\noindent The smallest non empty polar variety of the cone $(X,0)$ over $V$ is $P_{d-\delta (V)}(X,0)$. \par\medskip\noindent This suffices to show that Proposition \ref{degd} is valid in general: \textit{Obtaining a precise formula for the degree of $\check V$ in the general case is reduced to the computation of Euler-Poincar\'e characteristics and local vanishing Euler-Poincar\'e characteristics of general linear sections of $V$ and of the strata of its minimal Whitney stratification}.\par It would be interesting to compare this with the viewpoints of \cite{Pi1}, \cite{Er}, \cite{A1} and \cite{A2}. The comparison with \cite{Er} would hinge on the following two facts:
\begin{itemize}
\item By corollary 5.1.2 of \cite{L-T1} we have at every point $v\in V$ the equality
$${\rm Eu}(V,v)= \sum_{k=0}^{d-1} (-1)^km_v(P_k(V,v)).$$
\item As an alternating sum of multiplicities of polar varieties, in view of theorem \ref{PE}, the Euler obstruction is constant along the strata of a Whitney stratification.
\end{itemize}
Indeed, if we expand the formula written above Proposition \ref{degd} in terms of the Euler characteristics of the strata $V_\alpha$ and their general linear sections, and then remove the symbols $\chi$ in front of them, we obtain a linear combination of the $V_\alpha$ and their sections, with coefficients depending on the local vanishing Euler-Poincar\'e characteristics along the $V_\alpha$, which has the property that taking formally the Euler characteristic gives $(-1)^d{\rm deg}\check V$. Redistributing the terms using theorem \ref{formula} should then give Ernstr\"om's theorem. We leave this as a problem for the reader. Another interesting problem is to work out in the same way formulas for the other polar classes, or ranks (see \cite[\S 2]{Pi1}).\par\medskip Finally, by Proposition \ref{projEuler} the formula of theorem \ref{formula} appears in a new light, as containing an extension to the case of non-conical singularities of the generalized Pl\"ucker formulas of projective geometry. Interesting connections between the material presented here and the theory of characteristic classes for singular varieties are presented in \cite{Br}, \cite{Br2}, \cite{AB} and \cite{A2}. \par\medskip\noindent
\textbf{Acknowledgement}\\ The authors are grateful to Ragni Piene for useful comments on a version of this text, to Steve Kleiman and Anne Fr\"uhbis-Kr\"uger  for providing references with comments, and to the referee for his attentive reading and many very useful suggestions.
\section{APPENDIX: SOME COMPLEMENTS}
\subsection{Whitney's conditions and the condition (w)}\label{defw}

In the vector space $\C^n$ equipped with its hermitian metric corresponding to the bilimear form $(u,v)=\sum_{i=1}^n u_i\overline v_i$, let us consider two non zero vector subspaces $A,B$ and define the diatance from $A$ to $B$, in that order, as:
$${\rm dist}(A,B)={\rm sup}_{u\in B^\perp\setminus\{0\}, v\in A\setminus\{0\}}\frac{\vert (u,v)\vert}{\Vert u\Vert \Vert v\Vert},$$
where $B^\perp =\{u\in\C^n\vert (u,b)=0$ for all $b\in B\}$ and $\Vert u\Vert^2=(u,u)$.\par\noindent
Note that ${\rm dist}(A,B)=0$ means that we have the inclusion $B^\perp\subset A^\perp$, that is $A\subset B$. Note also that $(A,B)\mapsto {\rm dist}(A,B)$ defines a real analytic function on the product of the relevant grassmanians, and that Schwarz's inequality implies ${\rm dist}(A,B)\leq 1$.\par\medskip
In \cite[Lemma 5.2]{Hi}, Hironaka proved that if the pair $(X^0,Y)$ satisfies the Whitney conditions at every point $y$ of $Y$ in a neighborhood of $0$, there exists a positive real number $e$ such that 
$${\rm lim}_{x\to y,x\in X^0}\frac{{\rm dist}(T_{Y,y},T_{X,x})}{{\rm dist}(x,Y)^e}=0.$$
This is a "strict", or analytic,  version of condition $a)$. There is a similar statement for condition $b)$.\par\noindent  In \cite{V}, Verdier defined another incidence condition: the pair $(X^0,Y)$, with $Y\subset X$ nonsingular, satisfies the condition $(w)$ at a point $y\in Y$ if there exist a neighborhood $U$ of $y$ in $X$ and a constant $C>0$ such that for all $x\in X^0\cap U,\ y'\in Y\cap U$, we have:
$${\rm dist}(T_{Y,y'},T_{X,x})\leq C\Vert y'-x\Vert.$$
This is to Hironaka's strict Whitney condition $a)$ as Lipschitz is to H\"older. These metric conditions also make sense in $\R^n$ endowed with the euclidean metric.\par\noindent Condition $(w)$ implies both Whitney conditions (see \cite[Th\'eor\`eme 1.5]{V}). \par Verdier showed that $(w)$ is a stratifying condition, also in the subanalytic case (see \cite[Th\'eor\`eme 2.2]{V}).\par
In \cite[Chap.V, Th\'eor\`eme 1.2]{Te3} (where $(w)$ is called "condition a) stricte avec exposant 1"), it is shown, using the algebraic characterization of Whitney conditions,  that in the complex-analytic case, $(w)$ is in fact equivalent to the Whitney conditions. Brodersen and Trotman showed in \cite{B-T} that this is not true in real algebraic geometry. \par\medskip
In the subanalytic world, and in particular in semialgebraic geometry, Comte and Merle showed in \cite{CM} that one could define local polar varieties, and that one could define real analogues of the local vanishing Euler-Poincar\'e characteristics as well as local real analogues of the multiplicities of polar varieties,  relate them by a real analogue of theorem \ref{formula} and prove that they are constant along the strata of a $(w)$ stratification of a subanalytic set. 
\subsection{Other applications}
According to Ragni Piene in "Polar varieties revisited" (see  \cite{Pi2} and\break www.ricam.oeaw.ac.at/specsem/specsem2013/workshop3/slides/piene.pdf), 
polar varieties have been applied to study:
\begin{itemize}
\item Singularities (L\^e-Teissier, Teissier, Merle, Comte. . . ; see the references to L\^e, Teissier, Gaffney, Kleiman, and  \cite{Co},\cite{CM})
\item The topology of real affine varieties (Giusti, Heinz et al., see \cite{BGHM1}, \cite{BGHM2}, \cite{BGHP},
Safey El Din-Schost; see \cite{S-S})
\item Real solutions of polynomial equations (Giusti, Heinz, et al., see \cite{BGHP})
\item Complexity questions (see \cite{M-R} for the complexity of computations of Whitney stratifications, and \cite{B-L})
\item Foliations (Soares, Corral, see \cite{So}, \cite{Cor})
\item Focal loci and caustics of reflection (Catanese-Trifogli,
Josse-P\`ene; see \cite{CT}, \cite{J-P})
\item Euclidean distance degree : The Euclidean distance degree of a variety is the number of critical points of the squared distance to a generic point outside the variety. (J. Draisma et al., see \cite{DHOST}).
\end{itemize}

\noindent
Arturo Giles Flores,
Departamento de Matem\'aticas y F\'isica,
Centro de Ciencias B\'asicas,
Universidad Aut\'onoma de Aguascalientes.\par\noindent
Avenida Universidad 940, Ciudad Universitaria
C.P. 20131,\par\noindent
Aguascalientes, Aguascalientes, M\'exico.\par\noindent
arturo.giles@cimat.mx
\par\medskip\noindent
Bernard Teissier, Institut math\'ematique de Jussieu-Paris Rive Gauche, \par\noindent 
UP7D - Campus des Grands Moulins,
Boite Courrier 7012. \par\noindent B\^at. Sophie Germain, 8 Place Aur\'elie de Nemours, 75205 PARIS Cedex 13, France.\par\noindent
bernard.teissier@imj-prg.fr

\end{document}